\RequirePackage{fix-cm}
\RequirePackage{silence}
\documentclass[sn-mathphys,Numbered]{sn-jnl}
\usepackage{geometry}
\usepackage{comment}
\usepackage{leftidx}
\usepackage{graphicx}
\usepackage{multirow}
\usepackage{amsmath,amssymb,amsfonts}
\usepackage{amsthm}
\usepackage{mathtools}
\usepackage[title]{appendix}
\usepackage{xcolor}
\usepackage{textcomp}
\usepackage{booktabs}
\usepackage{algorithm}
\usepackage{algorithmicx}
\usepackage{algpseudocode}
\usepackage{listings}
\usepackage{enumerate}
\usepackage{ulem}
\usepackage{manyfoot}

\newtheorem{theorem}{Theorem}
\newtheorem{proposition}{Proposition}
\newtheorem{corollary}{Corollary}
\newtheorem{remark}{Remark}
\newtheorem{lemma}{Lemma}
\newtheorem{example}{Example}

\newtheorem{definition}{Definition}

\makeatletter
\def\thm@space@setup{%
  \thm@preskip=12pt plus 4pt minus 2pt  
  \thm@postskip=12pt plus 4pt minus 2pt 
}
\makeatother

\makeatletter
\renewcommand\paragraph{\@startsection{paragraph}{4}{\z@}%
  {3.25ex \@plus1ex \@minus.2ex}%
  {-1em}%
  {\normalfont\normalsize}}
\makeatother

\usepackage{hyperref}
\hypersetup{
  colorlinks=true,
  linkcolor=blue,
  filecolor=red,
  urlcolor=blue,
  linktoc=page,
 citecolor=magenta,
hypertexnames=false
}



\usepackage{zref-clever}
\let\oldnewtheorem\newtheorem

\RenewDocumentCommand{\newtheorem}{momo}{
  \IfValueTF{#2}{
    \AddToHook{env/#1/begin}{
      \zcsetup{countertype={#2=#1}}}
      \zcRefTypeSetup{#1}{
Name-sg = #3 ,
      }
    \oldnewtheorem{#1}[#2]{#3}
  }{
    \AddToHook{env/#1/begin}{
      \zcsetup{countertype={#1=#1}}}
    \zcRefTypeSetup{#1}{
Name-sg = #3 ,
      }
    \IfValueTF{#4}{
      \oldnewtheorem{#1}{#3}[#4]
    }{
      \oldnewtheorem{#1}{#3}
    }
  }
}

\newcommand{\cref}[1]{\zcref{#1}}
\newcommand{\Cref}[1]{\zcref[S]{#1}} 

\begin{document}

\title[]{\texorpdfstring{$C^0$}{C0}-inextendibility of a class of warped-product black hole spacetimes}

\author{Karim Mosani} 
\affil{Fakultät für Mathematik, Universität Wien, Oskar-Morgenstern-Platz 1, 1090 Wien, Austria}

\email{karim.mosani@univie.ac.at}


\abstract{
We adapt Sbierski’s proof of $C^{0}$-inextendibility of the maximal analytic Schwarzschild spacetime to a broad class of warped-product black hole spacetimes with a static exterior region. These spacetimes are globally hyperbolic, have a codimension-two Riemannian fibre and a radial coordinate $(r)$, which serves as the warping function of the fibre. They admit a spacetime singularity as $r\to0$, characterised by the divergence of the Kretschmann scalar. This class encompasses nonvacuum black hole models and geometries beyond spherical symmetry. Under suitable assumptions, including that the fibre is closed (compact without boundary), connected, homogeneous, and orientable, we establish future $C^{0}$-inextendibility for spacetimes in this class. The result further extends to spacetimes possessing more than one regular black hole horizon.\\


}

\maketitle

\tableofcontents

\section{Introduction}\label{sec:intro}
The analysis of spacetime inextendibility at low regularity is intimately tied to two central questions in general relativity. The first is to understand and characterise gravitational singularities, in particular to determine how strong such singular behaviour can be. The second is connected with the statement of the strong cosmic censorship, according to which spacetimes should generically fail to admit extensions beyond singular boundaries as weak solutions of the Einstein equations. Consider the following definition.

\begin{definition}{}[$C^k$-inextendibility]
Let $k\in \{0\}\cup\mathbb{N}\cup\{\infty\}$. A smooth spacetime $(\mathcal{M},g)$ is called $C^k$-inextendible if there does not exist a spacetime $(\tilde{\mathcal{M}},\tilde{g})$ of the same dimension, equipped with a $C^k$ Lorentzian metric, with the property that there exists an isometric embedding $\iota: \mathcal{M} \hookrightarrow \tilde{\mathcal{M}}$ with $\iota(\mathcal{M}) \subsetneq \tilde{\mathcal{M}}$
\end{definition}

The regularity level in this definition is essential. A spacetime may fail to admit an extension in one regularity class while remaining extendible in a weaker one. The following elementary example illustrates this point.

\begin{example}
Let $(\mathcal{M},g)$ be a smooth spacetime with
$$
\mathcal{M}:=(0,\infty)\times\mathbb{R}^3
\quad\text{and}\quad
g:=-dt^2+a(t)^2\sum_{i=1}^{3}\left(dx^i\right)^2;\quad a(t)=1+t^{\frac{3}{2}} 
$$
and the time orientation induced by $\partial_t$. It is easy to see that $(\mathcal{M},g)$ does not admit a $C^2$-extension through $\{t=0\}$ since the Ricci scalar, given by
$$
\mathcal{R}(t)=6\left(\frac{\ddot a}{a}+\left(\frac{\dot a}{a}\right)^2\right)=\frac{9\left(1+4t^{\frac{3}{2}}\right)}{2\sqrt{t}\left(1+t^{\frac{3}{2}}\right)^2}
$$
diverges as $t\to 0$. However, $(\mathcal{M},g)$ does admit a $C^0$-extension given by $\iota:\mathcal M\hookrightarrow\tilde{\mathcal M}$ with $\tilde{\mathcal M}:=(-\infty,\infty)\times\mathbb{R}^3$, and
\[
\tilde g:=
\begin{cases}
g, & \text{for points in } \mathcal M, \\[0.3em]
-dt^2+\sum_{i=1}^{3}\left(dx^i\right)^2, & \text{for points in } (-\infty,0]\times \mathbb{R}^3.
\end{cases}
\]
\end{example}

Only in 2016, Sbierski’s seminal result showed that the maximal analytic extension of Schwarzschild spacetime—one of the basic black hole models exhibiting a curvature singularity (see \S 6 in~\cite{Wald_1984})—is $C^{0}$-inextendible ~\cite{Sbierski_2018a}. A key structural insight was later provided by Galloway, Ling, and Sbierski \cite{Galloway_2018}, who identified timelike geodesic completeness as an obstruction to extending a spacetime even at $C^{0}$-regularity, thereby enhancing the classical result by Beem-Ehrlich-Easley (see \S 6, proposition 6.16 in \cite{Beem_2017}). Their result is the following.

\begin{theorem}\label{thmobstruction}[Obstruction to $C^0$-extension, \cite{Galloway_2018}]
Every smooth globally hyperbolic Lorentzian manifold that is timelike geodesically complete is $C^{0}$-inextendible.
\end{theorem}

It is worth noting that Minguzzi and Suhr \cite{MinguzziSuhr_2019} subsequently strengthened this result by, among other improvements, removing causality assumptions on the original spacetime. See also the related work of Graf and Ling \cite{Graf_2018}, who earlier dispensed with global hyperbolicity, but at the cost of imposing stronger regularity assumptions, i.e. $C^2$ on the original, and $C^{0,1}_{\textrm{loc}}$ on the extended spacetimes. 

Sbierski later gave a streamlined proof of the $C^0$-inextendibility of the maximal analytic extension of Schwarzschild spacetime by applying the obstruction stated in ~\Cref{thmobstruction} above, in \cite {Sbierski_2018}. The aim of the present paper is to isolate the mechanism underlying this $C^0$-inextendibility argument in the Schwarzschild spacetime and to show that it continues to hold for a broader class of warped-product black hole spacetimes with a static exterior region, under symmetry assumptions weaker than spherical symmetry. Before turning to this analysis, however, we briefly place the research program of investigating the $C^0$-(in)extendibility of spacetimes in a broader context by discussing, in the next two subsections, the two central questions in general relativity mentioned above.

\subsection{Structure of spacetime singularities}\label{sec:1.1}

The singularity theorems by Penrose and Hawking show that the formation of spacetime singularities is not an exceptional occurrence but a generic outcome of Einstein’s equations. More precisely, these theorems demonstrate that, under broad and physically reasonable conditions—such as those describing complete gravitational collapse \cite{Penrose_1965} (see also \cite{Mosani_2022thesis}) or cosmological expansion \cite{Hawking_1966}—spacetimes inevitably exhibit causal geodesic incompleteness. This incompleteness is then identified as the presence of spacetime singularities. See classic monographs such as Hawking and Ellis \cite{Hawking_1973} and Wald~\cite{Wald_1984}, and also the review by Senovilla \cite{Senovilla_1998} for extensive descriptions of the theorems; see also \cite{Sen:22}, which is a more recent critical appraisal of the theorems.
One may interpret these results as only signalling a breakdown in the regularity of the metric tensor of the spacetime. This problem led to the extensions of the classical singularity theorems, assuming lower regularity on the Lorentzian metric--see \cite{Kunzinger_2015, Graf_2017, Schinnerl_2021, Kunzinger_2022}, and more recently \cite{Calisti_2025, kunzinger2026hawkingsingularitytheoremholder}-- for a pedagogical review of low-regularity extensions of the singularity theorems, see \cite{Ste:23} by Steinbauer. 

A further subtlety in the interpretation of geodesic incompleteness is illustrated by examples such as Minkowski spacetime with a point removed. Although such a spacetime is geodesically incomplete, the incompleteness does not correspond to any physical singularity, since all curvature invariants remain regular near the missing point. This highlights the importance of studying geodesic incompleteness together with inextendibility of the spacetime: only when no extension (of suitable regularity) of the spacetime exists can incomplete geodesics be understood as genuine indicators of singular behaviour rather than artefacts of an artificially truncated spacetime. In this sense, the study of inextendibility refines and complements the conclusions of the singularity theorems.

Historically, the analysis of spacetime singularities has largely proceeded through the study of curvature and of the behaviour of Jacobi fields along incomplete timelike geodesics. This reflects the expectation that, for singularities arising in physically relevant situations, such as possible end states of unhindered gravitational collapse \cite{joshi2007gravitational}, geodesic incompleteness should be accompanied by further geometric or physical pathology. Such pathologies may include curvature blow-up, but also more generally tidal distortion effects detected through the behaviour of Jacobi fields (such properties are not implied by the singularity theorems). This has led to a few notions of the ``strength of singularities" beginning with Ellis and Schmidt \cite{Ellis:1977pj} and followed by more precise formulation due to Tipler \cite{Tipler:1977zza}; see also Clarke and Królak \cite{ClarkeKrolak_1985}, Nolan \cite{Nolan:1999tw}, and Ori \cite{Ori:2000}, and the book by Clarke \cite{clarke1993analysis}, for further developments. 

The relation of such notions of strength of singularity to the extendibility of spacetime is, however, subtle, and is an open area of research. In \cite{Tipler:1977zza}, Tipler suggested that the existence of strong curvature singularities (identified by the vanishing of the volume element obtained due to the wedge product of independent Jacobi fields along the congruence of incomplete causal geodesics) is connected with the  $C^{0}$-inextendibility of the spacetime across the singularity. However, this connection was not established rigorously in his work. Nolan \cite{Nolan:1999tw} later emphasised that the presence of a strong singularity (in the sense of Tipler) does not, in principle, exclude the existence of a  $C^0$-extension of the spacetime, although no concrete example was then available. Ori subsequently provided such an example by constructing a spacetime toy model admitting a strong curvature singularity (in the sense of Tipler), but also admitting $C^{0}$ extension (see \S 2 of \cite{Ori:2000}). In particular, this demonstrates that a strong curvature singularity, as defined by Tipler, does not, in general, preclude $C^{0}$-extendibility. The connection between curvature blow-up and low regularity inextendibility (precisely $C^{0,1}_{\textrm{loc}}$-inextendibility) was  however established for the first time only recently by Sbierski in \cite{Sbi:25}.

\subsection{Strong cosmic censorship and weak solutions of Einstein's equations}\label{sec:1.2}

Strong cosmic censorship is one of the central conjectures in mathematical general relativity because it gives a precise formulation to the expectation that the Einstein equations should determine the evolution of spacetime in a genuinely deterministic way. In its original form, associated with Penrose, it asserts that spacetimes should not admit any naked singularities. The conjecture is also formulated through the Cauchy problem for the Einstein equations as follows. 

\[
\parbox{0.93\textwidth}{\centering
``For generic compact or asymptotically flat initial data for the Einstein equations, in vacuum or with reasonable matter fields, its maximal global hyperbolic development should be inextendible as a suitably regular Lorentzian manifold."
}
\]

The maximal globally hyperbolic development is the maximal spacetime evolution of the initial data that is uniquely determined, a priori, by the initial value problem. Any extension beyond this development, as a weak solution of the Einstein equations, is a priori non-unique. Thus, the strong cosmic censorship conjecture may be understood as asserting that general relativity is generically deterministic. To make precise what is meant by “suitably regular” in the formulation above, one looks for a regularity class that accommodates weak solutions--for which the Einstein tensor is well-defined distributionally. It is widely considered that the sufficient regularity class is $g\in C^0$ (where $g$ is the metric tensor)  and $\partial g\in L^2_{\textrm{loc}}$ \cite{Geroch:1986jjl} (see also \cite{Christodoulou2009}, \cite{Chrusciel1991}, \S 1.1 of \cite{Sbierski_2021}). Not very long ago, Dafermos and Luk \cite{Dafermos_2025} proved that for dynamical vacuum black holes which settle down to Kerr black hole spacetimes in the exterior, the Lorentzian metric tensor extends continuously to the Cauchy horizon. If the exterior region of the Kerr family is proven to be dynamically stable, it would follow that $C^0$-inextendibility formulation of the strong cosmic censorship conjecture is false. For further discussions on the strong cosmic censorship conjecture, see the introduction of \cite{Dafermos_2025} (see also the monograph by Joshi \cite{Joshi_1997} for the original formulation of Penrose and its relation to gravitational collapse and the formation of naked singularities; and a more recent review article by Van de Moortel \cite{VanDeMoortel_2025}).

These developments, discussed in \Cref{sec:1.1} and \Cref{sec:1.2}, underscore the importance of understanding the inextendibility of spacetimes at low regularity. They also motivate the search for geometric mechanisms that obstruct continuous extensions, both in cosmological as well as black hole spacetimes. In the cosmological setting, Galloway and Ling proved past $C^{0}$-inextendibility for a subclass of Friedmann–Lemaître–Robertson–Walker (FLRW) spacetimes within the restricted class of spherically symmetric extensions \cite{Galloway_2017}. Their result was subsequently extended by Graf and van den Beld-Serrano to the broader class of axisymmetric extensions \cite{Graf_2025}. Without imposing symmetry assumptions on the extension, Sbierski established the past $C^{0}$-inextendibility of a subclass of FLRW spacetimes without particle horizons and with constant spatial curvature $K=\pm1$ \cite{Sbierski_2023}. Ling later obtained the corresponding result for a subclass of spatially flat FLRW spacetimes, namely in the case $K=0$ \cite{Ling_2024}. Finally, in an anisotropic cosmological setting, Miethke proved the $C^{0}$-inextendibility of Kasner spacetime \cite{Miethke_2024}.

In the black hole setting, Sbierski observed in \cite{Sbierski_2018a} that the singularity structure of the Schwarzschild spacetime is not expected to be generic. In the spirit of the investigations suggested there, we consider here a larger class of black hole spacetimes with a static exterior region, including the Schwarzschild spacetime as a special case. 

\subsection{Outline of the paper}  
The paper is organised as follows. In \Cref{sec2}, we introduce the class of black hole spacetimes under consideration and state the three main results proved in this paper. More precisely, in  \Cref{sec2.1} we recall the basic definitions needed to describe the class of spacetimes under study; in \Cref{sec2.2} we specify the class of black hole spacetimes of interest; and in \Cref{sec2.3} we formulate the main theorems, namely \Cref{maintheorem}, \Cref{eitheror}, and \Cref{maintheorem2}. In \Cref{sec3}, we gather several causality results, some known and some new, that will be used later in the proofs of the main theorems. In particular, this section contains the proof of future one-connectedness for a subclass of the spacetimes described in \Cref{sec2.2}. In \Cref{sec4}, we prove  \Cref{maintheorem}, adapting Sbierski's method from \cite{Sbierski_2018} to our setting; the argument is by contradiction. In \Cref{sec5}, we prove  \Cref{eitheror}, which provides sufficient conditions for the exterior of a general warped-product black hole spacetime, without assuming spherical symmetry, to be timelike geodesically complete. In \Cref{sec6}, we present the proof of  \Cref{maintheorem2}, establishing that the class of warped-product black hole spacetimes considered here is future $C^0$-inextendible. In \Cref{sec7}, we highlight the main technical distinctions between our adapted proof of Sbierski, and Sbierski's original argument in the case of the maximal analytic extension of Schwarzschild spacetime. Finally, in \Cref{sec:examples-topology}, we discuss two examples of black hole spacetimes that are future $C^0$-inextendible, and also are not ruled out by the topological censorship theorem of Galloway \cite{Galloway_2008}. Additionally, we study Birmingham black hole spacetimes in \cite{Birmingham_1999}, and prove their future $C^0$-inextendibility through the central curvature singularity.

\section*{Acknowledgements}

The author would like to thank Clemens Sämann for his involvement in the initial phase of this project and for valuable discussions throughout the development of the paper. The author is also grateful to Leonardo García-Heveling, Eric Ling, Jan Sbierski, and Kharanshu Solanki for helpful discussions and for pointing out typos and other improvements that contributed to the current version of the manuscript.

The author gratefully acknowledges support from the program: Geometry and Convergence in Mathematical General Relativity at the Simons Center for Geometry and Physics, Stony Brook University, where this research project began.

This research was funded in whole or in part by the Austrian Science Fund (FWF) [grants DOI \href{https://www.fwf.ac.at/en/research-radar/10.55776/EFP6}{10.55776/EFP6}]. For open access purposes, the author has applied a CC BY public copyright license to any author-accepted manuscript version arising from this submission.

\section{Setup and Results}\label{sec2}
The spacetimes considered in this paper belong to a class of black hole spacetimes whose exterior region is static. Before introducing this class precisely, we fix the geometric terminology used throughout the paper in the following subsection.

\subsection{Basic definitions}\label{sec2.1}
Here, we recall the notions of Killing horizon, stationarity, black hole horizon, and staticity.
\begin{definition}\label{khdef}[Killing horizon]
   Let $(\mathcal{M},g)$ be a smooth spacetime and let  $X \in \Gamma(T\mathcal{M})$ be a Killing vector field. Consider the set $N_X := \{ p \in \mathcal{M} \mid g_p(X,X)= 0 \}.$ Any connected component of $\mathcal{N}_X$ that is a smooth null hypersurface is called a Killing horizon associated with $X$. The Killing horizon associated with the stationary Killing vector field is called the stationary Killing horizon. 
\end{definition}

To discuss stationary black hole spacetimes, one should, in principle, specify what is meant by the asymptotic region of the spacetime. Informally, a spacetime is called stationary if it admits a Killing vector field which is timelike in the asymptotic region. This formulation is sufficient for the intuition behind the class of examples considered below, but it is not by itself a precise definition unless an appropriate notion of infinity has been specified. One standard way to make this precise is through a conformal completion at infinity. Given a spacetime $(\mathcal M,g)$, the pair $(\mathcal M_1,g_1)$ is called a conformal completion at infinity of $(\mathcal M,g)$ if (Definition 3.1.1, Chruściel \cite{Chrusciel_2020}; see also \cite{Penrose_1963})
 \begin{enumerate}[(i)]
        \item $\mathcal{M}$ is the interior of $\mathcal{M}_1$,
        \item $\exists$ a differentiable function $\Omega:\mathcal{M}_1\to \mathbb{R}$ such that $g_1$, defined as $\Omega^2g$ on $\mathcal{M}$, extends by continuity to the boundary of $\mathcal{M}_1$ (with the extended metric maintaining its signature on the boundary),
       \item $\Omega (q)$
       $\begin{cases}
       >0 &~\forall~q\in \mathcal{M},\quad \textrm{and}\\
       =0 &~\forall ~q\in \partial\mathcal{M}_1~(\neq \emptyset),
       \end{cases}$
       \item $d\Omega$ is nowhere vanishing on $\partial \mathcal{M}_1$.
    \end{enumerate}
For spacetimes admitting such a conformal completion, one can define the conformal boundary at infinity by
$
\mathcal J^\pm:=\partial\mathcal M_1\cap J^\pm(\mathcal M),
$
and the domain of outer communication by
$
I^+(\mathcal J^-)\cap I^-(\mathcal J^+)\cap \mathcal M.
$
For such spacetimes, we have the following definition of asymptotic stationarity. 
\begin{definition}\label{definitionstationary}[Asymptotically stationary spacetimes admitting conformal completion]
Let $(\mathcal M,g)$ be a spacetime admitting a conformal completion at infinity $(\mathcal M_1,g_1)$. We call $(\mathcal M,g)$ asymptotically stationary if it admits a Killing vector field which is timelike in the set $U\cap\mathcal M$, where $U\subset \mathcal M_1$ is a neighbourhood of the conformal boundary $\partial \mathcal{M}_1$. This Killing vector field is called a stationary Killing vector field.
\end{definition}
In situations where no conformal completion is specified, the phrase “timelike in the asymptotic region” should be understood as a working description rather than as a complete coordinate-free definition. This will be sufficient for the class of warped-product black hole spacetimes considered below, where the radial coordinate explicitly identifies the exterior region, and where we do not know a priori if they admit conformal completion at infinity. If the stationary Killing vector field is hypersurface-orthogonal, then the spacetime is called static. Before recalling this notion, we first introduce the definition of a black hole horizon in the stationary setting.
\begin{definition}\label{bhh}[Black hole horizon]
In a given $(d+1)$-dimensional stationary spacetime $(\mathcal{M},g)$, a (cross-section of) black hole horizon is a smooth, closed (compact without a boundary), oriented, $(d-1)$-dimensional embedded submanifold $\mathcal{S} \subset \Sigma,$ where $\Sigma$ is a spacelike hypersurface in $\mathcal{M}$. The union
of the orbits of the stationary Killing vector field through $\mathcal{S}$ forms a connected component of the stationary Killing horizon, called a black hole horizon $\mathcal{H}$. By abuse of terminology, some authors also refer to the cross-section $\mathcal S$ as the black hole horizon.
\end{definition}
The above definition of black hole horizon is also a working definition for the stationary warped-product class considered in this paper, and not a general definition of a black hole horizon, which is usually defined in terms of the boundary of conformal completion of the spacetime (if it exists), and is beyond the scope of our paper.
We now recall the stronger notion of staticity. Unlike stationarity, which only requires the timelikeness of the Killing vector field near infinity, staticity requires the relevant stationary Killing vector field to be timelike throughout the spacetime and to be hypersurface-orthogonal.
\begin{definition}\label{definitionstatic}[Static spacetimes]
An asymptotically stationary spacetime is called static if the stationary Killing vector field $T$ is timelike everywhere, and $T^\perp:=\{X\in T\mathcal{M}: g(X,T)=0\}$ is integrable. In other words, $T$ is everywhere orthogonal to a family of spacelike hypersurfaces.
\end{definition}
Consider the Schwarzschild spacetime $(\mathcal{M}_S,g_S)$ with positive mass $M>0$ in the Schwarzschild coordinates $(t,r,\theta,\phi)$, where $\mathcal{M}_S=\mathbb{R}\times(0,\infty)\times \mathbb{S}^2$, and
\begin{equation*}
    g_S=-\left(1-\frac{2M}{r}\right)dt^2+\frac{dr^2}{\left(1-\frac{2M}{r}\right)}+r^2g_{\mathbb{S}^2}.
\end{equation*}
In the Kruskal-Szekeres extension (see Chap. 13 in \cite{oneill}), the set $\{r=2M\}$ is the black hole horizon. The set $\{r>2M\}$ is the domain of outer communication. The manifold $\mathcal{M}_E:=\mathcal{M}_S~\cap\{r>2M\}$ with the metric tensor $g_E:=g_S$ restricted to $\mathcal{M}_E$ is a static spacetime, while $(\mathcal{M}_S,g_S)$ is asymptotically stationary in the sense of the definitions above. 

 \subsection{The class of black hole spacetimes}\label{sec2.2}
With the abovementioned notions in place, we are now in a position to introduce the class of spacetimes considered in this paper. Consider a $(d+1)$-dimensional ($d\geq 3$) globally hyperbolic, warped-product spacetime $(\mathbb{R}\times(0,\infty)\times\mathcal{F},g)$ with a warped-product metric tensor in Schwarzschild time and radial coordinates $(t,r)$ as
$$
g:=-A(r)\ dt\otimes dt +B(r)\ dr\otimes dr +r^2 h.
$$
Here $h$ is a Riemannian metric tensor on the codimension-two fibre $\mathcal{F}$. For the class of black hole spacetimes considered here, to keep the physical interpretation of the metric coefficients transparent, it is convenient to introduce two functions
$$
\psi(r):=-\frac{1}{2}\ln{\left(A(r)B(r)\right)}, \quad \textrm{and}\quad m(r):=\frac{r}{2}\left(1-\frac{1}{B(r)}\right),
$$
and  rewrite the spacetime as follows.
\begin{equation}\label{sssst}
    \begin{split}
        \mathcal{M} & :=\mathbb{R}\times(0,\infty)\times\mathcal{F}, ~\textrm{and}\\
g &
:= 
-\exp{\left(-2\psi(r)\right)}\left(1-\frac{2m(r)}{r}\right)\,dt\otimes dt
+\left(1-\frac{2m(r)}{r}\right)^{-1}dr\otimes dr
+r^2 h.
    \end{split}
\end{equation}
Here, $\psi$ and $m$ are at least $C^2$ functions of the radial coordinate $r$. If we set the functions $\psi(r)\equiv0$ and $m(r)$ as a positive constant, then $(\mathcal{M},g)$ is a Schwarzschild spacetime with positive mass. One can see that the metric tensor above is not well defined at $r=2m(r)$. This is an issue of the choice of coordinates, as we can see by considering the ingoing Eddington--Finkelstein coordinate $v$, which is defined as
$$
v:=t+\int \frac{\exp{\left(\psi(r)\right)}}{\left(1-\frac{2m(r)}{r}\right)}\,dr.
$$
Then
$$
dt=dv-\frac{\exp{\left(\psi(r)\right)}}{\left(1-\frac{2m(r)}{r}\right)}\,dr,
$$
and the metric tensor $g$ takes the form
\begin{equation}\label{sssstefcoordinate}
    g = -\exp{\left(-2\psi(r)\right)}\left(1-\frac{2m(r)}{r}\right)\,dv^2 + 2\exp{\left(-\psi(r)\right)}\,dv\,dr + r^2h.
\end{equation}
The spacetime $(\mathcal{M},g)$ has a black hole horizon if the function $r-2m(r)$ has a positive real and regular \footnote{
By a regular root, we mean that $1-2m(r_H)/r_H=0$ and $ \frac{d}{dr}\left(1-2m(r)/r\right)\big|_{r=r_H}\neq 0.$ Since the first equation is equivalent to $r_H=2m(r_H)$, this is the same as requiring
$m'(r_H)\neq \frac12.$
This condition ensures that the zero set of $1-2m(r)/r$, or equivalently of $g(\partial_t,\partial_t)$, is non-degenerate at $r=r_H$.} root. In principle, $(\mathcal{M},g)$ could admit more than one black hole horizon.  In the main body of the paper, however, we assume for simplicity that $r-2m(r)$ has a unique positive regular root, denoted by $r_H$. The hypersurface $\{r=r_H\}$ is then the black hole horizon. The case of more than one positive regular root is treated separately in the \Cref{app:multiplebhh}, where we show that the arguments below continue to apply in such scenarios.

The vector field $\partial_t$ is a Killing vector field. 
We consider the case in which there exists exactly one root of $r-2m(r)$. This implies that $\partial_t$ is timelike for $r>r_H$ (hence the class of spacetimes as described in \eqref{sssst} is stationary. Here we assume $m(r)>0$ for all $r$). Assuming that $r_H$ is a regular root of $g(\partial_t,\partial_t)$, the set $\mathcal{H}:=\{r=r_H\}$ is a smooth hypersurface. To check that $\mathcal{H}$ is null, in Eddington-Finkelstein coordinates, $\mathcal H$ is the level set of the function $\Phi=r-r_H.$ Therefore its normal covector is $d\Phi=dr,$ i.e., $dr(X)=X(r)=0~\forall~X\in T_p\mathcal{H}$ (where $p\in\mathcal{H}$). We have that
$
g^{-1}(dr,dr)=\left(1-2m(r)/r\right)$, and hence
$
g^{-1}(dr,dr)\vert_{\mathcal H}=0.
$
Thus, the normal covector $dr$ is null along $\mathcal H$, and therefore $\mathcal H=\{r=r_H\}$ is a null hypersurface. Since $\mathcal H$ is also a connected component of the zero set of $g(\partial_t,\partial_t)$, it is a Killing horizon associated with the stationary Killing field $\partial_t$. If the fibre $\mathcal{F}$ is compact, connected, and oriented, then the cross-sections of $\mathcal H$ are smooth, closed, oriented spacelike submanifolds, and the union of the $\partial_t$-orbits through any such cross-section is the hypersurface $\mathcal H$. Therefore, $\mathcal H$ is a black hole horizon in the sense of \Cref{bhh}.

We say that the spacetime has a curvature singularity as $r\to0$ if $\lim_{r\to0}\mathcal{K}=\infty$ (here $\mathcal{K}$ is the Kretschmann scalar). The sufficient condition for this blow up is $\lim_{r\to 0}m(r)=m_0>0$ (See the proof in  \Cref{appendixproposition} in the \Cref{app:kret2})
If the spacetime is asymptotically flat, then
$\lim_{r\to \infty} m(r)<\infty,$ and $\lim_{r\to \infty} \psi(r)=0.$
These conditions ensure that for large $r$, the metric coefficients $g_{tt}\to -1$  and $g_{rr}\to 1$. Moreover, on $\{r>r_H\}$ the Killing vector field $\partial_t$ is orthogonal to the spacelike hypersurfaces $\{t=\mathrm{constant}\}$.  Hence, this class of spacetimes $(\mathcal M,g)$ restricted to $\{r>r_H\}$ is static.

For the spacetimes of the type described in \eqref{sssst} (see also \eqref{sssstefcoordinate}), we assume that $m_0>0$, which ensures the presence of a curvature singularity as $r\to 0$, manifested by the blow-up of the Kretschmann scalar. On the other hand, we do not impose asymptotic flatness. We will refer to the spacetime in the interior of the black hole horizon as $(\mathcal{M}_{\textrm{int}},g_{\textrm{int}})$, defined precisely as $\mathcal{M}_{\textrm{int}}:=\mathbb{R}\times (0,r_{H})\times \mathcal{F}$, and $g_{\textrm{int}}:=g$ restricted to $\mathcal{M}_{\textrm{int}}$. 

\subsection{Results}\label{sec2.3}

The main theorems proved in this paper are stated below. In the main body of the paper, we assume for simplicity that the function $r-2m(r)$ admits exactly one positive regular root $r_H$, so that the spacetime has a single black hole horizon. This assumption is not essential. The results remain valid when $r-2m(r)$ has several positive regular roots, corresponding to multiple black hole horizons. We discuss this case separately in \Cref{app:multiplebhh}, where we establish that \Cref{eitheror} continues to hold in the multiple-horizon setting. Consequently, the conclusions of \Cref{maintheorem} and \Cref{maintheorem2} also continue to hold in such a setting.

\begin{theorem}\label{maintheorem}[No \texorpdfstring{$C^0$}{C0}-extension through the central curvature singularity]
    Let $(\mathcal{M},g)$ be a globally hyperbolic warped-product black hole spacetime expressed in the Schwarzschild coordinates as described in \eqref{sssst}. 
    Let the functions $m(r)$ and $\psi(r)$ in the metric coefficients satisfy 
    \begin{equation}\label{conditionthm2}
            m_0=\lim_{r\to 0}m(r)>0, \quad \textrm{and} \quad \sup_{r\in(0,r_H]}\vert\psi(r)\vert<\infty.
    \end{equation}
If the fibre $(\mathcal{F},h)$ is closed, connected, homogeneous \footnote{Recall that a Riemannian manifold $(\mathcal{F},h)$ is called homogeneous if its isometry group $\mathrm{Isom}(\mathcal{F},h)$ acts transitively on $\mathcal{F}$. In other words, for every $p,q\in \mathcal{F}$, there exists $\phi\in \textrm{Isom}\left(\mathcal F,h\right)$ such that $\phi\cdot p=q$.}, and orientable \footnote{We assume that the fibre $\mathcal F$ is orientable and fix an orientation on it. Since each cross-section of the horizon is diffeomorphic to $\mathcal F$, this induces an orientation on the horizon cross-section.}, then, there does not exist a $C^0$-extension $\iota: \mathcal{M}\hookrightarrow \tilde{ \mathcal{M}}$ with the following property: there exists an affinely parametrised, future-directed and future-inextendible timelike geodesic $\rho:[0,b)\to \mathcal{M}$, where $b<\infty$, with the $\lim_{s\to b} r\circ \rho (s)=0$, and such that its future endpoint is contained in $\tilde{\mathcal{M}}$, i.e.,  $\lim_{s\to b}\left(\iota\circ \rho\right)(s)\in \partial^{+} \iota\left(\mathcal{M}\right)\subset\tilde{\mathcal{M}}$. 
\end{theorem}
The condition $m_0>0$ ensures that the level-sets of $r$ are spacelike inside the horizon and close to $r=0$, thereby ensuring that the central singularity (identified by the blow-up of the Kretschmann scalar) is spacelike (just like the case of Schwarzschild spacetime). \Cref{maintheorem} shows that, for the class of black hole spacetimes satisfying the assumptions stated above, the spacetime does not admit a future $C^0$-extension through the central curvature singularity as $r\to 0$. 

In the case of maximal analytic extension of Schwarzschild spacetime denoted by $(\mathcal{M}_{\textrm{max}},g_{\textrm{max}})$, it is already known that for a future-directed, future-inextendible timelike geodesic $\gamma:[-1,0)\to \mathcal{M}_{\textrm{max}}$, either $r\circ\gamma\to 0$ or $\gamma$ is future complete (Proposition 36 in Chap 13 of \cite{oneill}). So for $\gamma$ such that $\textrm{Im}(\gamma)\subset \mathcal{M}_{\textrm{ext}}$, $\gamma$ is future complete. This is then used together with \Cref{thmobstruction} to exclude the $C^0$-extension through the exterior of the spacetime. The analogue of Proposition 36 in Chap 13 of \cite{oneill} for a warped product black hole spacetime as described in \Cref{sec2.2} is the following.

\begin{theorem}\label{eitheror}[The timelike geodesic dichotomy]
    Let $(\mathcal{M},g)$ be a globally hyperbolic warped-product black hole spacetime as described in \eqref{sssst} with closed \footnote{We require compactness and completeness of  $(\mathcal{F},h)$. Completeness is implied by closedness: compactness with empty boundary.} and orientable Riemannian fibre $(\mathcal{F},h)$. Let the functions $m(r)$ and $\psi(r)$ in the metric coefficients satisfy 
    \begin{equation}\label{eitherorconstraints}
    \begin{split}
       &  \lim_{r\to0}m(r)=m_0>0, \quad \sup_{r>r_H}\exp{\left(-2\psi(r)\right)\left(1-\frac{2m(r)}{r}\right)}<\infty, 
       \\  & \sup_{r>r_H}\vert\psi(r)\vert<\infty, \quad \psi'(r)\leq \min\left\{ 0, \frac{2}{r} \left(\frac{m(r)-rm'(r)}{r-2m(r)}\right)\right\}~\forall~r>r_H.
    \end{split}
    \end{equation}
    Consider an affinely parametrised, future-directed, and future-inextendible timelike geodesic $\rho:[0,b)\to \mathcal{M}$ with $b\in(0,\infty]$. Then we either have $r\circ\rho(s):=\rho_r(s)\to 0$ as $s\to b$ or $\rho$ is future complete, i.e., $b=\infty$.
\end{theorem}

Finally, combining \Cref{maintheorem}, \Cref{eitheror}, and the result by Galloway-Ling-Sbierski \cite{Sbierski_2018}, the following theorem can be obtained. 

\begin{theorem}\label{maintheorem2}[Main $C^0$-inextendibility result]
      Let $(\mathcal{M},g)$ be a globally hyperbolic warped-product black hole spacetime as described in \eqref{sssst} with the Riemannian fibre $(\mathcal{F},h)$ being closed, connected, homogeneous, and orientable.  Let the functions $m(r)$ and $\psi(r)$ in the metric coefficients satisfy 
   \begin{equation}\label{conditionfinaltheorem}
    \begin{split}
&\lim_{r\to0}m(r)=m_0>0, \\  
&\sup_{r>r_H}\exp{\left(-2\psi(r)\right)\left(1-\frac{2m(r)}{r}\right)}<\infty,\\ &\sup_{r\in(0,\infty)}\vert\psi(r)\vert<\infty, \textrm{and}\\
&\psi'(r)\leq \min\left\{ 0, \frac{2}{r} \left(\frac{m(r)-rm'(r)}{r-2m(r)}\right)\right\}~\forall~r>r_H.
    \end{split}
    \end{equation}
Then, $(\mathcal{M},g)$ is future $C^0$-inextendible. 
\end{theorem}
The first and third conditions in \eqref{conditionfinaltheorem} ensures that \Cref{maintheorem} applies. The remaining conditions in \eqref{conditionfinaltheorem} ensure that \Cref{eitheror} applies as well. Consequently, \Cref{maintheorem2} follows from \Cref{maintheorem} and \Cref{eitheror} (along with the main result in \cite{Galloway_2018}). This theorem generalises Sbierski's result in \cite{Sbierski_2018} by adapting his methodology to a bigger class of black hole spacetimes. Its proof, in particular the proofs of \Cref{maintheorem} and \Cref{eitheror}, is based on the arguments of \cite{Sbierski_2018a,Sbierski_2018,Sbierski_2021}, adapted here to the present warped-product setting. 

It is not essential for the exterior region of $(\mathcal M,g)$ to be static. Rather, what is crucial is that the interior admits the additional Killing symmetry generated by $\partial_t$ \footnote{private communication with Jan Sbierski.}. Indeed, consider \Cref{gls2} in the next section. If a $C^0$-extension exists, then there is a future-directed timelike geodesic in $\mathcal M$ with an endpoint on the future boundary of the extension. Let this geodesic escape through the exterior region, and let its image be entirely contained in the exterior and not intersect the black hole horizon. Then, \Cref{eitherorprop1} implies that it is future complete (which then acts as an obstruction to future $C^0$-extension from \Cref{thmobstruction}). Moreover, the proof of \Cref{eitherorprop1} does not rely on the existence of a stationary Killing vector field $\partial_t$.  Thus, the argument excluding a $C^0$-extension through the exterior does not use the existence of a stationary Killing vector field: staticity of the exterior plays no role in obstructing such an extension.

By contrast, the proof excluding a $C^0$-extension through the interior, and in particular through the central singularity $r=0$, does rely on the Killing symmetry generated by $\partial_t$ in the interior.

\section{Technical Tools}\label{sec3}
In this section, we collect the technical tools needed for the proof of ~\Cref{maintheorem}, \Cref{eitheror}, and \Cref{maintheorem2}. Some of these are known results from the literature, while others are established here for the class of warped-product black hole spacetimes under consideration.
\begin{definition}{}[Future/past boundary set]
    Let $(\mathcal{M},g)$ be a spacetime, with $g\in C^0$, admitting a $C^0$-extension $\iota:\mathcal{M}\hookrightarrow\tilde{\mathcal{M}}$. The future boundary set, denoted by $\partial^+\iota(\mathcal{M})$, is defined as the set of all points $\tilde p\in\tilde M$ with the property that there exists a timelike curve $\tilde \gamma : [-1,0]\to \tilde M$ such that 
    \begin{enumerate}[(i)]
        \item $\textrm{Im}\left(\tilde \gamma\vert_{[-1,0)}\right)\subseteq \iota(\mathcal{M})$,
        \item $\tilde \gamma(0)=\tilde{p}\in \partial\iota (\mathcal{M})$, and
        \item $\iota^{-1}\circ \tilde \gamma \vert_{[-1,0)}$ is future-directed in $\mathcal{M}$.
    \end{enumerate}
    The past boundary set $\partial ^{-}\iota (\mathcal{M})$ is defined analogously.
\end{definition}
 We have from Lemma 2.17 of \cite{Sbierski_2018a} (see also Proposition 2.3 of \cite{Galloway_2016}) that if $(\mathcal{\mathcal{M}},g)$ admits a $C^0$-extension, then $\partial^{-}\iota(\mathcal{M})\cup\partial^{+}\iota(\mathcal{M})\neq \emptyset$. Additionally, $\partial^{+}\iota(\mathcal{M})$ and $\partial^{-}\iota(\mathcal{M})$ may not necessarily be disjoint. Moreover, a priori, the boundary set, or any subset of it, may not be achronal in the extension. However, if the spacetime is globally hyperbolic, then we have the following result by Sbierski (Proposition 1 in \cite{Sbierski_2018} and Lemma 2.4 in \cite{Sbierski_2018a}). 
 
\begin{proposition}\label{minkowskianchartproposition}[The existence of Minkowskian chart]
      Let $(\tilde{\mathcal{M}},\tilde{g})$ be a $C^0$ extension of a $(d+1)$-dimensional ($d\geq 3$) globally hyperbolic spacetime $(\mathcal{M},g)$. Let $\tilde p\in \partial^+{\iota}(\mathcal{M})$. For every $\delta>0$, there exists a chart $\tilde{\varphi}:\tilde{U} \to \mathbb{R}_{\varepsilon_0,\varepsilon_1}\coloneqq(-\varepsilon_0,\varepsilon_0)\times(-\varepsilon_1,\varepsilon_1)^d$, where $\varepsilon_0,\varepsilon_1 >0$, having the following properties:
    \begin{enumerate}[(i)]
        \item $\tilde p\in\tilde{U}$ and $\tilde{\varphi}(\tilde p)=(0,\cdots,0)$,
        \item $|\tilde{g}_{\mu\nu} -\eta_{\mu\nu}|<\delta$, where $\eta_{\mu\nu} = \operatorname{diag}(-1,1,\cdots,1)$ is the metric tensor of the Minkowski spacetime,
        \item There exists a Lipschitz continuous function $f:(-\varepsilon_1,\varepsilon_1)^d\to (-\varepsilon_0,\varepsilon_0)$ such that:
        \begin{equation*}
            f_{<}:= \{(x_0,\underline{x})\in \mathbb{R}_{\varepsilon_0,\varepsilon_1} | x_0 < f(\underline{x})\} \subseteq \tilde{\varphi}(\iota(\mathcal{M})\cap \tilde{U}),
        \end{equation*}
        and
        \begin{equation*}
            \mathrm{graph}f := \{(x_0,\underline{x})\in \mathbb{R}_{\varepsilon_0,\varepsilon_1} | x_0 = f(\underline{x})\} \subseteq \tilde{\varphi}(\partial^{+}_{\iota}(\mathcal{M})\cap\tilde{U}).
        \end{equation*}
        Moreover $\operatorname{graph}f$ is achronal in $\mathbb{R}_{\varepsilon_0,\varepsilon_1}$. 
    \end{enumerate} 
\end{proposition}
~\Cref{minkowskianchartproposition} provides a local coordinate description of the extension near a future boundary point in which the metric tensor is ``near" the Minkowski metric and the future boundary is represented by an achronal Lipschitz graph. The achronality of this graph is a crucial feature in Sbierski's method, since it prevents past-directed timelike curves from crossing from below the graph to above it. See \Cref{fig:minkchart} for a pictorial representation of the Minkowskian chart and the graph of $f$. An important remark that follows from \Cref{minkowskianchartproposition}, and that is used in Sbierski's method of disproving $C^0$-extensions is the following.
\begin{remark}\label{minkremark}
    The achronality of the graph of $f$ implies the following for a past-directed timelike curve $\sigma:[-1,0)\to \tilde{\mathcal{M}}$. If $\sigma(-1)\in  f_<$, then $\sigma([-1,0))\subset  f_<$. 
\end{remark}

The next proposition is a key ingredient in the proof of \Cref{thmobstruction}. In turn, this theorem was used by Sbierski in \cite{Sbierski_2018} to simplify the first proof of the $C^0$-inextendibility of the maximal analytic extension of Schwarzschild spacetime presented in \cite{Sbierski_2018a}.  The following \Cref{gls2} was also used by Sbierski to refine \Cref{minkowskianchartproposition}, which will be useful again in the proof of \cite{Sbierski_2018}, as we shall see next.

\begin{proposition}\label{gls2}[Galloway-Ling-Sbierski, {\cite[Theorem 2]{Galloway_2018}}]
    Let $(\mathcal{M},g)$ be smooth (at least $C^2$) globally hyperbolic spacetime admitting a $C^0$-extension $\iota:\mathcal{M}\hookrightarrow\tilde {\mathcal{M}}$. Assume that $\partial^+\iota\left(\mathcal{M}\right)\neq \emptyset$. Then, there exists a future-directed timelike geodesic $\tau$ such that $\textrm{Im}(\tau)\subseteq \mathcal{M}$ and such that it has an endpoint in $\partial \iota\left(\mathcal{M}\right)$. 
\end{proposition}

    In the case of $C^2$-extension, the analogue of the \Cref{gls2} is straightforward to obtain as follows: Pick a point $\tilde p\in \partial^+\iota (\mathcal{M})\subset \tilde{\mathcal{M}}$. Now since $\tilde g\in C^2$, the standard ordinary differential equations theory dictates that there exists a timelike geodesic $\tilde \gamma:[0,a]\to \tilde{\mathcal{M}}$ ($a\in \mathbb{R}_{>0}$) such that $\tilde \gamma \vert_{[0,a)}\subset \iota(\mathcal{M})$---and hence $\textrm{Im}\left(\gamma\right)\subseteq \mathcal{M}$, and $\tilde \gamma (a)=\tilde p\in \tilde{\mathcal{M}}\backslash {\iota\left(\mathcal{M}\right)}$. Here $\gamma:=\iota^{-1}\left(\tilde \gamma  \vert_{[0,a)}\right)$.

Putting together the \Cref{minkowskianchartproposition} and \Cref{gls2}, we obtain the following refined statement on the existence of the Minkowskian chart, in a directly applicable form.
\begin{proposition}\label{minkowskianchartproposition2}[Updated \Cref{minkowskianchartproposition}]
    Consider a globally hyperbolic spacetime $(\mathcal{M},g)$ (with $g$ at least $C^2$) having a $C^0$-extension. Assume without loss of generality that $\partial ^+\iota(\mathcal{M})\neq \emptyset$, and let $\tilde p\in \partial^+\iota(\mathcal{M})$. Let $\tilde \varphi: \tilde U \to \mathbb{R}_{\varepsilon_0,\varepsilon_1}$ as in \Cref{minkowskianchartproposition}. Then, there exists a future-directed timelike geodesic $\tau: [-1,0)\to \mathcal{M}$ that is future-inextendible in $\mathcal{M}$, and such that
    \begin{enumerate}[(i)]
        \item $\tilde \tau:=\tilde \varphi\circ \iota \circ \tau:[-1,0)\to \mathbb{R}_{\varepsilon_0,\varepsilon_1}$ maps into $f_{<}$,
        \item $\tilde \tau$ has an endpoint in $\textrm{graph} ~f$, and 
        \item The chart $\tilde{\varphi}$ can be recentered  such that $\lim_{s\to 0}\tilde{\tau}(s)=(0,...,0)$.
    \end{enumerate}
\end{proposition}
\begin{figure}
    \centering
    \includegraphics[width=0.5\linewidth]{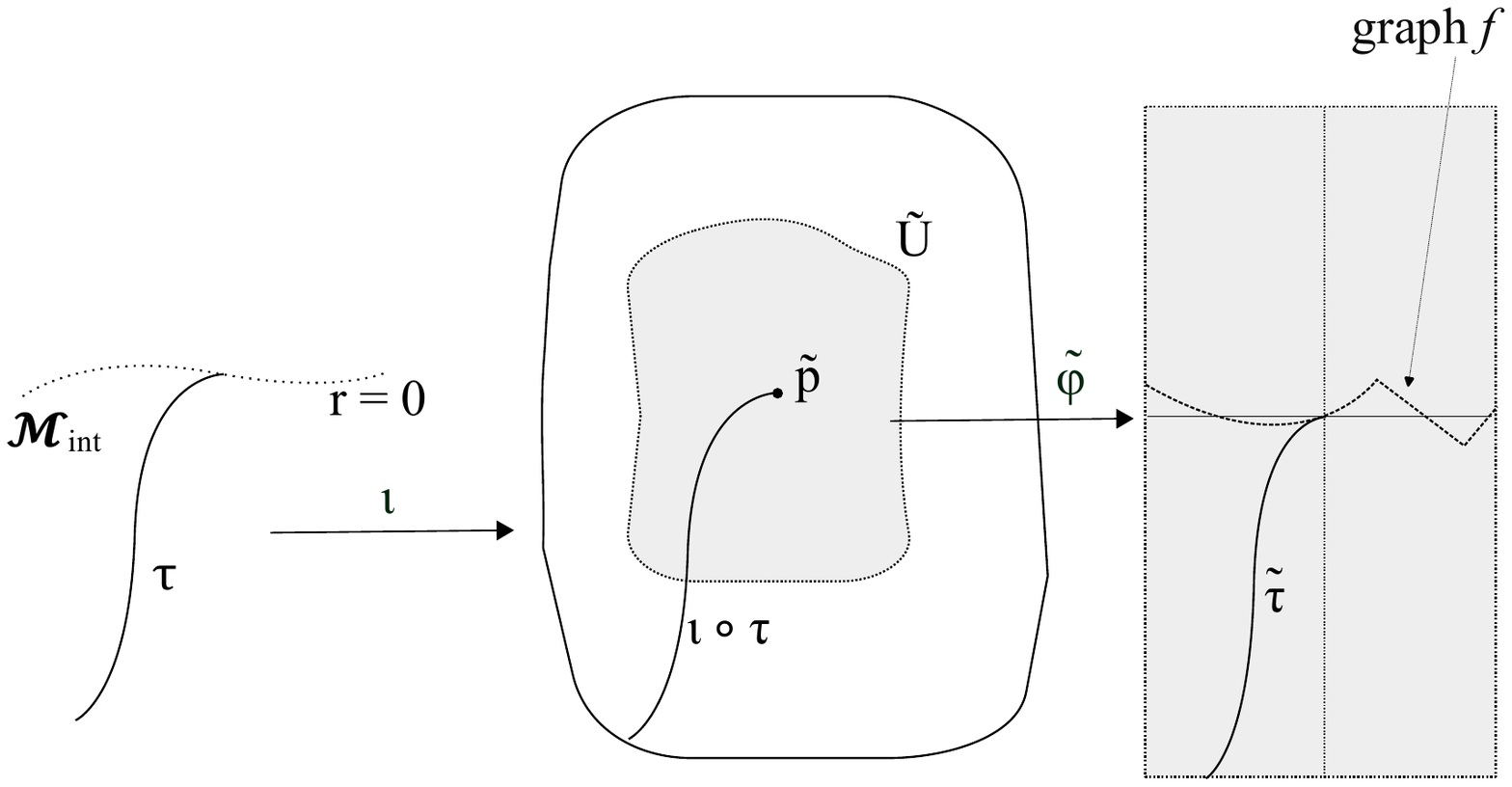}
    \caption{For a spacetime $(\mathcal{M},g)$ admitting a $C^0$-extension $(\tilde{\mathcal{M}},\tilde{g})$ (for e.g., through $r=0$), where $\tilde{g}\in C^0$, there exists a neighbourhood  in $\tilde{\mathcal{M}}$, of the boundary point $\tilde{p}$ of $\mathcal{M}$, that admits a Minkowskian chart as discussed in  \Cref{minkowskianchartproposition2}. A priori, $\tilde \tau$ can intersect with the region above the graph of  the Lipschitz continuous function $f$.} 
    \label{fig:minkchart}
\end{figure}

The next lemma will be used to prove \Cref{causaldiamonds same}, which will subsequently be an important ingredient in the proof of \Cref{maintheorem}. 

\begin{lemma}\label{imageofhomotopy}[Lemma 2.12 in \cite{Sbierski_2021}]
   Let $(\mathcal{M},g)$ be a time-oriented Lorentzian manifold with $g \in C^0$, and let
$
\iota : \mathcal{M} \hookrightarrow \tilde{\mathcal{M}}
$
be a $C^0$-extension of $\mathcal{M}$. Let
$
\gamma : [0,1] \to \mathcal{M}
$
be a future-directed timelike curve, and let $\tilde{U} \subseteq \tilde{\mathcal{M}}$ be an open set such that
$
\tilde{\gamma} := \iota \circ \gamma
$
takes values in $\tilde{U}$ and
$
J^+(\tilde{\gamma}(0),\widetilde{U}) \cap J^-(\tilde{\gamma}(1),\tilde{U})
\Subset \tilde{U}.
$
Suppose
$
\Gamma : [0,1]\times[0,1] \to \mathcal{M}
$
is a causal homotopy with fixed endpoints starting from $\gamma$, i.e., for every $r \in [0,1]$, the curve
    $
    s \mapsto \Gamma(s,r)
    $
    is future-directed causal, with
    $
    \Gamma(0,r)=\gamma(0)$,  $\Gamma(1,r)=\gamma(1),
    $ and for every $s \in [0,1]$,
    $
    \Gamma(s,0)=\gamma(s).
    $
Then the image of $\iota \circ \Gamma$ is contained entirely in $\tilde{U}$.
\end{lemma}

The above statement is not obvious. Indeed, if $\sigma \in \mathcal{M}$ is a causal curve from $\gamma(0)$ to $\gamma(1)$, then by definition
$
\iota \circ \sigma \subset J^+(\tilde{\gamma}(0), \tilde{\mathcal{M}}) \cap J^-(\tilde{\gamma}(1), \tilde{\mathcal{M}}),
$
but, in general $\iota \circ \sigma$ does not need to lie in
$
J^+(\tilde{\gamma}(0), \tilde{U}) \cap J^-(\tilde{\gamma}(1), \tilde{U})
$
even if it is compactly contained in $\tilde{U}$.
However, the proof is not difficult and follows from a continuity argument. The lemma remains valid even if one replaces causal homotopy by timelike homotopy, and
$
I^+(\tilde{\gamma}(0),\tilde{U}) \cap I^-(\tilde{\gamma}(1),\tilde{U}) \Subset \tilde{U}
$
in place of
$
J^+(\tilde{\gamma}(0),\tilde{U}) \cap J^-(\tilde{\gamma}(1),\tilde{U}) \Subset \tilde{U}.
$

We now show that for a subclass of general black hole spacetimes discussed in  \Cref{sec2.2}, inside the horizon and sufficiently close to $r=0$, for any future-directed timelike curve $\sigma$ parametrised by $s:=-r\in[r_1,0)$, the components $\sigma_t(s)$ and $\sigma_{y}(s)$ converge for $s\nearrow0$. The following lemma is an adaptation of Lemma 1 in \cite{Sbierski_2018}, in the context of the spacetimes described in \eqref{sssst}.

\begin{lemma}{}[Coordinate convergence along future-directed future-inextendible timelike curves parametrised by radial coordinate]\label{lemma1}
   Let $(\mathcal{M},g)$  be a warped-product black hole spacetime with the warping function $r$ (the radial coordinate) as described by \eqref{sssst} in \Cref{sec2.2}. Let $\lim_{r\to 0}m(r)=m_0>0$, and $\psi(r)$ be bounded from above and below $\forall ~r\in (0,r_0]$. Then, for a given $\epsilon>0$, there exists $\tilde{r}_0\in(0,r_0)$ such that for any future-directed timelike curve 
 $$
\sigma:(-r_0,0)\to \mathcal{M}_{\textrm{int}}~:~ \sigma(s)=(\sigma_t(s),-s,\sigma_{y}(s)),
 $$
 we have
    $$
    d_{h}(\sigma_{y}(s),\sigma_{y}(s'))<\epsilon ~\textrm{and} ~\vert\sigma_t(s)-\sigma_t(s') \vert<\epsilon ~~\forall ~-\tilde{r}_0<s,s'<0.
    $$ 
\end{lemma}
\begin{proof}
Consider the timelike curve $\sigma$ parametrised by $s=-r$ as mentioned above. We have for all $s\in(-r_0,0)$,
\begin{align*}
   0> &-\exp\left(-2\psi(-s)\right)\left(1-\frac{2m(-s)}{(-s)}\right)\left(\dot \sigma_t(s)\right)^2+\left(1-\frac{2m(-s)}{(-s)}\right)^{-1}+s^2 h\left(\dot \sigma_y(s),\dot \sigma_y(s)\right),\\
    \left(\frac{s+2m(-s)}{(-s)}\right)^{-1}>& -\exp\left(-2\psi(-s)\right)\left(1-\frac{2m(-s)}{(-s)}\right)\left(\dot \sigma_t(s)\right)^2+s^2h\left(\dot \sigma_y(s),\dot \sigma_y(s)\right), \\
\left(\frac{s+2m(-s)}{(-s)}\right)^{-1}> &      -\exp\left(-2\psi(-s)\right)\left(1-\frac{2m(-s)}{(-s)}\right)\left(\dot \sigma_t(s)\right)^2. 
\end{align*}
Here we used the fact that $2m(-s)>-s$ since the curve is mapped to $\mathcal{M}_{\textrm{int}}$. Additionally, we used the fact that the fibre metric tensor $h$ is Riemannian.  
We know that the term $\left(2m(-s)-(-s)\right)$ is eventually bounded below by `say' $\bar\kappa\in(0,\infty)$ (we use the assumption that $\lim_{r\to0}m(r)=m_0>0$).
It then follows that
\begin{equation}\label{tbound}
\frac{(-s)\exp{\left(\Psi_0\right)}}{\bar{\kappa}} \geq\frac{(-s)\exp{\left(\psi(-s)\right)}}{2m(-s)-(-s)}>   \vert\dot \sigma_t(s)\vert.
\end{equation}
Here, $\Psi_0$ is the upper bound of $\psi$. Similarly, we have
\begin{equation}\label{ybound}
\frac{1}{\sqrt{(-s)}\sqrt{\bar{\kappa}}}>\frac{1}{\sqrt{(-s)}\sqrt{2m(-s)-(-s)}}>\vert\vert\dot \sigma_y(s)\vert\vert_{h}.  
\end{equation}

 The upper bounds of $\vert\dot \sigma_t(s)\vert$ and $\vert\vert\dot \sigma_y(s)\vert\vert_{h}$ above are integrable on $(-r_0,0)$, and the lemma then follows immediately from integration.
\end{proof}
\begin{corollary}\label{corollary1}
    With the setup $(\mathcal{M},g)$ of \Cref{lemma1}, for any given $\epsilon>0$, there exists $\tilde r_0 \in(0,r_0)$ such that for any $p:=(p_t,p_r,p_y)\in I^+\left(\sigma\left(-\tilde r_0\right),\mathcal{M}_{\textrm{int}}\right)$, we have 
    $$
    d_h\left(\sigma_y\left(-\tilde r_0\right), p_y\right)<\epsilon, ~\textrm{and} ~\vert \sigma_t(-\tilde r_0)- p_t\vert<\epsilon.
    $$
\end{corollary}
Next, we show that the subset $\{r<r_0\}$ (for some $0<r_0<r_H$) of the interior of the subclass of spacetimes in \eqref{sssst} containing a black hole horizon is future one-connected.  
For completeness, we recall the relevant definition (see also Definition 2.13 in \cite{Sbierski_2018a}) before proceeding.
\begin{definition}{}[Future one-connectedness]
    A time-oriented Lorentzian manifold is called future one-connected if and only if any two future-directed timelike curves with the same endpoints are homotopic through future-directed timelike curves, with endpoints fixed throughout the homotopy.
\end{definition}
    \begin{proposition}\label{prop:foc}
  Let  $(\mathcal{M},g)$ be a warped-product black hole spacetime, with the warping function $r$, as describd by \eqref{sssst} in  \Cref{sec2.2}. Let $\lim_{r\to 0}m(r)=m_0>0$, and $\sup_{r\in(0,\infty)}\vert\psi(r)\vert<\infty$. If the fibre $(\mathcal{F},h)$ has strictly positive injectivity radius, i.e. $inj(\mathcal{F},h)>0$, then there exists $0<\tilde r<r_H$ such that $\mathcal{M}_{\tilde r}:=\mathcal{M}_{\textrm{int}}\cap \{r<\tilde r\}$ is future one-connected.
\end{proposition}

\begin{proof}
    Let us consider a timelike separated pair of points $(t_a,r_a, y_a)$ and $(t_b,r_b, y_b)\in \mathcal{M}_{\tilde r}$ such that $(t_a,r_a, y_a)\ll (t_b,r_b, y_b)$.  Consider two past directed timelike curves $\gamma_i: [r_b,r_a]\to \mathcal{M}_{\tilde r}$ ($i\in\{1,2\}$) given by $\gamma_i(r) =((\gamma_i)_t (r),r,(\gamma_i)_y (r))$ with $\gamma_i(r_a) = (t_a,r_a,y_a)$ and $\gamma_i(r_b)=(t_b,r_b,y_b)$. In the following, we will show that there exists a timelike homotopy with fixed endpoints $(t_a,r_a, y_a)$ and $(t_b,r_b, y_b)$ between $\gamma_1$ and $\gamma_2$.

    Choose $\varepsilon\in(0,\textrm{inj}(\mathcal F,h))$. By \Cref{lemma1}, after decreasing $\tilde r>0$ if necessary, one can obtain
$$
d_h\big((\gamma_i)_y(r),y_b\big)<\varepsilon<\operatorname{inj}(F,h)
\qquad\text{for all }r\in[r_b,r_a].
$$
Thus $(\gamma_i)_y$ lies in the normal ball centred at $y_b$, and $\exp_{y_b}^{-1}$ is well-defined on this ball. Let 
    $
    \operatorname{exp}_{y_b}: T_{y_b} \mathcal{F}\supseteq B_{\textrm{inj}\left(\mathcal{F},h\right)}\left(0\right) \to \mathcal{F}
    $ 
    be the exponential map with base point $y_b$. We define two timelike homotopies $\Gamma_i :[r_b,r_a]\times[r_b,r_a] \to \mathcal{M}_{\tilde r}$ $(i\in\{1,2\})$: (i) between $\gamma_1$ and a timelike curve $\sigma_1$, which has the same $t$ and $r$ coordinates as that of $\gamma_1$, but whose projection in $\mathcal{F}$ is a geodesic in $\mathcal{F}$, and (ii) between $\gamma_2$ and a timelike curve  $\sigma_2$, which has the same $t$ and $r$ coordinates as that of $\gamma_2$, but whose projection in $\mathcal{F}$ is a geodesic in $\mathcal{F}$ (see \Cref{fig:foc}). We define $\Gamma_i$ as follows:
\begin{equation}\label{homotopy-1}
    \Gamma_i (\alpha,r)\coloneqq \begin{cases}
((\gamma_i)_t(r),r,\operatorname{exp}_{y_b}[f_i(\alpha,r)\operatorname{exp}^{-1}_{y_b}((\gamma_i)_{y}(\alpha))] & \text{if } r_b \leq r\leq \alpha,\\
\gamma_i (r)  & \text{if } \alpha \leq r \leq r_a
\end{cases}
\end{equation}
where,
\begin{equation}\label{homotopy-weight}
    f_i(\alpha,r)\coloneqq \frac{\displaystyle\int^r_{r_b} \| {\dot{(\gamma_i)}_y}(r')\|_{h} \ dr'}{\displaystyle\int^\alpha_{r_b} \| {\dot{(\gamma_i)}_y}(r')\|_{h} \ dr'}.
\end{equation}
$f_i(\alpha,r)\in [0,1]~\forall~r\in[r_b,\alpha]$, is monotone increasing, takes value $0$ for $r=r_b$, and takes value $1$ for $r=\alpha$. 
We then ensure that each such curve $\Gamma_i(\alpha,.)$ is timelike for all $r\in[r_b,r_a]$ by computing
\begin{equation}
\begin{aligned}
g(\partial_r \Gamma_i(\alpha,r),\partial_r \Gamma_i(\alpha,r)) &= -\exp\left(-2\psi(r)\right)\left(1-\frac{2m(r)}{r}\right)(\dot{(\gamma_i)}_t(r))^2 + \left(1-\frac{2m(r)}{r}\right)^{-1}\\
     & \ \ \  +r^2 \vert\partial_rf_i\left(\alpha,r\right)\vert^{2} \|\exp^{-1}_{y_b}\left(\left(\gamma_i\right)_{y}\left(\alpha\right)\right)\|^2_{h}\\
    &= -\exp\left(-2\psi(r)\right)\left(1-\frac{2m(r)}{r}\right)(\dot{(\gamma_i)}_t(r))^2 + \left(1-\frac{2m(r)}{r}\right)^{-1}\\
    & \ \ \  +r^2 \|\dot{(\gamma_i)}_y (r) \|^2_{h} \frac{d^2_{h}((\gamma_i)_y(\alpha)),(\gamma_i)_y(r_b))}{\left(\displaystyle \int_{r_b}^{\alpha} \|\dot{(\gamma_i)}_y (r') \|_{h}\ dr'\right)^2}\\
    &\leq g(\dot\gamma_i(r),\dot\gamma_i(r)) <0.
\end{aligned} 
\end{equation}
Here we have used the inequality
$$
\left(\displaystyle \int_{r_b}^{\alpha} \|\dot{(\gamma_i)}_y (r') \|_{h}\ dr'\right)^2\geq d^2_{h}((\gamma_i)_y(\alpha)),(\gamma_i)_y(r_b)).
$$
Moreover, for $\alpha \leq r \leq r_a$, $\Gamma_i$ is simply $\gamma_i$, and hence is timelike by definition. Therefore, $\Gamma_i$ is a timelike homotopy with fixed endpoints between $\gamma_i$ and a timelike curve $\sigma_i$ whose projection on $(\mathcal{F},h)$ lies on the geodesic arc connecting $y_a$ and $y_b$. We construct local coordinates $(s^1,..,s^{d-1})$ on $(\mathcal{F},h)$ such that the projections of $\sigma_i$ \footnote{The images of these projections overlap in $\mathcal{F}$.} ($i\in\{1,2\}$) on $\mathcal{F}$ between $y_a$ and $y_b\in \mathcal{F}$ is parameterised by $s^1 \in [0,1]$. Hence $\sigma_i$ maps into a three-dimensional submanifold (two base dimensions plus one fibre dimension) $\mathcal{N}:=\mathbb{R}\times (0,r_H)\times I \subset \mathcal{M}_{int}$, with the metric tensor
$$
g_{\mathcal{N}} = -\exp\left(-2\psi(r)\right)\left(1-\frac{2m(r)}{r}\right)dt^2+\left(1-\frac{2m(r)}{r}\right)^{-1}dr^2+r^2 ds_1^2,
$$
induced from the metric tensor $g$ (see \eqref{sssst}) of the ambient spacetime. Here, $I$ is an open interval containing $[0,1]$. In these coordinates, we have $\sigma_i(r):[r_b,r_a]\to N$ $\sigma_i(r) = ((\sigma_i)_t(r),r,(\sigma_i)_{s_1}(r))$. We define a timelike homotopy $\Gamma:[0,1]\times[r_b,r_a] \to N$ between $\sigma_1$ and $\sigma_2$ by,
\begin{equation}\label{homotopy-2}
    \Gamma(\alpha,r)\coloneqq ((1-\alpha)(\sigma_1)_t (r)+\alpha (\sigma_2)_t(r),r,(1-\alpha)(\sigma_1)_{s_1} (r)+\alpha (\sigma_2)_{s_1}(r)).
\end{equation}
We then ensure that each such curve $\Gamma(\alpha,.)$ is timelike for all $r\in[r_b,r_a]$ by computing
\begin{equation}\label{g_N}
    \begin{aligned}
        g_{\mathcal{N}} (\partial_r \Gamma(\alpha,r),\partial_r \Gamma(\alpha,r)) &= -\exp\left(-2\psi(r)\right)\left(1-\frac{2m(r)}{r}\right)[(1-\alpha)\dot{(\sigma_1)}_t(r)+\alpha \dot{(\sigma_2)}_t(r)]^2\\
        & \ \ \  + \left(1-\frac{2m(r)}{r}\right)^{-1} + r^2 [(1-\alpha)\dot{(\sigma_1)}_{s_1}(r)+\alpha \dot{(\sigma_2)}_{s_1}(r)]^2\\
        & \leq -\exp\left(-2\psi(r)\right)\left(1-\frac{2m(r)}{r}\right)(1-\alpha)[\dot{(\sigma_1)}_t(r)]^2+\alpha [\dot{(\sigma_2)}_t(r)]^2\\
        & \ \ \  + \left(1-\frac{2m(r)}{r}\right)^{-1} + r^2 (1-\alpha)[\dot{(\sigma_1)}_{s_1}(r)]^2+\alpha [\dot{(\sigma_2)}_{s_1}(r)]^2\\
        & =(1-\alpha)\times g_{\mathcal{N}}\left(\dot \sigma_1 (r), \dot \sigma_1(r)\right)+\alpha \times g_{\mathcal{N}}\left(\dot \sigma_2 (r), \dot \sigma_2(r)\right)<0.
    \end{aligned}
\end{equation}
Therefore, $\Gamma$ is a timelike homotopy with fixed endpoints $(t_a,r_a, y_a)$ and $(t_b,r_b, y_b)$ between $\sigma_1$ and $\sigma_2$. The concatenation of $\Gamma_1$, $\Gamma$, and $\Gamma_2$ gives a timelike homotopy between $\gamma_1$ and $\gamma_2$. Hence $\mathcal{M}_{int}\cap\{r<\tilde r\}$ is future one-connected.
\end{proof}

\begin{figure}
    \centering
    \includegraphics[width=0.5\linewidth]{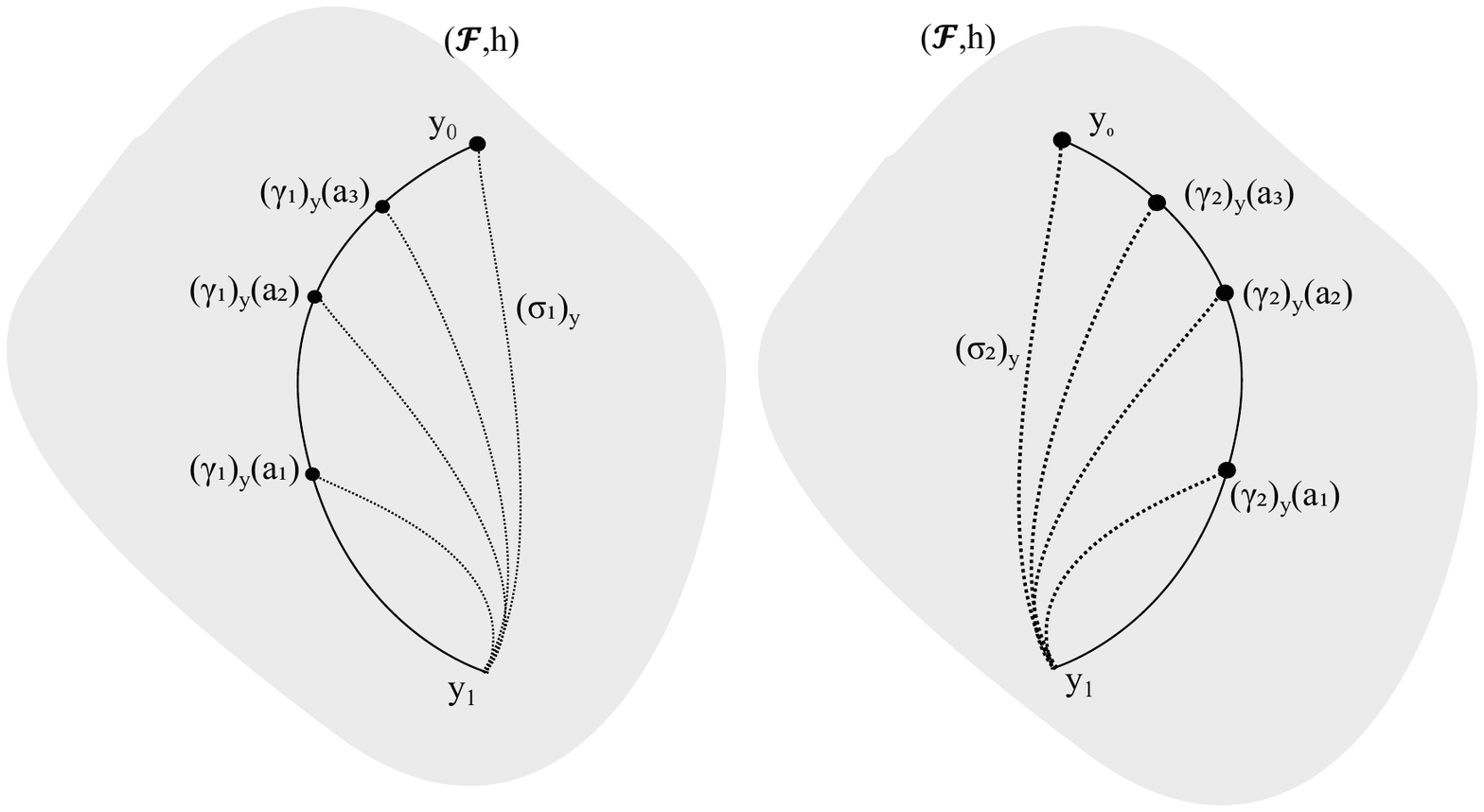}
    \caption{Homotopies between the fibre-component of $\gamma_i$ and the geodesic $\left(\sigma_i\right)_{y}$ in $\mathcal{F}$. Here $i\in\{1,2\}$. The lift of these homotopies to the spacetime are timelike homotopies between $\gamma_i$ and $\sigma_i$. Additionally, $\textrm{Im}\left(\left(\sigma_1\right)_y\right)=\textrm{Im}\left(\left(\sigma_2\right)_y\right)$ in $\mathcal{F}$.}
    \label{fig:foc}
\end{figure}

Next, we establish another property of the class of spacetimes described by \eqref{sssst} in \Cref{sec2.2} in the limit as $r \to 0$. In particular, we identify a subclass exhibiting the phenomenon commonly referred to as ``spaghettification.'' The following lemma is adapted from Step 3 in the proof of Theorem 1 of \cite{Sbierski_2018} to the class of spacetimes under consideration, and provides an obstruction to its $C^0$-extension. More precisely, assuming the existence of a $C^0$-extension through the central singularity will lead to a conclusion that contradicts \Cref{obstruction}. It therefore follows that no such $C^0$-extension through the central singularity can exist.
\begin{lemma}\label{obstruction}
    Let $\Sigma_n\coloneqq \{r=1/n\}$ be a family of Cauchy hypersurfaces 
    \footnote{In $\mathcal M_{\mathrm{int}}$, $-\partial_r$ is future-directed timelike. Thus, for every nonzero future-directed causal vector $X$, $g(X,\partial_r)
=-\left(1-2m(r)/r\right)^{-1}X(r)<0,$
so $X(r)<0$. Hence $r$ decreases strictly along every future-directed causal curve $\gamma$. If such a curve is inextendible, then $r\circ\gamma\to0$, since a positive limiting radius would permit extension within the smooth interior. Therefore, every inextendible causal curve intersects each $\{r=r^*\}$, $r^*<r_H$, exactly once; these hypersurfaces are Cauchy.}
of the interior of the black hole spacetimes described in  \Cref{sec2.2}, and $\bar{g}_n$ be its induced Riemannian metric tensor. Let $\gamma:[-1,0)\to \mathcal{M}_{\textrm{int}}$ be a future-directed and future-inextendible timelike curve. If the metric coefficients satisfy conditions $m_0=\lim_{r\to 0}m(r)>0$, and $\sup_{r\in(0,r_H]}\vert\psi(r)\vert<\infty$, then for any given $\mu\in(0,1]$,  there exists a sequence of pair of points $p_n,q_n \in I^{+}(\gamma(-\mu), \mathcal{M}_{\mathrm{int}})\cap \Sigma_n$ such that the $\bar{g}_n$-distance between them diverges as $n\to \infty$.
\end{lemma}

\begin{proof}
   After fixing a $-\mu\in[-1,0)$, let $(t_0,r_0,y_0)$ be the coordinates of $\gamma\left(-\mu/2\right)$. We know that $I^+\left(\gamma\left(-\mu\right)\right)$ is open, hence there exists a $\lambda\in \mathbb{R}^+$ such that 
   $$
   (t_0-\lambda,t_0+\lambda)\times \{r_0\}\times \{y_0\}\subseteq I^+\left(\gamma\left(-\mu\right),\mathcal{M}_{int}\right).
   $$ 
   Additionally, $-\partial_r$ is future-directed and timelike vector field in $(\mathcal{M}_{\textrm{int}},g_{\textrm{int}})$. Therefore, we have,
    $$
    [t_0-\lambda,t_0+\lambda]\times (0,r_0)\times \{y_0\}\subseteq I^+\left(\gamma\left(-\mu\right),\mathcal{M}_{int}\right).
    $$
    Now, construct a sequence of pairs of points
    $
    p_n:=(t_0-\lambda,\frac{1}{n}, y_0) \quad \textrm{and} \quad q_n:=(t_0+\lambda,\frac{1}{n}, y_0).
    $
    Eventually, $p_n,~q_n\in I^+\left(\gamma\left(-\mu\right),\mathcal{M}_{int}\right)$ (see \Cref{fig:distanceblowup}) The induced metric tensor $\bar{g}_n$ on $\Sigma_n$ is given by
    $$
    \bar{g}_n=-\exp{\left(-2\psi\left(\frac{1}{n}\right)\right)}\left(1-2m\left(\frac{1}{n}\right)n\right)dt\otimes dt+n^{-2}h.
    $$
    \begin{figure}[h]
    \centering
    \includegraphics[width=0.5\linewidth]{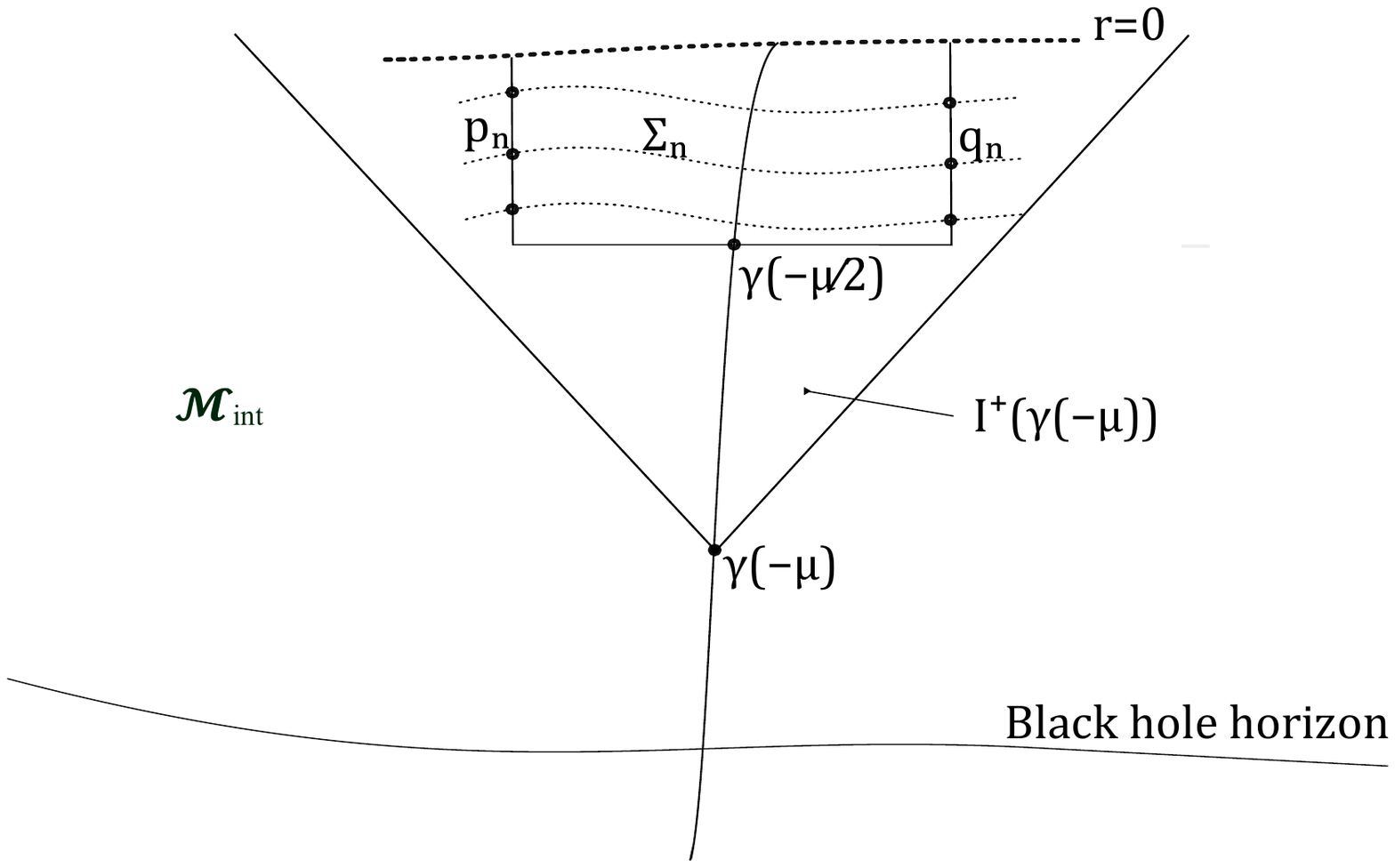}
    \caption{A sequence of points $(p_n,q_n)$ lying on the sequence of Cauchy surfaces $\Sigma_n:=\left\{1/n\right\}$. The distance $d_{\bar g_n}\left(p_n,q_n\right)$ blows to infinity as $n\to \infty$, or as $r\to 0$. }
    \label{fig:distanceblowup}
\end{figure}
    The shortest curve between $p_n$ and $q_n$ is given by $\gamma_n:(-\lambda,\lambda)\to\Sigma_n$, where 
    $
    \gamma_n:=(t+s,\frac{1}{n},y_0).
    $
    The $\bar{g}_n$-length of this curve, which is also the $\bar{g}_n$-distance between points $p_n$ and $q_n$ can then be calculated as 
    \begin{align*}
   d_{\bar{g}_n}\left(p_n,q_n\right)= L_{\bar{g}_n}(\gamma_n) &:=\int^{\lambda}_{-\lambda}\sqrt{\bar{g}_n\left(\dot \gamma_n(s),\dot \gamma_n(s)\right)} ~ds\\
    &= \int^{\lambda}
_{-\lambda}\sqrt{-\exp{\left(-2\psi\left(\frac{1}{n}\right)\right)\left(1-2m\left(\frac{1}{n}\right)n\right)}}~ds\\
&= 2\lambda\sqrt{\exp{\left(-2\psi\left(\frac{1}{n}\right)\right)\left(2m\left(\frac{1}{n}\right)n-1\right)}}. 
\end{align*}    
This distance $d_{\bar{g}_n}\left(p_n,q_n\right)$ diverges as $n\to \infty$ (or as $r\to 0$), as can be seen in the following (use conditions $m_0=\lim_{r\to 0}m(r)>0$, and $\sup_{r\in(0,r_H]}\vert\psi(r)\vert<\infty$)
\begin{align*}
    \lim_{n\to\infty}d_{\bar{g}_n}(p_n,q_n)&=\lim_{n\to \infty}2\lambda\sqrt{\exp{\left(-2\psi\left(\frac{1}{n}\right)\right)}\left(2m\left(\frac{1}{n}\right)n-1\right)}\\
    &=\lim_{r\to 0}2\lambda\sqrt{\exp{\left(-2\psi\left(r\right)\right)}\left(\frac{2m\left(r\right)}{r}-1\right)}=\infty.
\end{align*}
    \end{proof}

In \cite{Sbierski_2018}, Sbierski uses the spherical symmetry of Schwarzschild spacetime at several parts in the proof. Here, we show that a symmetry weaker than full spherical symmetry is already sufficient for the argument to go through. The following lemma identifies the relevant symmetry, namely a symmetry of the fibre component of the spacetime. This symmetry will play a key role in the proof of \Cref{Wtls2}, where, under the assumption that a $C^0$-extension through the interior (or equivalently through the central singularity) exists, one shows that a certain set timelike-separates the future of a point on a specific future-directed, future-inextendible timelike curve from a point lying to its past on the same curve. 
\begin{lemma}\label{lemmaofisometry}
Let $(\mathcal{F},h)$ be a compact connected Riemannian manifold, and let $G \subseteq \textrm{Isom}(\mathcal{F},h)$ be a Lie subgroup acting transitively on $\mathcal{F}$. Then for every $\delta>0$ there exists $\epsilon>0$ such that for any $p,q\in \mathcal{F}$,
$$
d_h(p,q)<\epsilon \implies~\exists~f\in G~:~f(p)=q,~\textrm{and}~\sup_{y\in F} d_h\bigl(f(y),y\bigr)<\delta. 
$$
\end{lemma}
\begin{proof}

Fix a basepoint $o\in \mathcal F$, and consider the orbit map 
$$
\pi_o:G\to \mathcal F~:~\pi_o(\phi)=\phi\cdot o.
$$ 
Since the action of $G$ on $\mathcal{F}$ is transitive, $\pi_o$ is surjective. Moreover, because the action is smooth, the derivative map at $e\in G$ (where $e$ is the identity element of $G$) exists, and is given by 
$$
\left(d\pi_o\right)_e:T_eG\to T_o \mathcal{F}.
$$ 
The image of this map is the tangent space at $o$ to the orbit $G\cdot o$, i.e. $\textrm{Im}\left(\left(d\pi_o\right)_e\right)= T_o\left(G\cdot o\right)$. Since $\pi_o$ is surjective, we have $T_o\left(G\cdot o\right)=T_o\left(\mathcal{F}\right)$, and hence, this derivative map is surjective. Thus, by the submersion theorem, there exists an open neighbourhood $\mathcal{U}\subset \mathcal{F}$
of $o$ and a smooth map
$$
\mathcal{S}:\mathcal{U}\to G~:~\mathcal{S}(o)=e,~\textrm{and}~\pi_o(\mathcal{S}(y))=\mathcal{S}(y)\cdot o=y,\quad\forall ~y\in \mathcal{U}.
$$
In other words, $\mathcal{S}(y)$ is a smoothly chosen group element that sends the base point $o$ to a nearby point $y$, where $y\in \mathcal{U}$.

Fix $\delta>0$. Since $\mathcal S(o)=e$, and since $\mathcal F$ is compact, the convergence $\mathcal S(y)\to e$ as $y\to o$ is uniform on $\mathcal F$. Therefore, after shrinking $\mathcal U$ if necessary, we may assume that
$$
\sup_{x\in\mathcal F} d_h\bigl(\mathcal S(y)\cdot x,x\bigr)<\delta ~\forall ~y\in\mathcal U.
$$
Choose \(\epsilon>0\) such that
$
B_\epsilon(o)\subset \mathcal U.
$
Let $p,q\in\mathcal F$ satisfy $d_h(p,q)<\epsilon$. By transitivity of $G$ on $\mathcal{F}$, we can choose $j\in G$ with
$
j(o)=p.
$
Set
$$
v:=j^{-1}(q).
$$
Since $j$ is an isometry,
$$
d_h(o,v)
=
d_h(j^{-1}p,j^{-1}q)
=
d_h(p,q)
<
\epsilon.
$$
Thus $v\in B_\epsilon(o)\subset\mathcal U$. Now define
$
f:=j\mathcal S(v)j^{-1}.
$
Then \(f\in G\), and
\begin{equation}\label{isometryeq1}
    f(p)
=
j\mathcal S(v)j^{-1} \cdot p
=
j\mathcal S(v) \cdot o
=
j\cdot v
=
q.
\end{equation}
Finally, for every $x\in\mathcal F$, using again that $j$ is an isometry,
$$
d_h(f(x),x)
=
d_h\bigl(j\mathcal S(v)j^{-1}(x),x\bigr)
=
d_h\bigl(\mathcal S(v)(j^{-1}x),j^{-1}x\bigr).
$$
Taking the supremum over $x\in\mathcal F$, and writing $y=j^{-1}x$, gives
\begin{equation}\label{isometryeq2}
    \sup_{x\in\mathcal F} d_h(f(x),x)
=
\sup_{y\in\mathcal F} d_h\bigl(\mathcal S(v)\cdot y,y\bigr)
<
\delta.
\end{equation}
From \eqref{isometryeq1} and \eqref{isometryeq2}, we have the desired $f$, thereby completing the proof (see \Cref{isometrylemmafigure} for reference).

\end{proof}
\begin{figure}[h]
    \centering
\includegraphics[width=0.5\linewidth]{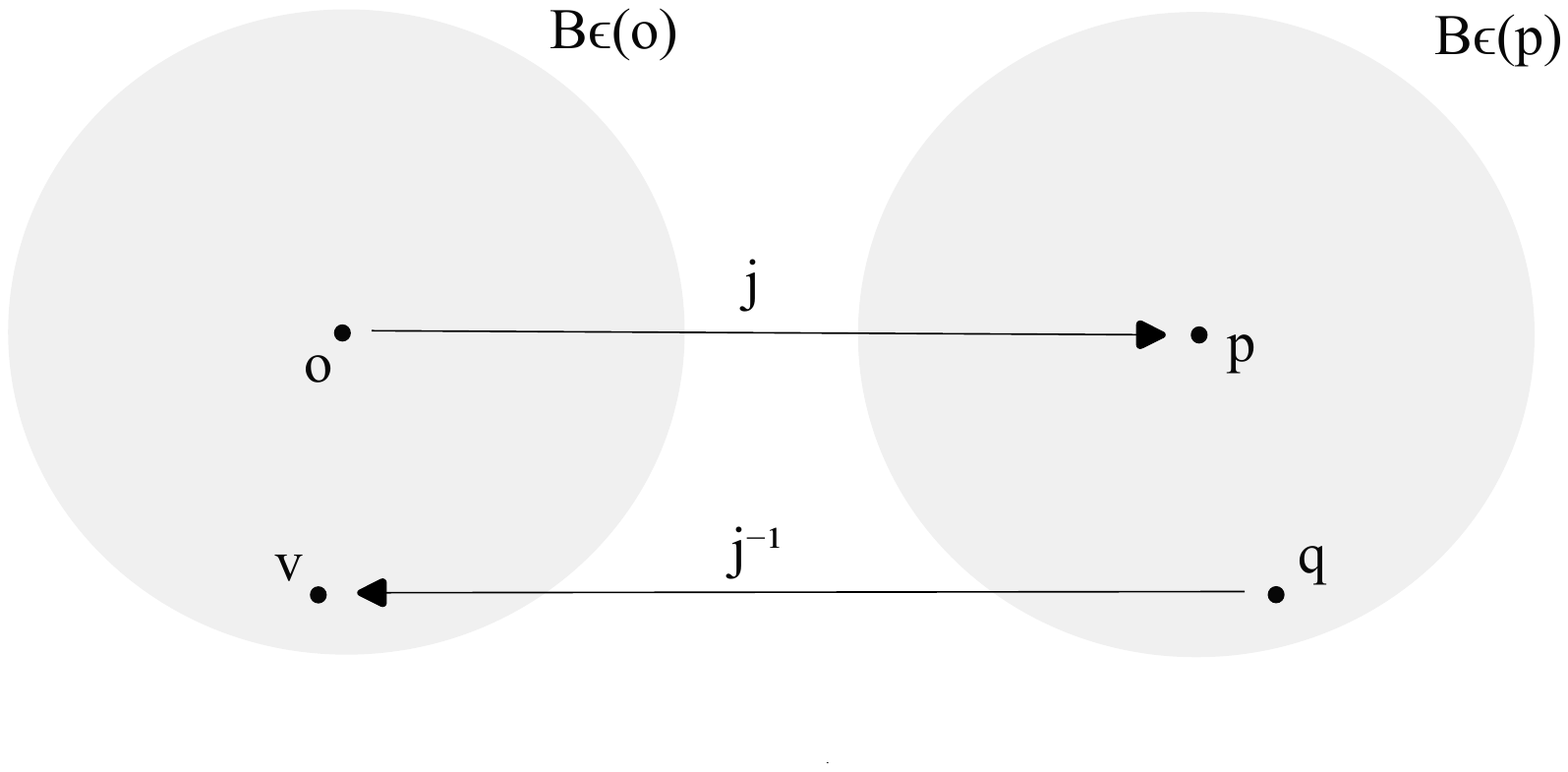}
    \caption{Consider an arbitrary $p\in\mathcal{F}$, where $(\mathcal{F},h)$ is a compact, connected Riemannian manifold admitting a Lie subgroup $G$ (of isometry group of $(\mathcal{F},h)$) acting transitively on $\mathcal{F}$ with base point $o\in\mathcal{F}$. If $q\in B_{\epsilon}(p)$ for some $p\in\mathcal{F}$, $\epsilon>0$, and $j\in G$ such that $j\cdot o=p$, then $v:=j^{-1}\cdot q\in B_{\epsilon}(o)$.}
    \label{isometrylemmafigure}
\end{figure}

The next  \Cref{prerequisite to clKcompact} is a prerequisite to \Cref{lemmakclcompact}, which is  then used in the argument for proving \Cref{maintheorem}. 

\begin{lemma}\label{prerequisite to clKcompact} 
Let $(\mathcal F,h)$ be compact, connected, and homogeneous, and let $G\subset \textrm{Isom}(\mathcal F,h)$ be a Lie subgroup acting transitively on $\mathcal F$. Let $\gamma: [-1,0)\to \mathcal{M}_{\textrm{int}}$ be future-directed and future-inextendible timelike curve. Let $\tau_1\in(-1,0)$. Then, for any $\tau_0\in(\tau_1,0)$, there exists a $\delta>0$ and a neighborhood $V\subset G$ of the identity element $e\in G$, such that 
\begin{equation*}
\big(\gamma_t(\tau) - \delta, \gamma_t(\tau) +\delta\big) \times \big{\{}\gamma_r(\tau)\big{\}} \times \left\{a\cdot\gamma_y(\tau)\vert a\in V\right\} \subseteq I^+\big(\gamma(\tau_1),\mathcal{M}_{\textrm{int}}\big)
\end{equation*}
holds for all $\tau \in (\tau_0,0)$.
\end{lemma}
\begin{proof}
    Since $\gamma$ is timelike and $g_{\textrm{int}}$ is continuous, for any $\tau_0\in(\tau_1,0)$, there exists a $\mu >0$ such that $\forall ~\tau \in [\tau_1,\tau_0]$,
\begin{equation}
\label{gdotgammas}
g_{\textrm{int}}\big(\dot{\gamma}(\tau), \dot{\gamma}(\tau)\big) < - \mu. 
\end{equation}

We now choose functions $\lambda\in C^\infty([\tau_1,\tau_0],\mathbb R)$ and $a\in C^\infty([\tau_1,\tau_0],G)$, and define a curve $\sigma : [\tau_1,\tau_0] \to \mathcal{M}_{\textrm{int}}$ by 
$$
\sigma(\tau) := (\gamma_t(\tau)+\lambda(\tau),\gamma_r(\tau),a(\tau)\cdot\gamma_y(\tau)).
$$ 
Since the $G$-action is by isometries of $(\mathcal F,h)$, and since $\gamma_r(\tau)$ is bounded away from $0$ on $[\tau_1,\tau_0]$, it follows from \eqref{gdotgammas} that there exists $\eta>0$ such that $\sigma$ is timelike whenever 
\begin{equation}\label{smallnesslemma4}
    \|\dot\lambda\|_{L^\infty([\tau_1,\tau_0])} + \|\dot a\|_{L^\infty([\tau_1,\tau_0])} <\eta .
\end{equation}
Here, the norm on $\dot a$ is taken with respect to any fixed Riemannian metric on the Lie group $G$. Hence there exist $\delta>0$ and an open neighbourhood $V\subset G$ of the identity $e\in G$ such that, for every $\lambda_{\tau_0}\in(-\delta,\delta)$ and every $a_{\tau_0}\in V$, one can choose smooth $\lambda$ and $a$ satisfying 
$$
\lambda(\tau_1)=0,\qquad \lambda(\tau_0)=\lambda_{\tau_0}, \qquad a(\tau_1)=e,\qquad a(\tau_0)=a_{\tau_0},
$$
and such that the inequality \eqref{smallnesslemma4} holds. Thus $\sigma$ is timelike from $\gamma(\tau_1)$ to $ (\gamma_t(\tau_0)+\lambda_{\tau_0}, \gamma_r(\tau_0), a_{\tau_0}\cdot\gamma_y(\tau_0)).$  Now consider the curve 
$$
\hat\sigma(\tau) := (\gamma_t(\tau)+\lambda_{\tau_0}, \gamma_r(\tau), a_{\tau_0}\cdot\gamma_y(\tau)), \qquad \tau\in[\tau_0,0). 
$$
This curve is timelike because it is the image of $\gamma|_{[\tau_0,0)}$ under the spacetime isometry $(t,r,y)\mapsto (t+\lambda_{\tau_0},r,a_{\tau_0}\cdot y).$ Concatenating $\sigma$ and $\hat\sigma$, the claim follows (see \Cref{fig:clK}).

\end{proof}

\begin{lemma}\label{lemmakclcompact}
 Let $\gamma:[-1,0)\to \mathcal{M}_{\textrm{int}}$ be a future-directed and future-inextendible timelike curve with compact, connected, and homogeneous fibre $\mathcal F$ (like in \Cref{prerequisite to clKcompact}). For any $\tau_1,\tau_2\in(-1,0)$, with $\tau_2<\tau_1$,  define
\begin{equation}\label{Kdefn}
  K \left(\tau_1,\tau_2\right):=
\Big[\Big(\bigcup_{\tau_2 < \tau < 0} I^-\big(\gamma(\tau), \mathcal{M}_{\textrm{int}} \big)\Big) \cap I^+\big( \gamma(\tau_2), \mathcal{M}_{\textrm{int}}\big) \Big] \setminus I^+\big( \gamma(\tau_1), \mathcal{M}_{\textrm{int}}\big).  
\end{equation}
Fix $\tau_1$ and $\tau_2$. Then $\bar{K}\left(\tau_1,\tau_2\right)$ (closure of $K$) is compact in $M_{\textrm{int}}$.
\end{lemma}
\begin{proof}
In \Cref{prerequisite to clKcompact}, we fix $\tau_0\in(\tau_1,0)$, and obtain $\delta>0$ and $V\subset G$. We reparametrise $ \gamma:[-1,0)\to \mathcal{M}_{\textrm{int}}$ by the radial coordinate along $\gamma$ and obtain a new curve $\hat{\gamma}:[-\gamma_r(\tau_1),0)\to M_{\textrm{int}}$ (whose image is the same as that of $\gamma$). The reparametrisation map
\begin{align*}
    X:[\tau_1,0)&\to [-\gamma_r(\tau_1),0)
    \\
    \tau&\mapsto X(\tau)=:s
\end{align*}
is a strictly monotone diffeomorphism onto its image such that $X(\tau_1)=-\gamma_r(\tau_1)$, $\lim_{\tau\to 0}X(\tau)=0$. Let us define $s_0:=X(\tau_0)$. Since $\hat\gamma$ is parametrised by the radial coordinate, we can use \Cref{lemma1} as follows: There exists $\bar s\in(s_0,0)$ such that 
\begin{equation*}
\begin{aligned}
   &  I^-\left(\hat\gamma(s),M_{\textrm{int}}\right)\cap\left\{r\leq\hat \gamma_r(\bar s)\right\}\\ &\subseteq    \bigcup_{s_0\leq s'<0} (\hat \gamma_t (s')-\delta,\hat \gamma_t(s')+\delta) \times \{\hat \gamma_r(s')\}\times \left\{ a\cdot \hat\gamma_y (s')\vert a\in V\right\},
\end{aligned}
\end{equation*}
for all $s\in (-\gamma_r(\tau_1),0)$. Reparametrising  $\hat\gamma$ back by $\tau$, the last line above can be rewritten as
\begin{equation*}
    \bigcup_{\tau_0\leq \tau'<0}
    \bigl(
        \gamma_t(\tau')-\delta,
        \gamma_t(\tau')+\delta
    \bigr)
    \times
    \{\gamma_r(\tau')\}
    \times
    \left\{ a\cdot \gamma_y (\tau')\vert a\in V\right\},
\end{equation*}
where $\tau':=X^{-1}(s')$. Additionally,
we obtain
$$
  I^-\left(\hat \gamma(s),\mathcal{M}_{\textrm{int}}\right) = I^-\left(\gamma (X^{-1}(s)),\mathcal{M}_{\textrm{int}}\right)= I^-\left(\gamma(\tau),\mathcal{M}_{\textrm{int}}\right) ~\forall~\tau\in(\tau_1,0).
$$
We can hence conclude that there exists $X^{-1}(\bar s)=:\bar\tau\in (\tau_0,0)$ such that
\begin{equation}\label{subsetkclosurecompact}
      I^-\left(\gamma(\tau),\mathcal{M}_{\textrm{int}}\right)\cap\left\{r\leq\gamma_r(\bar \tau)\right\} \subseteq I^{+}(\gamma(\tau_1),\mathcal{M}_{\text{int}})~\forall~ \tau\in (\tau_1,0),
\end{equation}
where we used  \Cref{prerequisite to clKcompact} to get the expression on the right hand side (see \Cref{fig:clK}). This implies 
$K \subseteq \mathbb{R} \times \big(\gamma_r(\bar{\tau}), \gamma_r(\tau_2)\big) \times \mathcal{F} \;.$ Moreover, the bound \eqref{tbound} implies that there are $t_0, t_1\in\mathbb{R}$ such that 
$$
I^+\big(\gamma(\tau_2),\mathcal{M}_{\textrm{int}}\big) \subseteq (t_0,t_1) \times \big(0, \gamma_r(\tau_2)\big) \times \mathcal F \;.
$$
It follows that 
$$
K \subseteq (t_0,t_1) \times \big(\gamma_r(\bar{\tau}), \gamma_r(\tau_2)\big) \times \mathcal F ,
$$
which implies that $\overline{K}$ is compact, since $\mathcal F$ is compact.
\begin{figure}[h]
    \centering
    \includegraphics[width=0.5\linewidth]{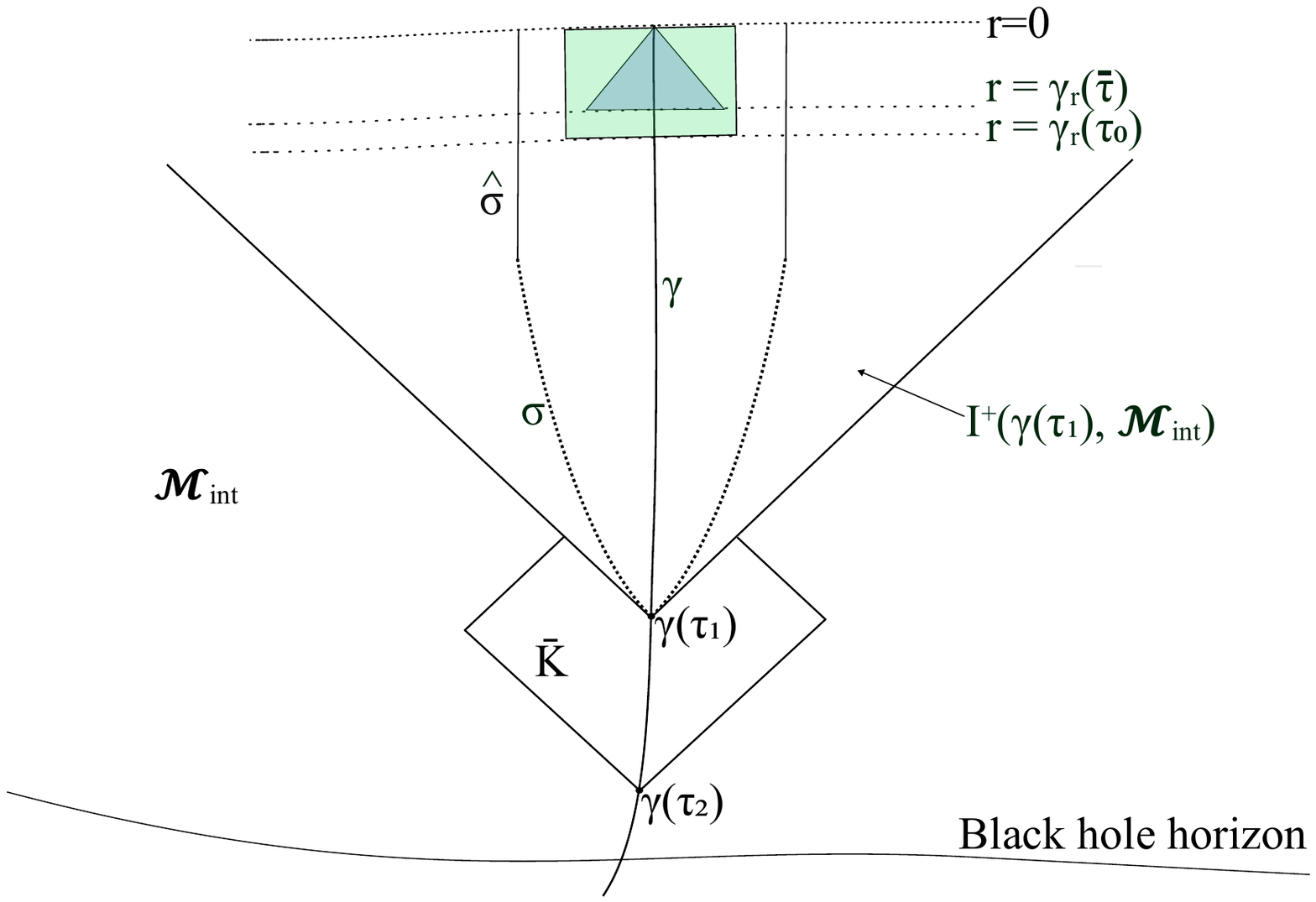}
    \caption{The picture depicts the timelike curves $\sigma$ and $\hat\sigma$ constructed in the proof of \Cref{prerequisite to clKcompact}, the closure of set $K$ defined in Eq. \eqref{Kdefn}, and the statement \eqref{subsetkclosurecompact} in the proof of  \Cref{lemmakclcompact}}
    \label{fig:clK}
\end{figure}
\end{proof}

Now that all the necessary auxiliary results and definitions have been established, we are ready to proceed to the proofs of the theorems.

\section{Proof of \Cref{maintheorem} -- No \texorpdfstring{$C^0$}{C0}-extension through the central curvature singularity}\label{sec4}
The proof follows Sbierski's strategy in proving Theorem 1 in \cite{Sbierski_2018}, and is as follows.  Assume that there exists a $C^0$-extension $\iota: \mathcal{M}\hookrightarrow \tilde{ \mathcal{M}}$ with the following property: there exists an affinely parametrised, future-directed and future-inextendible timelike geodesic $\rho:[0,b)\to \mathcal{M}$, where $b<\infty$, with the $\lim_{s\to b} r\circ \rho (s)=0$, and such that its future endpoint is contained in $\tilde{\mathcal{M}}$, i.e.,  $\lim_{s\to b}\left(\iota\circ \rho\right)(s)\in \partial ^{+}\iota\left(\mathcal{M}\right)\subset\tilde{\mathcal{M}}$. The \Cref{minkowskianchartproposition} and \Cref{minkowskianchartproposition2} imply the existence of the Minkowskian chart $\tilde \varphi: \tilde U\to \mathbb{R}_{\varepsilon_0,\varepsilon_1}$ which is centred at $\tilde p:=\lim_{s\to b}\left(\iota\circ \rho\right)(s)$, i.e., $\lim_{s\to b}\left(\tilde \varphi\circ\iota\circ \rho\right)(s)=\left(0,..,0\right)$. 
 
Hereafter, we work with the curve $\tilde{\gamma} : (-\varepsilon_0, 0] \to \tilde{\mathcal{M}}$, which, in the chart  $\tilde{\varphi}$ is given by a vertical line $(-\varepsilon_0, 0] \ni s \mapsto (s, 0, \ldots, 0)$ \footnote{Here, by abuse of notation, the parameter $s$ in the definition of $\tilde\gamma$ is the coordinate parameter in the Minkowskian chart, and takes values in $[-\varepsilon_0,0)$. It is different from the affine parameter of $\rho$ which is also denoted by $s$ but takes values in $[0,b)$, where $b<\infty$.}. We then consider $\gamma :=  \iota^{-1} \circ \tilde{\gamma}|_{(-\varepsilon_0, 0)}$, which is a future-directed and future-inextendible (not assumed to be affinely parametrised) timelike curve in $\mathcal{M}$, which can be shown to have the property that $(r \circ \gamma)(s) \to 0$ for $s \to 0$. 

The first step is then to show the following statement:
 There exists a $\mu>0$ such that
    \begin{enumerate}[(A):]
\item \label{statementA} $\iota \Big( I^+\big( \gamma(-\mu), \mathcal{M}\big)\Big) \subseteq \tilde{\varphi}^{-1} \left(f_{<}\right)$
\item \label{statementB} $(-\varepsilon_0, -\frac{49}{50} \varepsilon_0) \times (-\varepsilon_1, \varepsilon_1)^d \subseteq I^-\left(\left(\tilde{\varphi}\circ \tilde{\gamma}\right)(-\mu), \mathbb{R}_{\varepsilon_0,\varepsilon_1}\right) \;.$
\end{enumerate}
 The next step is to prove the following: The \Cref{statementA} and \Cref{statementB} together imply \Cref{statementC}, where 
 \begin{enumerate}[(A):]
 \setcounter{enumi}{2}
     \item \label{statementC} There exists a constant $0<C_d < \infty$ such that for any Cauchy hypersurface $\Sigma\subset \mathcal{M}$ such that $\Sigma \cap I^+\left(\gamma(-\mu), \mathcal{M}_{\textrm{int}}\right)\neq \emptyset$,  the distance in $\Sigma$ of any two points in $I^+\left(\gamma(-\mu), \mathcal{M}_{\textrm{int}}\right) \cap \Sigma$ is bounded from above by $C_d$. 
 \end{enumerate}
 Finally, the statement \Cref{statementC} contradicts \Cref{obstruction}. This then implies that $\tilde \gamma$ as defined above does not have the property that  $(r \circ \gamma)(s) \to 0$ for $s \to 0$, which will in turn imply that $\rho$ as defined above does not have the property that  $(r \circ \rho)(s) \to 0$ for $s \to b$, which will finally imply that the assumption of $C^0$-extension through $r\to0$ is not true. 

\subsection{The setup in the Minkowskian chart}\label{thesetupmainproof}
We consider the following notations.
    \begin{align*}
      &  C^+_a := \big{\{} X \in \mathbb{R}^{4} \, | \, \frac{<X,e_0>_{\mathbb{R}^{4}}}{|X|_{\mathbb{R}^{4}}}   > a \big{\}}\\
       & C^-_a := \big{\{} X \in \mathbb{R}^{4} \, | \, \frac{<X,e_0>_{\mathbb{R}^{4}}}{|X|_{\mathbb{R}^{4}}}   < -a \big{\}}\\
       & C^c_a := \big{\{} X \in \mathbb{R}^{4} \, | \, -a < \frac{<X,e_0>_{\mathbb{R}^{4}}}{|X|_{\mathbb{R}^{4}}}   < a \big{\}}.
    \end{align*}
We have $C_a^+$ to denote the cone of vectors whose angle with the $x_0$--axis is strictly smaller than $\cos^{-1}(a)$ and which points in the positive $x_0$--direction. We call it the forward cone of vectors. Analogously, $C_a^-$ denotes the corresponding cone in the negative $x_0$--direction, and called the backward cone. In the Minkowski spacetime, the cones of future and past-directed timelike vectors are characterised by the threshold value
$
a = \cos\!\left(\frac{\pi}{4}\right) = \frac{1}{\sqrt{2}}.
$
Since
$
\frac{5}{8} < \frac{1}{\sqrt{2}} < \frac{5}{6},
$
we can fix $\delta=\delta_0 > 0$ such that, in the coordinate chart $\tilde{\varphi}$ which follows from the \Cref{minkowskianchartproposition} and \Cref{minkowskianchartproposition2}, all vectors belonging to $C_{5/6}^+$ are future-directed timelike, all vectors in $C_{5/6}^-$ are past-directed timelike, and all vectors in $C^c_{5/8}$ are spacelike. The following estimates then hold for $x \in \mathbb{R}_{\varepsilon_0\varepsilon_1}$.
\begin{equation}
\label{diamondsinchart}
\begin{split}
&\big( x + C^+_{5/6}\big) \cap  \mathbb{R}_{\varepsilon_0,\varepsilon_1} \subseteq I^+(x, \mathbb{R}_{\varepsilon_0,\varepsilon_1}) \subseteq \big( x + C^+_{5/8}\big) \cap  \mathbb{R}_{\varepsilon_0,\varepsilon_1} \\
&\big( x + C^-_{5/6}\big) \cap  \mathbb{R}_{\varepsilon_0,\varepsilon_1} \subseteq I^-(x, \mathbb{R}_{\varepsilon_0,\varepsilon_1}) \subseteq \big( x + C^-_{5/8}\big) \cap  \mathbb{R}_{\varepsilon_0,\varepsilon_1}  \;.
\end{split}
\end{equation}
We fix the coordinate chart given by $\tilde \phi$ such that $0< \varepsilon_1 < \frac{1}{2} \varepsilon_0$. This can be done without any loss of generality. Choose points
$
x^+ := (x^+_0,0,0,0)$, $x^+_0 \in(0, \varepsilon_0),
$ and $ x^- := (x^-_0,0,0,0)$, where $ x^-_0\in (-\varepsilon_0,0),
$
such that (see \Cref{fig:mincones} for reference)
\begin{equation}\label{minksetup2}
    \overline{\bigl(x^+ + C^-_{5/6}\bigr) \cap \bigl(x^- + C^+_{5/6}\bigr)}
\subset \mathbb{R}_{\varepsilon_0,\varepsilon_1}
\end{equation}
is compact. Next, fix
$
y^- := (y^-_0,0,\ldots,0)$, where $-\tfrac{1}{5}\varepsilon_0 < y^-_0 < 0,
$
with the property that
\begin{equation}\label{minksetup3}
    \overline{C^-_{5/8} \cap \bigl(y^- + C^+_{5/8}\bigr)}
\subset \bigl(x^+ + C^-_{6/7}\bigr) \cap \bigl(x^- + C^+_{6/7}\bigr).
\end{equation}
Under the assumptions \( -\tfrac{1}{5}\varepsilon_0 < y^-_0 \) and \( 0 < \varepsilon_1 < \tfrac{1}{2}\varepsilon_0 \), it then follows that
\begin{equation}\label{BottomofMink}
(-\varepsilon_0, -\tfrac{49}{50}\varepsilon_0) \times (-\varepsilon_1, \varepsilon_1)^d
\subset I^-\bigl(y^-, \mathbb{R}_{\varepsilon_0,\varepsilon_1}\bigr).
\end{equation}

\begin{figure}[h]
    \centering
\includegraphics[width=0.8\linewidth]{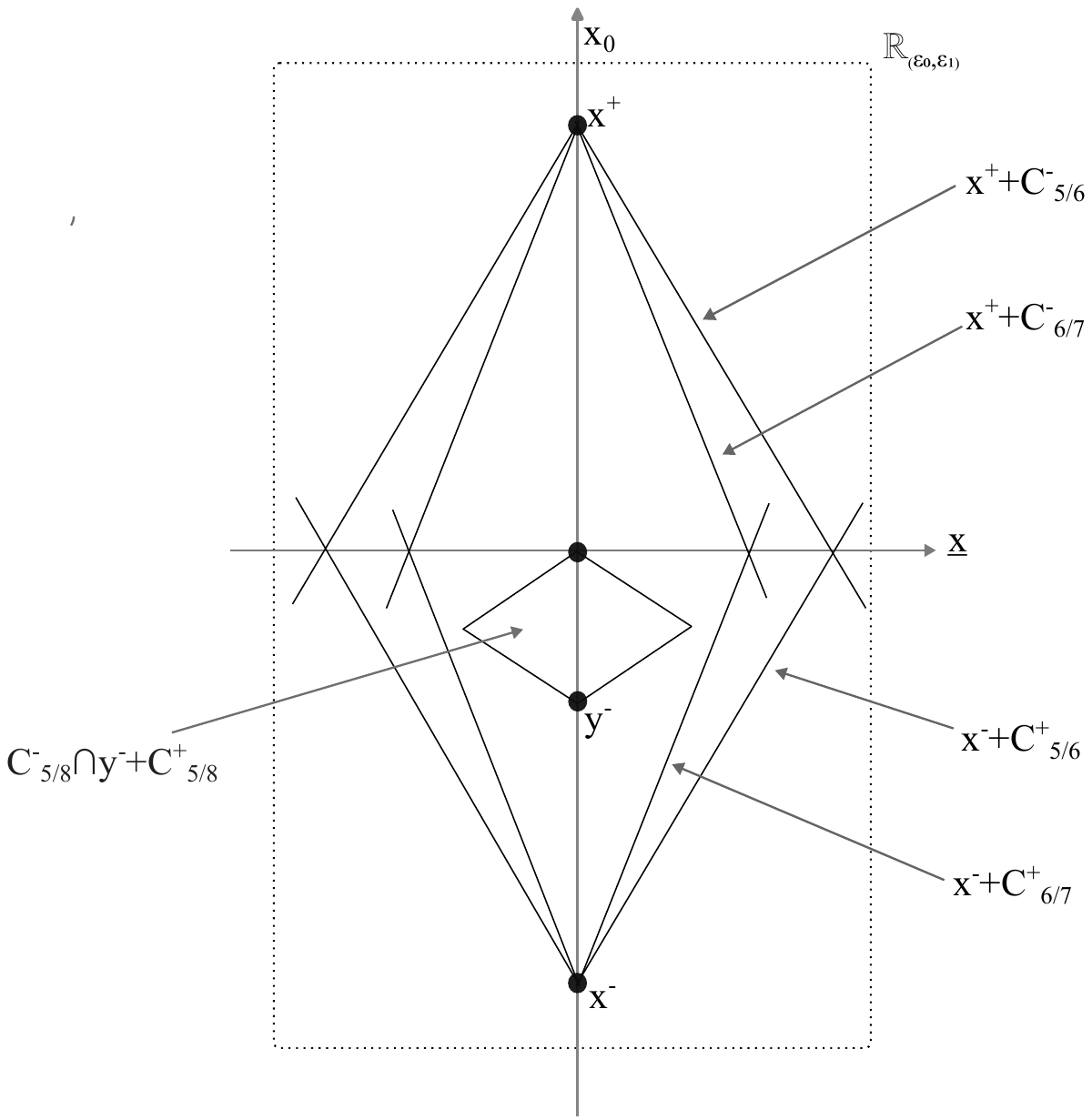}
    \caption{The setup of the Minkowskian chart in Sbierski's method \cite{Sbierski_2018}.}
    \label{fig:mincones}
\end{figure}

\subsection{The proof of  \Cref{statementA} and \Cref{statementB}}\label{sec4.2}
We now turn to the proof of \Cref{statementA} and \Cref{statementB}, using the Minkowskian chart from  4.1 together with the technical tools discussed in  \Cref{sec3}. To this end, we prove \Cref{causaldiamonds same}, \Cref{Ktls}, \Cref{Wtls}, and \Cref{Wtls2}. These were proved in \cite{Sbierski_2018} for Schwarzschild spacetime and extend to the more general class of spacetimes considered here, with the exception of \Cref{Wtls2}, whose proof additionally relies on spherical symmetry.

\begin{lemma}\label{causaldiamonds same}
Assume that there exists a $C^0$-extension of $\mathcal{M}_{\textrm{int}}$ through $r\to 0$. Let $\tilde r\in (0,r_H)$ be chosen as in \Cref{prop:foc}. Let $\gamma:[-\varepsilon_0,0)\to \mathcal{M}_{\tilde r}:=\mathcal{M}_{\textrm{int}}\cap\{r<\tilde r\}$ be a future-directed and future-inextendible timelike curve as described in the introduction of \Cref{sec4}, and consider the setup of the Minkowskian chart as described in  \Cref{thesetupmainproof}. Then, for all $s\in(y^-_0, 0)$, the following equality holds. 
\begin{equation}
\label{causaldiamondssameeq}
\tilde{\varphi}^{-1} \left(I^-\left((s,0, \ldots, 0), \mathbb{R}_{\varepsilon_0,\varepsilon_1}\right) \cap I^+\left(y^-, \mathbb{R}_{\varepsilon_0,\varepsilon_1}\right) \right) \\
= \iota \left(I^-\big(\gamma(s), \mathcal{M}_{\textrm{int}} \big) \cap I^+\left( \gamma(y^-_0), \mathcal{M}_{\textrm{int}}\right) \right) \;.
\end{equation}
\end{lemma}
\begin{proof}
The inclusion ``$\,\subseteq\,$'' is an immediate consequence of \Cref{minkremark}. Indeed, if $\sigma$ is a past-directed timelike curve in $\mathbb{R}_{\varepsilon_0,\varepsilon_1}$ from $(s,0,\ldots,0)$ to $y^-$, then \Cref{minkremark} implies that $\tilde{\varphi}^{-1}\circ \sigma$ is entirely contained in $\iota(\mathcal{M}_{\mathrm{int}})$.

For the reverse inclusion ``$\,\supseteq\,$'', let
$
\sigma \colon [y_0^-,s] \to \mathcal{M}_{\tilde r}
$
be a future-directed timelike curve from $\gamma(y_0^-)$ to $\gamma(s)$. \Cref{prop:foc} say that $\left(\mathcal{M}_{\tilde r},g_{\textrm{int}}\vert_{r<\tilde r}\right)$ is future one-connected. Hence, there exists a timelike homotopy
$
\Gamma \colon [0,1]\times [y_0^-,s] \to \mathcal{M}_{\tilde r}
$
with fixed endpoints connecting $\gamma|_{[y_0^-,s]}$ to $\sigma$, i.e., we have $\Gamma(0,\tau):=\gamma|_{[y_0^-,s]}(\tau)$, and $\Gamma(\tau,1):=\sigma(\tau)$. Consequently,
$
\iota\circ\Gamma \colon [0,1]\times [y_0^-,s] \to \tilde{\mathcal{M}}
$
is a timelike homotopy with fixed endpoints in $\tilde{\mathcal{M}}$. Additionally, $\tilde \gamma=\iota\circ \gamma$ maps into $\tilde U$.
Moreover, the set 
$$
I^-\left(\left(s,0,...0\right),\mathbb{R}_{\varepsilon_0,\varepsilon_1}\right)\cap I^+\left(y^-,\mathbb{R}_{\varepsilon_0,\varepsilon_1}\right)
$$
is precompact in $\mathbb{R}_{\varepsilon_0,\varepsilon_1}$ (refer to \eqref{minksetup2}, \eqref{minksetup3}, and \eqref{BottomofMink}). Also we know that $\tilde \phi$ is a homeomorphism from $\tilde U$ to $\mathbb{R}_{\varepsilon_0,\varepsilon_1}$. Therefore, the set 
$$
 \tilde\phi ^{-1}\left[I^-\left(\left(s,0,...0\right),\mathbb{R}_{\varepsilon_0,\varepsilon_1}\right)\cap I^+\left(y^-,\mathbb{R}_{\varepsilon_0,\varepsilon_1}\right)\right], 
$$
is precompact in $\tilde U$.  In other words,
$$
I^-\left(\tilde \gamma(s),\tilde U\right)\cap I^+\left(\tilde \gamma \left(y_0^-\right),\tilde U\right)\Subset \tilde U.
$$

One can now invoke \Cref{imageofhomotopy} to conclude that $\textrm{Im}\left(\iota\circ \Gamma\right)\subset \tilde U$. In particular, $\textrm{Im}\left(\iota \circ \sigma\right)=\textrm{Im}\left(\iota\circ \Gamma(s,1)\right)\subset \tilde U$. This establishes the inclusion ``$\,\supseteq\,$'' and thereby completes the proof of  \Cref{causaldiamondssameeq}.
\end{proof}
It is worth noting that \Cref{causaldiamonds same} is the only part of the entire paper where the future one-connectedness property of the small region of $\mathcal{M}_{\textrm{int}}$ is used. The following definition will be useful in the proof of \Cref{maintheorem}, more specifically in proving \Cref{Ktls}. 

\begin{definition} \label{tlseparates}[Timelike separation, Definition 2.14 in \cite{Sbierski_2018a}]
 For a time-oriented Lorentzian manifold $(\mathcal{M},g)$, given two sets $X,Y\subseteq \mathcal{M}$, are called timelike separated by a set $Z\subseteq \mathcal{M}$ if, and only if, every timelike curve connecting $X$ and $Y$ intersects $Z$. In other words, for any timelike curve $\xi:[0,1]\to \mathcal{M}$ such that $\xi(0)\in X$ and $\xi(1)\in Y$, there exists a $\tau_0\in [0,1]$ such that $\xi(\tau_0)\in Z$.  
\end{definition}

\begin{lemma}\label{Ktls}
  Assume that there exists a $C^0$-extension of $\mathcal{M}_{\textrm{int}}$ through $r\to0$. Let $\gamma:[-\varepsilon_0,0)\to \mathcal{M}_{\textrm{int}}$ be a future-directed and future-inextendible timelike curve as described in the introduction of \Cref{sec4}, and consider the setup of the Minkowskian chart as described in  \Cref{thesetupmainproof}. Choose $y_0^+\in (y_0^-,0)$. 
   Then, the set $K\left(y_0^+,y_0^-\right)$ (as defined in  \Cref{lemmakclcompact}) timelike separates $\gamma\left(\left(y_0^+,0\right)\right)$ and $I^-\left(\gamma\left(y_0^-\right),\mathcal{M}_{\textrm{int}}\right)$ (see  \Cref{tlseparates}).
\end{lemma}
\begin{proof}
Let
$\sigma: [0,1]\to \mathcal{M}_{\mathrm{int}}$
be a past-directed timelike curve such that
$
\sigma(0)\in \gamma\big((y_0^+,0)\big)
\quad\text{and}\quad
\sigma(1)\in I^-\big(\gamma(y_0^-),\mathcal{M}_{\mathrm{int}}\big).
$ We first claim that there exists \(\Delta_- \in (0,1)\) such that
$
\sigma^{-1}\Big(I^+\big(\gamma(y_0^-),\mathcal{M}_{\mathrm{int}}\big)\Big)=[0,\Delta_-).
$
Indeed, since \(\sigma(0)\in \gamma\big((y_0^+,0)\big)\) and \(\gamma(y_0^+)\in I^+(\gamma(y_0^-),\mathcal{M}_{\mathrm{int}})\), it follows that
\(
0\in \sigma^{-1}\Big(I^+\big(\gamma(y_0^-),\mathcal{M}_{\mathrm{int}}\big)\Big).
\)
Moreover, by continuity of \(\sigma\) and openness of \(I^+(\gamma(y_0^-),\mathcal{M}_{\mathrm{int}})\), the set
$
\sigma^{-1}\Big(I^+\big(\gamma(y_0^-),\mathcal{M}_{\mathrm{int}}\big)\Big)
$
is open in \([0,1]\). Since \(\sigma\) is past-directed timelike, this preimage is also an initial interval: if
$
\tau_0\in \sigma^{-1}\Big(I^+\big(\gamma(y_0^-),\mathcal{M}_{\mathrm{int}}\big)\Big),
$
then
$
[0,\tau_0]\subseteq \sigma^{-1}\Big(I^+\big(\gamma(y_0^-),\mathcal{M}_{\mathrm{int}}\big)\Big).
$
Finally, because \((\mathcal{M}_{\mathrm{int}},g_{\mathrm{int}})\) satisfies the chronology condition, the sets
$
I^-\big(\gamma(y_0^-),\mathcal{M}_{\mathrm{int}}\big)
\quad\text{and}\quad
I^+\big(\gamma(y_0^-),\mathcal{M}_{\mathrm{int}}\big)
$
are disjoint. Since \(\sigma(1)\in I^-(\gamma(y_0^-),\mathcal{M}_{\mathrm{int}})\), it follows that \(\Delta_-<1\), proving the claim.

By the same argument, there exists \(\Delta_+\in (0,1)\) such that
$
\sigma^{-1}\Big(I^+\big(\gamma(y_0^+),\mathcal{M}_{\mathrm{int}}\big)\Big)=[0,\Delta_+).
$
We now show that \(\Delta_+<\Delta_-\). Since \((\mathcal{M}_{\mathrm{int}},g_{\mathrm{int}})\) is globally hyperbolic and \(g_{\mathrm{int}}\) is smooth, we have
$\overline{I^+\big(\gamma(y_0^+),\mathcal{M}_{\mathrm{int}}\big)}
=
J^+\big(\gamma(y_0^+),\mathcal{M}_{\mathrm{int}}\big).$
Using moreover that
$
\gamma(y_0^+)\in I^+\big(\gamma(y_0^-),\mathcal{M}_{\mathrm{int}}\big),
$
we obtain
\[
\overline{I^+\big(\gamma(y_0^+),\mathcal{M}_{\mathrm{int}}\big)}
=
J^+\big(\gamma(y_0^+),\mathcal{M}_{\mathrm{int}}\big)
\subseteq
J^+\Big(I^+\big(\gamma(y_0^-),\mathcal{M}_{\mathrm{int}}\big),\mathcal{M}_{\mathrm{int}}\Big)
=
I^+\big(\gamma(y_0^-),\mathcal{M}_{\mathrm{int}}\big).
\]
In particular,
$$
\sigma(\Delta_+)\in \overline{I^+\big(\gamma(y_0^+),\mathcal{M}_{\mathrm{int}}\big)}
\subseteq
I^+\big(\gamma(y_0^-),\mathcal{M}_{\mathrm{int}}\big),
$$
and therefore \(\Delta_+<\Delta_-\). Choose now \(\tau_0\in (\Delta_+,\Delta_-)\). Then
$$
\sigma(\tau_0)\in I^+\big(\gamma(y_0^-),\mathcal{M}_{\mathrm{int}}\big)
\setminus
I^+\big(\gamma(y_0^+),\mathcal{M}_{\mathrm{int}}\big).
$$
On the other hand, since \(\sigma\) is past-directed timelike and \(\sigma(0)\in \gamma((y_0^+,0))\), it is immediate that
$$
\sigma(\tau_0)\in \bigcup_{-\varepsilon_0<\tau<0} I^-\big(\gamma(\tau),\mathcal{M}_{\mathrm{int}}\big).
$$
This implies that $\sigma\left(\tau_0\right)\in K$, and hence completes the proof.
\end{proof}

\begin{lemma}\label{Wtls}
 Assume that there exists a $C^0$-extension of $\mathcal{M}_{\textrm{int}}$ through $r\to0$.  Let $\gamma:[-\varepsilon_0,0)\to \mathcal{M}_{\textrm{int}}$ be a future-directed and future-inextendible timelike curve as described in the introduction of \Cref{sec4},  and consider the setup of the Minkowskian chart as described in  \Cref{thesetupmainproof}. Then, the set 
     $$
     W:= (\tilde{\varphi} \circ \iota)^{-1}\Big(\big[x^+ + C^-_{6/7}\big] \cap \big[x^- + C^+_{6/7}\big]\Big)\subseteq \mathcal{M}_{\textrm{int}}
     $$
 is an open neighborhood of $\bar{K}\left(y_0^+,y_0^-\right)\subseteq\mathcal{M}_{\textrm{int}}$.    
\end{lemma}


\begin{proof}
We first prove that 
$$
I^-\big(\bar{0}, \mathbb{R}_{\varepsilon_0,\varepsilon_1}\big) = \bigcup_{y_0^- < \tau < 0} \Big(I^-\big((\tau,0, \ldots, 0), \mathbb{R}_{\varepsilon_0,\varepsilon_1}\big) \Big).
$$ Here $\bar 0$ is the centre of the Minkowskian chart $\mathbb{R}_{\varepsilon_0,\varepsilon_1}$.  The inclusion ``$\supseteq$" is straightforward. For the other way round, let $p\in I^-\big(\bar 0, \mathbb{R}_{\varepsilon_0,\varepsilon_1}\big)$. This implies $\bar 0\in I^+\left(p,\mathbb{R}_{\varepsilon_0,\varepsilon_1}\right)$. Since $I^+(p,\mathbb{R}_{\varepsilon_0,\varepsilon_1})$ is open (see Proposition 2.6 in \cite{Sbierski_2018a}), there exists $\tau$ close to $0$ such that $\tilde\varphi \circ \tilde \gamma(\tau)\in I^+\left(p,\mathbb{R}_{\varepsilon_0,\varepsilon_1}\right)$. This implies $p\in I^-\big((\tau,0, \ldots, 0), \mathbb{R}_{\varepsilon_0,\varepsilon_1}\big)$, since $\tilde\varphi \circ \tilde \gamma(\tau)=(\tau,0, \ldots, 0)$, hence proving the inclusion ``$\subseteq$". We therefore obtain
    \begin{equation}
        I^-\big(\bar 0, \mathbb{R}_{\varepsilon_0,\varepsilon_1}\big) \cap I^+\big(y^-, \mathbb{R}_{\varepsilon_0,\varepsilon_1}\big) = \bigcup_{y_0^- < \tau < 0} \Big(I^-\big((\tau,0, \ldots, 0), \mathbb{R}_{\varepsilon_0,\varepsilon_1}\big) \Big)\cap I^+\big(y^-, \mathbb{R}_{\varepsilon_0,\varepsilon_1}\big).
    \end{equation}
Moreover, together with \Cref{causaldiamonds same}, in particular, \eqref{causaldiamondssameeq}, we have 
\begin{align}\label{Wopennbdeq2}
\begin{split}
 \tilde{\varphi}^{-1} \Big(I^-\big(\bar 0, \mathbb{R}_{\varepsilon_0,\varepsilon_1}\big) \cap I^+\big(y^-, \mathbb{R}_{\varepsilon_0,\varepsilon_1}\big) \Big) 
&= \tilde{\varphi}^{-1}\left(\bigcup_{y_0^{-}<\tau<0}\left(I^-\left(\left(\tau,0,\ldots,0\right), \mathbb{R}_{\varepsilon_0,\varepsilon_1}\right)\right)\cap I^+\left(y^-,\mathbb{R}_{\varepsilon_0,\varepsilon_1}\right)\right)\\
&= \bigcup_{y_0^{-}<\tau<0}\tilde{\varphi}^{-1}\left[I^-\left(\left(\tau,0,\ldots,0\right), \mathbb{R}_{\varepsilon_0,\varepsilon_1}\right)\cap I^+\left(y^-,\mathbb{R}_{\varepsilon_0,\varepsilon_1}\right)\right]\\
&=\bigcup_{y_0^{-}<\tau<0}\iota \left[I^-\left(\gamma(\tau),\mathcal{M}_{\textrm{int}}\right)\cap I^+\left(\gamma(y_0^-),\mathcal{M}_{\textrm{int}}\right)\right]\\
&=\iota \left[\left(\bigcup_{y_0^- < \tau < 0} I^-\left(\gamma(\tau), \mathcal{M}_{\textrm{int}} \right)\right) \cap I^+\big( \gamma(y^-_0), \mathcal{M}_{\textrm{int}}\big) \right]    
\end{split}
\end{align}
It follows from the definition of $K$, in Eq. \eqref{Kdefn}, that $K\left(y_0^+,y_0^-\right) \subseteq \Big(\bigcup_{y_0^- < \tau < 0} I^-\big(\gamma(\tau), \mathcal{M}_{\textrm{int}} \big)\Big) \cap I^+\big( \gamma(y^-_0), \mathcal{M}_{\textrm{int}}\big)$. We thus obtain from \eqref{Wopennbdeq2},
\begin{align}\label{K6/7}
\begin{split}
    \overline{\tilde{\varphi} \big( \iota\left(K\left(y_0^+,y_0^-\right)\right)\big)}  &\subseteq \overline{I^-\big(\bar 0, \mathbb{R}_{\varepsilon_0,\varepsilon_1}\big) \cap I^+\big(y^-, \mathbb{R}_{\varepsilon_0,\varepsilon_1}\big)}\\
    &\subseteq  \overline{C^{-}_{\frac{5}{8}}\cap I^+\left(y^-,\mathbb{R}_{\varepsilon_0,\varepsilon_1}\right)}\\
    &\subseteq \overline{C^{-}_{\frac{5}{8}}\cap \left[y^-+C^+_{\frac{5}{8}}\right]\cap \mathbb{R}_{\varepsilon_0,\varepsilon_1}}\\
    & \subseteq \left[x^++C^-_{\frac{6}{7}}\right]\cap\left[x^-+C^+_{\frac{6}{7}}\right]
\end{split}
\end{align}
In the second step above, we used the second expression of \eqref{diamondsinchart} for $x=\bar 0$ i.e., $I^-\big(\bar 0, \mathbb{R}_{\varepsilon_0,\varepsilon_1}\big)\subseteq C^{-}_{\frac{5}{8}}$, in the third step above, we used the first expression of \eqref{diamondsinchart} for $x=y^-$, i.e., $I^+\big(y^-, \mathbb{R}_{\varepsilon_0,\varepsilon_1}\big)\subseteq \left[y^-+C^+_{\frac{5}{8}}\right]$, and in the final step above, we used \eqref{minksetup3} which is implied by a suitable choice of $y^-\in \mathbb{R}_{\varepsilon_0,\varepsilon_1}$. Additionally, we have the continuity of $\tilde{\varphi}$ and $\iota$, and compactness of $\bar{K}\left(y_0^+,y_0^-\right)\subseteq \mathcal{M}_{\textrm{int}}$, which implies
$$
\tilde{\varphi}\big(\iota(\overline{K}\left(y_0^+,y_0^-\right)) \big) = \overline{\tilde{\varphi} \big( \iota(K\left(y_0^+,y_0^-\right))\big)} \;.
$$
Hence, \eqref{K6/7} then yields that
$$
W:= (\tilde{\varphi} \circ \iota)^{-1}\Big(\big[x^+ + C^-_{\frac{6}{7}}\big] \cap \big[x^- + C^+_{\frac{6}{7}}\big]\Big) \subseteq \mathcal{M}_{\textrm{int}}
$$
is an open neighbourhood of $\overline{K}\left(y_0^+,y_0^-\right) \subseteq \mathcal{M}_{\textrm{int}}$, thereby concluding the proof.
\end{proof}

\begin{lemma}\label{Wtls2}
 Assume that there exists a $C^0$-extension of $\mathcal{M}_{\textrm{int}}$ through $r\to0$. Let $\gamma:[-\varepsilon_0,0)\to \mathcal{M}_{\textrm{int}}$ be a future-directed and future-inextendible timelike curve as described in \Cref{sec4}, and consider the setup of the Minkowskian chart as described in  \Cref{thesetupmainproof}. Let $(\mathcal{F},h)$ be a codimension two compact homogeneous Riemannian manifold. Then, there exists a $\mu>0$ such that $W$ timelike separates $I^{+}(\gamma(-\mu),\mathcal{M}_{\textrm{int}})$ and $\gamma(x_0^-)$, where $W$ is as defined in \Cref{Wtls}.
\end{lemma}
\begin{figure}[h]\label{wtlsfig}
    \centering
\includegraphics[width=0.5\linewidth]{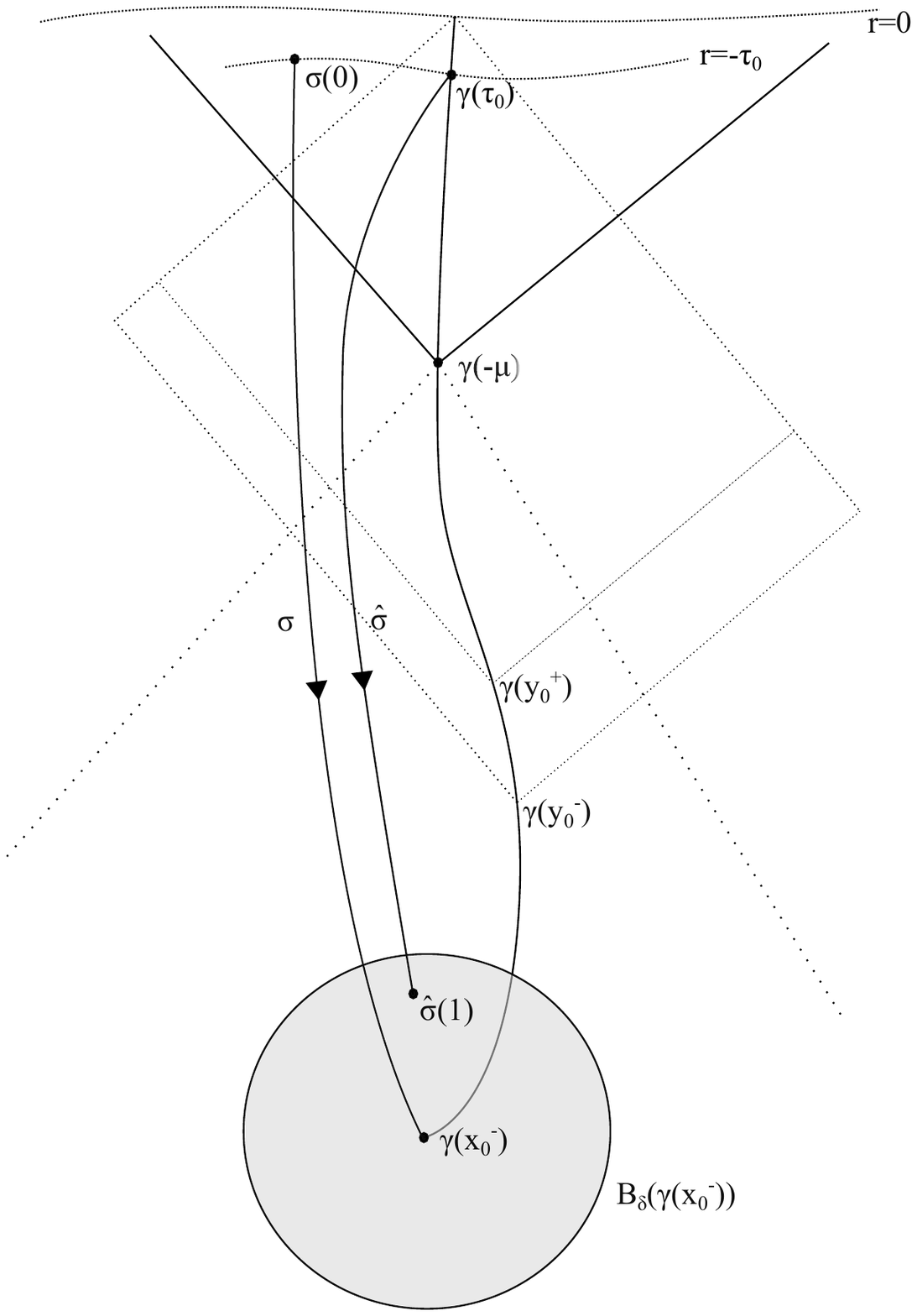}
    \caption{A pictorial representation of the proof depicting that $W$ timelike separates $I^{+}(\gamma(-\mu),\mathcal{M}_{\textrm{int}})$ and $\gamma(x_0^-)$, where $W$ is as defined in \Cref{Wtls}.}
\end{figure}
\begin{proof}
 Consider the metric $d_{\mathcal{M}_{\mathrm{int}}} : \mathcal{M}_{\mathrm{int}} \times \mathcal{M}_{\mathrm{int}} \to [0,\infty)$ defined by
\begin{equation}\label{distanceWtlslemma}
d_{\mathcal{M}_{\mathrm{int}}}\big((t_1,r_1,y_1),(t_2,r_2,y_2)\big)
:= |t_1-t_2| + |r_1-r_2| + d_{h}(y_1,y_2),
\end{equation}
for $(t_i,r_i,y_i)\in\mathcal{M}_{\mathrm{int}}$, $i=1,2$. The function
$$
\mathcal{M}_{\mathrm{int}} \ni (t,r,y)\mapsto d_{\mathcal{M}_{\mathrm{int}}}\big((t,r,y),\mathcal{M}_{\mathrm{int}}\setminus W\big)
:= \displaystyle \inf_{(t',r',y')\in \mathcal{M}_{\mathrm{int}}\setminus W}
d_{\mathcal{M}_{\mathrm{int}}}\big((t,r,y),(t',r',y')\big)
$$
is continuous. Since $\overline{K}\left(y_0^+,y_0^-\right)$ is compact (from \Cref{lemmakclcompact}) and disjoint from the closed set $\mathcal{M}_{\mathrm{int}}\setminus W$, this function attains its minimum on $\overline{K}\left(y_0^+,y_0^-\right)$, and that minimum is strictly positive. Therefore, there exists $\delta>0$ such that
$\overline{K}_\delta\left(y_0^+,y_0^-\right)
:= \left\{(t,r,y)\in\mathcal{M}_{\mathrm{int}}
\;\middle|\;
d_{\mathcal{M}_{\mathrm{int}}}\big((t,r,y),\overline{K}\big)<\delta
\right\}
\subseteq W.
$ 
By decreasing $\delta$ slightly if necessary, we may further assume that
\begin{equation}
\label{gammax0mdeltaball}
B_\delta\big(\gamma(x_0^-)\big)\subseteq I^-\big(\gamma(y_0^-),\mathcal{M}_{\mathrm{int}}\big),
\end{equation}
where $\delta$ is the radius with respect to the metric $d_{\mathcal{M}_{\textrm{int}}}$. 

 Let $G$ be a connected Lie subgroup of $\textrm{Isom}(\mathcal F,h)$ acting transitively on the compact manifold $\mathcal{F}$ (if $G$ is not connected, then consider that connected component which contains the identity element). Then, for $\delta>0$ as chosen above, \Cref{lemmaofisometry} says that there exists an $\epsilon(\delta)>0$ such that for $y_1,y_2\in \mathcal{F}$, 
\begin{equation}\label{deltaisometry}
    d_{h}\left(y_1,y_2\right)<\epsilon(\delta)\implies ~\exists~f\in G:~f(y_1)=y_2~ \textrm{and}~\sup_{y\in\mathcal{F}}d_{h}\left(fy,y\right)<\delta.
\end{equation}
From \Cref{corollary1} and using the triangle inequality, for this $\epsilon(\delta)>0$, there exists a $\mu (\epsilon(\delta) )\in(0,-y_0^+)$ such that for all $(t_1,r_1, y_1), (t_2,r_2, y_2) \in I^+\big(\gamma(-\mu),\mathcal{M}_{\textrm{int}}\big)$, we have
\begin{equation}
\label{ChoiceMu}
|t_1 - t_2| < \epsilon \qquad \textnormal{ and } \qquad d_{h}\big(y_1, y_2\big) < \epsilon \;.
\end{equation}
Once such a $\mu(\epsilon(\delta))>0$ is fixed, consider an arbitrary past-directed timelike curve $\sigma: [0,1]\to \mathcal{M}_{\textrm{int}}$ with $\sigma(0) \in I^+\big(\gamma(-\mu),\mathcal{M}_{\textrm{int}}\big)$ and $\sigma (1) = \gamma(x^-_0)$. Also let $\tau_0 \in (-\mu ,0)$ be such that $\gamma_r(\tau_0) = \sigma_r (0)$. From the second inequality of \eqref{ChoiceMu}, we have 
$d_{h}\big(\sigma_y(0), \gamma_y(\tau_0)\big) < \epsilon.$ Using \eqref{deltaisometry}, we have the existence of $f\in G$ such that $f(\sigma_y(0))=\gamma_y(\tau_0)$, and   $d_{h}\left(f\left(\sigma_y(1)\right),\sigma_y(1)\right)<\delta$. Now, we define a curve $\hat{\sigma} : [0,1] \to \mathcal{M}_{\textrm{int}}$, given by 
$$
\hat{\sigma}(\tau) = \Big(\sigma_t(\tau) + [\gamma_t(\tau_0) - \sigma_t(0)], \sigma_r(\tau), f\big(\sigma_y(\tau)\big)\Big).
$$
$\hat \sigma$ is past directed timelike with $\hat{\sigma}(0) = \gamma(\tau_0)$ and, using the second part of \eqref{ChoiceMu}, together with \eqref{gammax0mdeltaball}, we have  $\hat{\sigma}(1) \in I^-\big(\gamma(y^-_0),\mathcal{M}_{\textrm{int}}\big)$. Since  $K\left(y_0^+,y_0^-\right)$ (as defined in \Cref{lemmakclcompact}) timelike separates $\gamma\left(\left(y_0^+,0\right)\right)$ and $I^-\left(\gamma\left(y_0^-\right),\mathcal{M}_{\textrm{int}}\right)$ according to  \Cref{Ktls}, there exists an $\hat{\tau} \in [0,1]$ with $\hat{\sigma}(\hat{\tau}) \in K$. It now follows that $\sigma(\hat{\tau}) \in \overline{K}_\delta\left(y_0^+,y_0^-\right)\subset W$, thereby concluding the proof.
\end{proof}
Once we have all the Lemmas as above one can proceed to prove \Cref{statementA} and \Cref{statementB} similar to the arguments in the case of Schwarzschild spacetime. 

\begin{proof}[The proof of \Cref{statementA} and \Cref{statementB}]
    We first show that
$$
\iota \big(I^+(\gamma(-\mu), \mathcal{M}_{\textrm{int}})\big)
\subseteq
\tilde{\varphi}^{-1}\Big(\big[x^+ + C^-_{\frac{6}{7}}\big] \cap \big[x^- + C^+_{\frac{6}{7}}\big]\Big).
$$
Arguing by contradiction, let $\sigma : [0,1] \to \mathcal{M}_{\textrm{int}}$ be a future-directed timelike curve with $\sigma(0)=\gamma(-\mu)$, and suppose that there exists $\tilde{\tau}\in[0,1]$ such that
$$
(\tilde{\varphi}\circ\iota\circ\sigma)(\tilde{\tau})
\notin
\big[x^+ + C^-_{\frac{6}{7}}\big]\cap \big[x^- + C^+_{\frac{6}{7}}\big].
$$
Define
$$
\tau_0:=\sup\Big\{\tau'\in[0,1]\;\big|\;(\tilde{\varphi}\circ\iota\circ\sigma)(\tau)\in \big[x^+ + C^-_{\frac{6}{7}}\big]\cap \big[x^- + C^+_{\frac{6}{7}}\big]\ \text{for all }\tau\in[0,\tau')\Big\}.
$$
Then clearly $0<\tau_0\leq 1$, and our assumption implies that
$$
(\tilde{\varphi}\circ\iota\circ\sigma)(\tau_0)\in \partial\Big(\big[x^+ + C^-_{\frac{6}{7}}\big]\cap \big[x^- + C^+_{\frac{6}{7}}\big]\Big).
$$
Since every vector in $C^-_{\frac{5}{6}}$ is past-directed timelike, we can choose a past-directed timelike curve $\xi:[0,1]\to\mathbb{R}_{\epsilon_1,\epsilon_2}$ with 
$\xi(0)=(\tilde{\varphi}\circ\iota\circ\sigma)(\tau_0)$ and $\xi(1)=x^-$
such that $\xi$ does not intersect $\big[x^+ + C^-_{\frac{6}{7}}\big]\cap \big[x^- + C^+_{\frac{6}{7}}\big]$; for instance, $\xi$ may be chosen to lie in
$
\partial\Big(\big[x^+ + C^-_{\frac{6}{7}}\big]\cap \big[x^- + C^+_{\frac{6}{7}}\big]\Big).
$
By the first part of \Cref{minkremark}, the image of $\xi$ is contained in $\tilde{\varphi}\big(\iota(\mathcal{M}_{\textrm{int}})\cap \tilde{U}\big)$. Consequently,
$
(\tilde{\varphi}\circ\iota)^{-1}\circ\xi
$
is a past-directed timelike curve in $\mathcal{M}_{\textrm{int}}$ satisfying
$$
\big((\tilde{\varphi}\circ\iota)^{-1}\circ\xi\big)(0)=\sigma(\tau_0)\in I^+\big(\gamma(-\mu),\mathcal{M}_{\textrm{int}}\big) \quad\textrm{and}\quad \big((\tilde{\varphi}\circ\iota)^{-1}\circ\xi\big)(1)=\gamma(x^-_0),
$$
and whose image does not intersect
$
W=(\tilde{\varphi}\circ\iota)^{-1}\Big(\big[x^+ + C^-_{\frac{6}{7}}\big]\cap \big[x^- + C^+_{\frac{6}{7}}\big]\Big).
$
This contradicts \Cref{Wtls2}. We therefore conclude that
$$
\iota \big(I^+(\gamma(-\mu), \mathcal{M}_{\textrm{int}})\big)
\subseteq
\tilde{\varphi}^{-1}\Big(\big[x^+ + C^-_{\frac{6}{7}}\big]\cap \big[x^- + C^+_{\frac{6}{7}}\big]\Big),
$$
which in particular proves statement \Cref{statementA}. Finally, statement \Cref{statementB} follows from \eqref{BottomofMink} together with
$
y_0^-<y_0^+<-\mu<0.
$
\end{proof}

\subsection{The proof of \Cref{statementC}}\label{sec4.3}
The proof of \Cref{statementC} can be adopted as it is in the case of maximal analytic extension of Schwarzschild spacetime from \cite{Sbierski_2018}. Hence, we only sketch the main idea. We claim that there exists a constant $C_d>0$ such that, for every Cauchy hypersurface $\Sigma$ of $\mathcal{M}_{\mathrm{int}}$, the distance in $\Sigma$ between any two points of $I^+(\gamma(-\mu),\mathcal{M}_{\mathrm{int}})\cap\Sigma$ is bounded above by $C_d$.

\subsubsection*{Sketch of the proof} 

Assume first that $\gamma(-\mu)\in I^-(\Sigma,\mathcal{M}_{\mathrm{int}})$, since otherwise there is nothing to prove. For every $x\in(-\epsilon_1,\epsilon_1)^d$, the vertical curve
$
\sigma_x:(-\epsilon_0,f(x))\to\mathcal{M}_{\mathrm{int}}$ given by  $\sigma_x(\tau):=(\tilde \varphi\circ \iota)^{-1}(\tau,x),
$
is future-directed timelike, future-inextendible, and starts in $I^-(\Sigma,\mathcal{M}_{\mathrm{int}})$. Because $\Sigma$ is a Cauchy hypersurface, each such curve intersects $\Sigma$ exactly once, i.e. at $\tau=\hbar(\underline{x})$. 
This defines the function
$
\hbar:(-\epsilon_1,\epsilon_1)^d\to(-\epsilon_0,\epsilon_0).
$
Moreover, one can show that $\hbar$ is smooth. Indeed, $\tilde\varphi(\iota(\Sigma)\cap\tilde U)$ is a smooth hypersurface of $\mathbb{R}_{\varepsilon_0,\varepsilon_1}$, and the timelike vector field $\partial_0$ is nowhere tangent to it since $\Sigma$ is a Cauchy surface. Hence, the implicit function theorem applies, proving $\hbar$ is smooth. 

Next, using the uniform bounds on the metric coefficients $\tilde g_{ij}$ of the Minkowskian chart, it follows that the derivatives (their absolute values) $\vert\partial_i \hbar\vert$ are uniformly bounded $\forall~\underline{x}\in(-\epsilon_1,\epsilon_1)^d$ (uniform because the bound does not depend on $\hbar$).  Now define $\omega: (-\epsilon_0,\epsilon_1)^d\to \mathbb{R}_{\varepsilon_0,\varepsilon_1}$ by $\omega(x):=(\hbar(x),x)$. This map parametrises a smooth submanifold $\tilde S$ of $\mathbb{R}_{\varepsilon_0,\varepsilon_1}$The ambient metric $\tilde g$ on $\mathbb{R}_{\varepsilon_0,\varepsilon_1}$ induces a metric $\bar g$ on $\tilde S$. Additionally,  via $\left(\tilde \varphi \circ \iota \right)^{-1}$, the set $\tilde S$ is isometric to an open subset $S$ of $\Sigma$. Consequently,  using the uniform bounds on the metric coefficients $\tilde g_{ij}$ in the Minkowskian chart along with the uniform bound on the uniform boundedness of $\vert\partial_i \hbar\vert$, the coefficients of the induced metric on $S$, i.e. $\bar g_{ij}(\underline{x})$ are also bounded uniformly, independently of the choice of $\Sigma$. 

Now let $p,q\in I^+(\gamma(-\mu),\mathcal{M}_{\mathrm{int}})\cap\Sigma$ be arbitrary points. By  \Cref{statementA}, there exist $x,y\in(-\epsilon_1,\epsilon_1)^d$ such that $\omega(x)=\tilde\varphi(\iota(p))$ and $\omega(y)=\tilde\varphi(\iota(q))$. We connect $x$ and $y$ by the Euclidean straight line segment $\sigma:[0,1]\to (-\epsilon_1,\epsilon_1)^d$, given by $\sigma(s)=\underline{x}+s\left(\underline{y}-\underline{x}\right)$. One can show that the length $L_{\bar g}\left(\omega\circ\sigma\right)$ is bounded from above, and is independent of the points $\underline{x}$ and $\underline{y}$,  and even the Cauchy surface $\Sigma$. Additionally, it follows that  $\left(\tilde \phi \circ \iota\right)^{-1}\circ \omega\circ \sigma$ is a smooth curve in $\Sigma$ joining $p$ and $q$ whose length is bounded by the same bound as that of $L_{\bar g}\left(\omega\circ\sigma\right)$. Therefore, there exists a constant $C_d$ that uniformly bounds the distance (from above) between any two points, and on any chosen Cauchy surface.

\section{Proof of  \Cref{eitheror} -- The timelike geodesic dichotomy}\label{sec5}

In Theorem A2 of \cite{Sbierski_2021}, Sbierski established a general criterion ensuring timelike geodesic completeness of the exterior region of a spherically symmetric black hole spacetime. More precisely, he gave a sufficient condition under which every future-directed timelike geodesic that is future-inextendible and remains entirely within the exterior region $\mathcal{M}_{\textrm{ext}}$ is future-complete. Here, first we reproduce an analogous statement for a class of spacetimes--see \eqref{spacetimetheorem2}--that generalises the warped-product black hole spacetimes of  \Cref{sec2.2}, in \Cref{eitherorprop1}. We then combine this result with \Cref{eitherorprop2}, which shows that future incompleteness of an affinely parametrised, future-directed, future-inextendible timelike geodesic can occur only by approaching the singularity. Together, these two propositions yield \Cref{eitheror}. We note that the class of spacetimes in \Cref{eitherorprop1} is more general than what is required for our purpose.

To begin with, let $(\mathcal{F},h)$ be a complete smooth Riemannian manifold (implied if it is compact, and without boundary), and let $(\mathcal{M},g)$ be a $(d+1)$-dimensional ($d\geq3$) spacetime such that 
\begin{equation}
    \begin{split}\label{spacetimetheorem2}
    \mathcal{M} &=[0,\infty)\times(0,\infty)\times \mathcal{F},\\
    g &=-f(v,r)\,dv^2+\frac{l(v,r)}{2}(dv\otimes dr+dr\otimes dv)+r^2h.
\end{split}
\end{equation}
We write points of $\mathcal{M}$ as $(v,r,y)$, where $v\in[0,\infty)$ and $r\in(0,\infty)$ are the standard $(v,r)$ coordinates, and $y\in \mathcal{F}$. Here, $f,l:[0,\infty)\times(0,\infty)\to\mathbb R$ are smooth. Fix the time orientation by declaring $-\partial_r$ to be future-directed null. Let
$
r_{H}:[0,\infty)\to(c_{H},\infty)
$
be a continuous function, where $c_{H}>0$. Define the black hole horizon hypersurface
$
\mathcal{H}:=\{(v,r,y)\in\mathcal M:r=r_{H}(v)\},
$
and the exterior region
$
\mathcal{M}_{\mathrm{ext}}:=\{(v,r,y)\in \mathcal{M}:r>r_{H}(v)\}.
$
The setup in \eqref{spacetimetheorem2} describes a class of spacetimes that generalises the class introduced in \Cref{sec2.2}. Indeed, the metric tensor in \eqref{spacetimetheorem2} reduces to the metric in \eqref{sssstefcoordinate}, which is the metric \eqref{sssst} written in Eddington--Finkelstein coordinates, when
$$
f(v,r)
=
\exp{\left(-2\psi(r)\right)}
\left(1-\frac{2m(r)}{r}\right),
\qquad
l(v,r)=2\exp{\left(-\psi(r)\right)} .
$$

In Theorem A2 of \cite{Sbierski_2021}, 
Sbierski exploits the spherical symmetry of the spacetimes there, and the argument is reduced to the equatorial plane $\theta=\pi/2$, where $\theta$ denotes the usual polar coordinate.  A key ingredient is the existence of a first integral along timelike geodesics, arising from the axisymmetric Killing vector field $\partial_\phi$, with $\phi$ the standard azimuthal coordinate.

In the absence of spherical symmetry, such a conserved quantity generally no longer exists. Nevertheless, we show here that the warped product structure of the spacetime gives rise to a quantity that is constant along geodesics, and is associated to the metric tensor $h$ of the fibre $\mathcal{F}$. This is stated in the following lemma.

\begin{lemma}\label{conservedquantityforh}
Let
$
\rho(s) = (\rho_v(s),\rho_r(s),\rho_y(s))
$
be an affinely parametrised geodesic for the metric in the setup described in \eqref{spacetimetheorem2}. Then
$
K := \rho_r^4 h(Y,Y)
$
is a non-negative finite-valued constant along $\rho$, where $Y\in \Gamma(T\mathcal{F})$ is the projection of $\dot \rho$ on $\mathcal{F}$.
\end{lemma}
\begin{proof}
Write $\dot\rho = U + Y$, where $U:=\dot \rho_v\,\partial_v + \dot\rho_r\,\partial_r$, and $Y(s)\in T_{\rho_y(s)}\mathcal{F}$. We have 
$$
\nabla_{\dot \rho}\dot \rho=\nabla_{\left(U+Y\right)}\left(U+Y\right)=\nabla_UU+\nabla_YU+\nabla_UY+\nabla_YY.
$$
Since the metric splits as a $(v,r)$-part plus the warped
fibre part $r^2 h$, the projection of $\nabla_{\dot \rho}\dot \rho$ on $\Gamma\left(T\mathcal{F}\right)$ is (see for e.g. \S 7, Proposition 35 in \cite{oneill})
$$
\bigl(\nabla_{\dot\rho}\dot\rho\bigr)^{\mathcal{F}}
=
\nabla^{\mathcal{F}}_{Y} Y + 2\frac{\dot \rho_r}{\rho_r}Y.
$$
Because $\rho$ is a geodesic in $\mathcal{M}$, $\nabla_{\dot\rho}\dot\rho=0$, hence its fibre component also satisfies
$$
\nabla^{\mathcal{F}}_Y Y + 2\frac{\dot \rho_r}{\rho_r}Y = 0.
$$
Next consider
$$
\frac{d}{ds}h(Y,Y)
=
2h(\nabla^{\mathcal{F}}_Y Y,Y)
=
-4\frac{\dot \rho_r}{\rho_r}h(Y,Y),
$$
where we used metric compatibility of the Levi-Civita connection of $h$. Therefore,
\begin{align*}
\frac{d}{ds}\bigl(\rho_r^4 h(Y,Y)\bigr)
&=
4\rho_r^3\dot \rho_r\, h(Y,Y)
+
\rho_r^4 \frac{d}{ds}h(Y,Y) \\
&=
4\rho_r^3\dot \rho_r\, h(Y,Y)
+
\rho_r^4\left(-4\frac{\dot \rho_r}{\rho_r}h(Y,Y)\right) \\
&= 0.
\end{align*}
Hence $\rho_r^4 h(Y,Y)$ is non-negative and finite-valued constant along $\rho$.
\end{proof}

We are now in a position to prove the statement analogous to Theorem A2 in \cite{Sbierski_2021}, but without assuming spherical symmetry. The only structural assumption required is that the spacetime admits the warped product form \eqref{spacetimetheorem2}, together with the additional assumption that the fibre $(\mathcal{F},h)$ is complete, besides some constraints on the metric tensor coefficients $f$ and $l$.

\begin{proposition}\label{eitherorprop1}
Let $(\mathcal{M},g)$ be as described in \eqref{spacetimetheorem2}. Assume on $\mathcal{M}_{\mathrm{ext}}$ that
\begin{equation}\label{exttheorembounds}
    |f|\le C_f,\qquad 0<c_l\le l\le C_l,\qquad 0\le \partial_r l,\qquad
f.\partial_r l-l.\partial_r f\le 0.
\end{equation}
Let $\rho:[0,b)\to \mathcal{M}$ be an affinely parametrised, future-directed, and future-inextendible timelike geodesic with $\rho(0)\in \mathcal{M}_{\mathrm{ext}}$, where $b\in(0,\infty]$. Then
\begin{enumerate}[(i)]
\item $\rho$ is future complete, or
\item there exists $s_0\in[0,b)$ such that $\rho(s_0)\in \mathcal{H}$.
\end{enumerate}
\end{proposition}

\begin{proof}
Contrary to the statement of the theorem, assume that $\rho$ is entirely contained in $\mathcal{M}_{\mathrm{ext}}$ and that $\rho$ is not future complete, i.e., $b<\infty$. Write
\[
\rho(s)=(\rho_v(s),\rho_r(s),\rho_y(s)),
\]
where $\rho_y(s)\in \mathcal{F}$.  Assuming affine reparametrisation without loss of generality gives
$$
-1=-f\dot \rho_v^{\,2}+l\dot \rho_v\,\dot \rho_r+\rho_r^2h(Y,Y).
$$
We have that $K:=\rho_r^4\,h(Y,Y)$ is constant along $\rho$ from \Cref{conservedquantityforh}. Rewriting the above equation gives 
$$
-1=-f\dot \rho_v^{\,2}+l\dot \rho_v\,\dot \rho_r+\frac{K}{\rho_r^2}.
$$
Thus
\begin{equation}\label{vdotrdot}
    \dot \rho_v\,\dot \rho_r=\frac{-1+f\dot \rho_v^{2}-K/\rho_r^2}{l}.
\end{equation}
We now prove that each component on the geodesic $\rho$ and its derivative $\dot\rho$ is bounded on $[0,b)$.
\newline

\underline{$\rho_v$ and $\dot \rho_v$ are bounded in  $[0,b)$}
\newline

Let $A,B$ denote indices on $\mathcal{F}$. The only Christoffel symbols with upper index $v$  are
$$
\Gamma^v_{AB}=-\frac{2\rho_r}{l}h_{AB}, \quad\Gamma^v_{vv}=\frac{\partial_r f+\partial_vl}{l}.
$$
Therefore, the $v$-component of the geodesic equation becomes
$$
0=\ddot \rho_v+\frac{\partial_v l+\partial_r f}{l}\dot \rho_v^{2}
-\frac{2\rho_r}{l}h(Y,Y),
$$
which can be rewritten in terms of $K$ as (see \Cref{conservedquantityforh})
\begin{equation}\label{vcomponentofgeodesicequation}
\ddot \rho_v+\frac{\partial_v l+\partial_r f}{l}\dot \rho_v^{2}
-\frac{2K}{l\rho_r^3}=0.
\end{equation}
We use 
$
\frac{d}{ds}\log l(\rho_v,\rho_r)=\dot \rho_v\,\partial_v\log l+\dot \rho_r\,\partial_r\log l.
$ Multiplying by $\dot \rho_v$, we obtain
$$
\dot \rho_v\,\frac{d}{ds}\log l
=
\frac{\partial_v l}{l}\,\dot \rho_v^{\,2}
+
\frac{\partial_r l}{l}\,\dot \rho_v\,\dot \rho_r.
$$
Using \Cref{vdotrdot} into the previous identity yields
$$
\dot \rho_v\,\frac{d}{ds}\log l
=
\frac{\partial_v l}{l}\,\dot \rho_v^{\,2}
+
\frac{\partial_r l}{l^2}
\left(
-1+f\,\dot \rho_v^{\,2}-\frac{K}{\rho_r^2}
\right).
$$
Rewriting for $\frac{\partial_v l}{l}\,\dot \rho_v^{\,2}$, we obtain
$$
\frac{\partial_v l}{l}\,\dot \rho_v^{\,2}= \dot \rho_v\,\frac{d}{ds}\log l+\frac{\partial_r l}{l^2}-\frac{\partial_r l}{l^2}
f\,\dot \rho_v^{\,2}+\frac{\partial_r l}{l^2}\frac{K}{\rho_r^2}.
$$
Substituting this expression for $\frac{\partial_v l}{l}\,\dot \rho_v^{\,2}$ in the $v$-component of the geodesic \Cref{vcomponentofgeodesicequation}, we obtain
$$
\ddot \rho_v
=
\frac{2K}{l\rho_r^3}
-\dot \rho_v\,\frac{d}{ds}\log l
-\frac{1}{l^2}\Bigl(1+\frac{K}{\rho_r^2}\Bigr)\partial_r l
+\frac{1}{l^2}\bigl(f\,\partial_r l-l\partial_r f\,\bigr)\dot \rho_v^{\,2}.
$$
By the assumptions $0\le \partial_r l$ and
$
f\,\partial_r l-l\partial_r f\le 0
$
from \eqref{exttheorembounds}, and from the fact that $K\geq0$, it follows that
\begin{equation}\label{ddotrhovub}
    \ddot \rho_v\le \frac{2K}{l\rho_r^3}-\dot \rho_v\,\frac{d}{ds}\log l.
\end{equation}
Since $l\ge c_l>0$ (again from \eqref{exttheorembounds}), and $\rho_r\ge c_{H}>0$ on $M_{\mathrm{ext}}$, we infer that
\[
\ddot \rho_v\le C-\dot \rho_v\,\frac{d}{ds}\log l
\]
for some constant $C>0$. Multiplying by $l$ and using $l\le C_l$, we obtain
\[
\frac{d}{ds}(l\dot \rho_v)=l \ddot \rho_v +\dot \rho_v\frac{d}{ds}l=l \ddot \rho_v +\dot \rho_vl\frac{d}{ds}\log l\leq lC\leq C_lC= C'
\]
for some constant $C'>0$. Integrating over $[0,s]$ gives
$$
l(s)\dot \rho_v(s)
\le
l(0)\dot \rho_v(0)+C' s
\le
l(0)\dot \rho_v(0)+C' b.
$$
Using again $l\ge c_l>0$ (from \eqref{exttheorembounds}), we conclude that
$$
\dot \rho_v(s)
\le
\frac{l(0)\dot \rho_v(0)+C' b}{c_l},
$$
so $\dot \rho_v$ is uniformly bounded from above on $[0,b)$. Since $\rho$ is future-directed timelike and $-\partial_r$ is future-directed null, we have $g\left(-\partial_r,\dot \rho\right)=-\frac{l}{2}\dot \rho_v(s)<0$. Since $l>0$, we have $\dot \rho_v(s)>0$. Therefore, the function $\rho_v$ is increasing, and hence, $\lim_{s\to b} \rho_v(s)\in (0,\infty)$ exists. Additionally
\begin{equation}\label{boundednessofrhodotv}
    \sup_{s\in[0,b)}\vert\dot \rho_v\vert<\infty.
\end{equation}

\underline{$\rho_r$ and $\dot \rho_r$ are bounded in $[0,b)$}
\newline

From the preceding part of the proof, there exists a constant
$C_v>0$ such that
$
0<\dot{\rho}_v(s)\leq C_v$ for all $s\in[0,b)$.
Recall from \Cref{vdotrdot},
since $K\geq0$, $l>0$, and $\dot{\rho}_v>0$, we obtain
$
\dot{\rho}_r
\leq
(f/l)
\dot{\rho}_v.
$
Using the bounds in \eqref{exttheorembounds}, it follows that
$$
\dot{\rho}_r(s)
\leq
\frac{C_fC_v}{c_l}
=:C_r^{+}
\qquad \text{for all } s\in[0,b).
$$
Consequently, for every $s\in[0,b)$,
$$
\rho_r(s)-\rho_r(0)
=
\int_0^s\dot{\rho}_r(\sigma)\,d\sigma
\leq
\int_0^s C_r^{+}\,d\sigma
=
C_r^{+}s
\leq
C_r^{+}b.
$$
This gives an upper bound on $\rho_r$ in $[0,b)$. Moreover, $\rho_r>r_H$, we have that 
\begin{equation}\label{rhorbound}
    \sup_{s\in[0,b)}\vert\rho_r\vert<\infty,
\end{equation}

We now show that $\dot\rho_v$ is bounded away from zero. On the contrary, assume that $\dot\rho_v$ is not bounded away from zero, then there exists a sequence $[0,b)\ni s_n\to b$ for $n\to\infty$ with
$\lim_{n\to\infty}\dot \rho_v(s_n)=0$ (we do know that $\dot \rho_v(s)>0~\forall~s\in[0,b)$). Hence, after restricting to a sufficiently small final interval $[s_0,b)$, the curve $(\rho_v,\rho_r)$ is contained in a compact subset of the $(v,r)$-coordinate domain (here we use the result that $\rho_r$ is bounded, see \eqref{rhorbound}). In particular, the coefficient
$
\left(\partial_v l+\partial_r f\right)/l
$
in \Cref{vcomponentofgeodesicequation} is bounded in this compact subset, and the coefficient $2/\left(l\rho_r^3\right)$ is bounded below by a strictly positive constant.  Suppose first that $K>0$. Since $\dot{\rho}_v(s)\to0$ by our assumption, \Cref{vcomponentofgeodesicequation} implies, after increasing $s_0$ if necessary, that
$$
\ddot{\rho}_v(s)\geq c>0
\qquad\text{for all }s\in[s_0,b)
$$
for some constant $c>0$. Thus $\ddot{\rho}_v$ is strictly increasing on $[s_0,b)$, and hence $\dot{\rho}_v(s)\not\to0$ (since $\dot\rho_v(s)>0$). It remains to consider $K=0$. In this case, equation \eqref{vcomponentofgeodesicequation} gives
$$
\frac{d}{ds}\left(\frac{1}{\dot{\rho}_v}\right)=
\frac{\partial_v l+\partial_r f}{l}
$$
The right-hand side is bounded on $[s_0,b)$. Since $b<\infty$, integrating both sides gives that $1/\dot{\rho}_v$ remains bounded on $[s_0,b)$. Therefore $\dot{\rho}_v$ is bounded away from zero in this case too.

Once we establish that $\dot\rho_v$ is bounded away from zero, then there exists $c>0$ with $0<c\le \dot\rho_v\le C$ (refer to \Cref{boundednessofrhodotv}).
We obtain from \Cref{vdotrdot} that
\begin{equation}\label{eqndotvlb}
\dot \rho_r
=
\frac{-1 + f\dot \rho_v^{\,2} - \frac{K}{\rho_r^2}}{\dot \rho_v\,l}
\end{equation}
which, together with $\rho_r>r_{H}>0$, and \eqref{exttheorembounds}, gives 
\begin{equation}\label{boundednessofrhodotr}
    \sup_{s\in[0,b)]}|\dot \rho_r|< \infty.
\end{equation}

\underline{$\rho_y$ and $h(Y,Y)$ are bounded in $[0,b)$.}
\newline

Finally,
$$
h(Y,Y)=\frac{K}{r^4}\le \frac{K}{r_{H}^4}.
$$
Thus $\rho_y(s)$ has uniformly bounded speed, so its $h$-length on $[0,b)$ is finite:
$$
L_h(\rho_y|_{[0,b)})\le \frac{\sqrt K}{c_{H}^2}\,b<\infty.
$$

Since all the components of $\rho(s)$ and $\dot \rho$ are bounded in $[0,b)$, we have  $(\rho(s),\dot\rho(s))\vert_{s\in[0,b)}\subset T\mathcal{M}_{\mathrm{int}}$ is contained in a compact set.  But for smooth manifold $T\mathcal{M}_{\mathrm{int}}$, the image of the curve $(\rho,\dot \rho): [0,b)\to T\mathcal{M}_{\mathrm{int}}$, where $b<\infty$, lies in a compact set of $T\mathcal{M}_{\mathrm{int}}$ if and only if it is extendible (refer to Lemma 56, Chap 1 of O'Neill \cite{oneill}). Since we know that  $(\rho,\dot \rho)$ is inextendible, our assumption  that $\rho$ is entirely contained in $\mathcal{M}_{\mathrm{ext}}$ and that $\rho$ is not future complete, is false. 
\end{proof}

\begin{remark}
    The assumptions in \eqref{exttheorembounds} impose the following restrictions on the functions
$m(r)$ and $\psi(r)$ in the exterior region $r>r_H$ as sufficient conditions for the statement of the  \Cref{eitherorprop1} to hold.
\begin{equation*}\label{assumptonforeitheror}
 \begin{split}
&\sup_{r>r_H}\exp{\left(-2\psi(r)\right)\left(1-\frac{2m(r)}{r}\right)}<\infty,\quad \sup_{r>r_H}\vert\psi(r)\vert<\infty,
    \quad \textrm{and}\quad \\
&\psi'(r)\leq \min\left\{ 0, \frac{2}{r} \left(\frac{m(r)-rm'(r)}{r-2m(r)}\right)\right\}. 
 \end{split}
\end{equation*}
\end{remark} 

\Cref{eitherorprop1} shows that, under suitable restrictions on the metric coefficients, any future-directed timelike geodesic starting in the exterior region either is future complete or intersects the horizon. To complete the proof of \Cref{eitheror}, it remains to analyse the behaviour of such a geodesic after it crosses the horizon.

In the Schwarzschild case, O'Neill's result, Proposition $36$ in Chapter 13 of \cite{oneill}, states that every future-directed, future-inextendible timelike geodesic either is future complete or satisfies $r\to 0$. We will prove an analogous statement for a subclass of the warped-product black-hole spacetimes introduced in \Cref{sec2.2}, in \Cref{eitheror}. Before doing so, we first establish the following proposition. 

\begin{proposition}\label{eitherorprop2}
Let $(\mathcal{M},g)$ be a warped-product black-hole spacetime as described by \eqref{sssst} in  \Cref{sec2.2}, and let $\mathcal{F}$ be compact. Let $\rho:[0,b)\to \mathcal{M}_{\mathrm{int}}$, where $\rho(s)=\left(\rho_t(s),\rho_r(s),\rho_y(s)\right)$, be an affinely parametrised, future-directed, future-inextendible timelike geodesic.  Then
$$
\rho \text{ is future incomplete}
\quad\Longleftrightarrow\quad
r_*:=\lim_{s\to b} \rho_r(s)=0.
$$
\end{proposition}
\begin{proof}
The time orientation in $\mathcal{M}_{\mathrm{int}}$ is chosen such that $-\partial_r$ is future-directed. Since $\rho$ is future-directed and timelike, $ g(\dot\rho,-\partial_r)=-\left(1-\frac{2m(\rho_r)}{\rho_r}\right)^{-1}\dot \rho_r<0$, and hence $\dot\rho_r<0$. Thus $\rho_r$ is monotone decreasing along $\rho$, and we know that $\rho_r$ is bounded from below. Hence, the limit $r_*$ exists. It satisfies $0\le r_*<r_H$. We have three first integrals, corresponding respectively to the stationary Killing vector field $\partial_t$, the  Riemannian metric tensor $h$ of the Riemannian fibre, and the trivial Killing tensor field $g$. These are given by
\begin{equation}\label{conservedquantityinterior}
    \begin{split}
    E &:=-g(\dot\rho,\partial_t)= e^{-2\psi(\rho_r)}\left(1-\frac{2m(\rho_r)}{\rho_r}\right)\dot \rho_t,\\
    K &:=\rho_r^4 h(Y,Y), \quad \mathrm{and}\\
    -\kappa &:=g(\dot \rho,\dot\rho),\quad (\kappa>0),
\end{split}
\end{equation}
respectively. The second one is as described in \Cref{conservedquantityforh}.  We first prove that if $r_*>0$, then $\rho$ is future complete. From the first and second of the equations \eqref{conservedquantityinterior}, we obtain that $\vert\dot \rho_t\vert~\mathrm{and}~h(Y,Y)$ are both bounded for all $s\in[0,b)$.  Since $\mathcal{F}$ is compact, the image of $\rho_y([0,b))$ is contained in a compact subset of $\mathcal{F}$. The last of the equations in \eqref{conservedquantityinterior}  gives
\begin{equation}\label{kappa}
    -\kappa=-\exp\left(-2\psi(\rho_r)\right)\left(1-\frac{2m(\rho_r)}{\rho_r}\right)\dot \rho_t^2+ \left(1-\frac{2m(\rho_r)}{\rho_r}\right)^{-1}\dot \rho_r^2+\rho_r^2h(Y,Y).
\end{equation}
Hence,
$$
\frac{\dot \rho_r^2}{\left(\frac{2m(\rho_r)}{\rho_r}-1\right)}=\kappa-\dot \rho_t E + \frac{K}{\rho_r^2}.
$$
On the RHS of the equation above, we know that $\kappa$, $K$, and $E$ are finite. Additionally, $\vert \dot \rho_t\vert$ is bounded from above $\forall~s\in[0,b)$ (from the previous argument), and $\rho_r$ has a positive lower bound (since $\rho_r$ is monotone decreasing and $r_*$ is assumed to be positive). Therefore, $\vert \dot\rho_r\vert$ is bounded $\forall~s\in[0,b)$. Hence $(\rho(s),\dot\rho(s))\vert_{s\in[0,b)}\subset T\mathcal{M}_{\mathrm{int}}$ is compact.  But for smooth manifold $T\mathcal{M}_{\mathrm{int}}$, the image of the curve $(\rho,\dot \rho): [0,b)\to T\mathcal{M}_{\mathrm{int}}$, where $b<\infty$, lies in a compact set of $T\mathcal{M}_{\mathrm{int}}$ if and only if it is extendible, unless $b=\infty$ (refer to Lemma 56, Chap 1 of O'Neill \cite{oneill}). Since we know that  $(\rho,\dot \rho)$ is inextendible, we have that $b=\infty$, i.e., $\rho$ is future-complete.

Conversely, suppose that $r_*=0$. From the equations in \eqref{conservedquantityinterior} we have,
\begin{equation*}
     \frac{\dot \rho_r^2}{\left(\frac{2m(\rho_r)}{\rho_r}-1\right)} = \kappa
+
\frac{E^2}{\exp\left(-2\psi\left(\rho_r\right)\right)\left(\frac{2m(\rho_r)}{\rho_r}-1\right)}
+
\frac{K}{\rho_r^2}.
\end{equation*}
Since the second and the third term in the RHS of the above equation are non-negative, we have 
\begin{equation*}
\frac{\dot \rho_r^2}{\left(\frac{2m(\rho_r)}{\rho_r}-1\right)} \ge \kappa, \quad \textrm{and hence}\quad 
\dot \rho^2_r \geq \kappa \left(\frac{2m(\rho_r)}{\rho_r}-1\right).
\end{equation*}
We therefore have 
$$
-\dot \rho_r \geq \sqrt{\left(\frac{2m(\rho_r)}{\rho_r}-1\right)\kappa}\geq  \sqrt{\left(\frac{2m_{\mathrm{min}}}{\rho_r}-1\right)\kappa}.
$$
for all $s\in(s_0,b)$ (the parameter $s_0$ is large enough such that $2m_{\textrm{min}}>\rho_r(s)$ for all $s\in(s_0,b)$). On integrating the first and the last term in this chain of inequality, we obtain
$$
-\int^{0}_{\rho_r(s_0)}\frac{\sqrt{\rho_r}}{\sqrt{\left(2m_{\mathrm{min}}-\rho_r\right)\kappa}}d\rho_r\geq \int^{b}_{s_0}ds,
$$
Since the integral on the left-hand side is finite, we have
$
b-s_0<\infty.
$
Thus $b<\infty$, and $\rho$ is future incomplete.
\end{proof}
Combining \Cref{eitherorprop1} and \Cref{eitherorprop2}, we get the statement in  \Cref{eitheror} as follows. 
\begin{proof}[Proof of \Cref{eitheror}]
    The \Cref{eitherorprop1} states that either $\rho$ is future complete, or there exists $s_0 \in [0,b)$ such that $\rho(s_0) \in \mathcal{H}$. A priori, these two alternatives are not mutually exclusive; both may hold simultaneously. In the second case, since $\rho$ is timelike whereas $\mathcal{H}$ is a null hypersurface, $\rho$ cannot be entirely contained in $\mathcal{H}$. Indeed, tangent vectors to a null hypersurface are either null or spacelike. Hence, once $\rho$ meets $\mathcal{H}$, it must enter $\mathcal{M}_{\mathrm{int}}$.

 \Cref{eitherorprop2} further states that, if $\rho$ is contained in $\mathcal{M}_{\mathrm{int}}$, then $\rho$ is future incomplete if and only if $\rho_r \to 0$ as $s \to b$. Therefore, \Cref{eitherorprop1} leads to three possible scenarios.
    \begin{enumerate}
        \item Let $\rho$ be future complete and there exists $s_0\in[0,b)$ such that $\rho(s_0)\in \mathcal{H}$. Since $\rho\not\subset \mathcal{H}$, we have $\rho$ enters $\mathcal{M}_{\textrm{int}}$.  Now consider that part of $\rho$ that is contained in $\mathcal{M}_{\textrm{int}}$. From  \Cref{eitherorprop2}, we can then imply that $\rho_r(s)\not\to 0$ as $s\to b$. 
        \item Let $\rho$ be future incomplete and there exists $s_0\in[0,b)$ such that $\rho(s_0)\in \mathcal{H}$. Just like the previous case, $\rho$ enters $\mathcal{M}_{\textrm{int}}$. From \Cref{eitherorprop2}, we can then imply that $\rho_r(s)\to 0$ as $s\to b$.
        \item Let $\rho$ be future complete and there does not exists $s_0\in[0,b)$ such that $\rho(s_0)\in \mathcal{H}$. Then $\rho\subset \mathcal{M}_{\textrm{ext}}$, and hence $\rho_r(s)\not \to0$ as $s\to b$.
    \end{enumerate}
\end{proof}

\section{Proof of \Cref{maintheorem2} -- Main $C^0$-inextendibility result}\label{sec6}

We use \Cref{maintheorem} and \Cref{eitheror}, together with Sbierski’s argument in \cite{Sbierski_2018}, to prove this theorem. Suppose, for contradiction, that the spacetime described by \eqref{sssst} in \Cref{sec2.2}, satisfying the assumptions of \Cref{maintheorem} and \Cref{eitheror}, admits a $C^0$-extension. Then  \Cref{gls2}, namely Theorem 2 of Galloway--Ling--Sbierski \cite{Galloway_2018}, implies that there exists a future-directed, future-inextendible timelike geodesic $\rho:[0,b)\to\mathcal{M}$, where $b\in(0,\infty]$, with a future endpoint in $\partial^+\iota(\mathcal{M})$. By  \Cref{eitheror}, such a timelike geodesic either remains in the exterior region $\{r>r_H\}$, in which case it is future complete ($b=\infty$), or approaches the central curvature singularity $r=0$ ($b<\infty$). The first possibility is excluded by  \Cref{thmobstruction}, which states that timelike geodesic completeness prevents a $C^0$-extension through the exterior. The second possibility is excluded by  \Cref{maintheorem}. Both alternatives therefore lead to a contradiction, and hence no such $C^0$-extension can exist.

\section{Technical distinctions from the Schwarzschild case}\label{sec7}

We briefly explain the two main origins of technical differences between the present (Sbierski's adapted) proof for a class of general warped product black hole spacetimes and Sbierski’s original proof for the maximal analytic Schwarzschild spacetime in \cite{Sbierski_2018}.

The first origin is the generalisation of the two-dimensional base part of the Schwarzschild metric. In the case of Schwarzschild spacetime, the relevant metric coefficients are explicit functions of the radial coordinate. In the present setting, they are replaced by the more general functions $\psi(r)$ and $m(r)$. This does not change the overall structure of the argument, but it requires the steps in the proof of $C^0$-inextendibility of Schwarzschild spacetime to be reproved with these general coefficients. In this sense, the resulting differences are mainly computational. This issue appears in the proofs of \Cref{lemma1}, \Cref{prop:foc}, and \Cref{obstruction},  \Cref{prerequisite to clKcompact}. 

The second origin is the generalisation of the codimension-two fibre from the round sphere to a Riemannian manifold that is compact, connected, and homogeneous. This change is more substantial. This generalisation manifests in the proof of  \Cref{lemma1},  \Cref{prop:foc}, \Cref{obstruction}, \Cref{prerequisite to clKcompact}, \Cref{lemmakclcompact}, \Cref{lemmaofisometry}, \Cref{Wtls2}, \Cref{conservedquantityforh}, \Cref{eitherorprop1}, and \Cref{eitherorprop2}. 

We now emphasize on this second origin, recall key results, valid under assumptions weaker than spherical symmetry, and which are used to establish \Cref{statementA} assuming $C^0$-extension through the central curvature singularity. For the class of spacetimes as described by \eqref{sssst} with no spherical symmetry or any assumptions on the Riemannian fibre $(\mathcal{F},h)$ a priori, if the fibre has positive injectivity radius, then the spacetime is future one-connected sufficiently close to the curvature singularity. This is precisely what is shown in \Cref{prop:foc}. Next, \Cref{prerequisite to clKcompact} provides the local causal control near the singular end $r=0$: sufficiently far along any future-inextendible timelike curve, small perturbations in the $t$-coordinate and in the fibre direction still remain in the chronological future of an earlier point of the unperturbed curve. This estimate replaces the more explicit coordinate control available in the case of  Schwarzschild spacetime, which exploits the spherical symmetry. The \Cref{prerequisite to clKcompact} is then used in  \Cref{lemmakclcompact}, which, by using \Cref{lemma1}, globalises this local control by showing that the set $K(\tau_1,\tau_2)$ as defined in \eqref{Kdefn}, has compact closure in $\mathcal M_{\mathrm{int}}$. In particular, this set is bounded away from the singular boundary $r=0$. The compactness conclusion is crucial in the subsequent separation argument in Sbierski’s method: the compactness of the closure of $K(\tau_1,\tau_2)$ allows one to choose a controlled neighbourhood of it which acts as the required barrier in the later argument. The \Cref{lemmaofisometry} shows that the compactness, connectedness, and homogeneity of the Riemannian fibre is enough to move nearby fibre points into one another by isometries which are uniformly close to the identity element of the isometry group of $(\mathcal F,h)$. This weaker symmetry (weaker than spherical symmetry) is precisely what is needed in \Cref{Wtls2} (along with the compactness of $\bar K$ as mentioned above), where one perturbs timelike curves by a small fibre isometry to another timelike curve, to prove that the neighbourhood $W$ (as defined in \Cref{Wtls}) timelike-separates two subsets of $\mathcal M_{\textrm{int}}$. 

The timelike geodesic dichotomy property in Schwarzschild spacetime is well known; see Proposition 36 of Chapter 13 in \cite{oneill}. In Theorem A2 of \cite{Sbierski_2021}, Sbierski later proved that, in a general spherically symmetric black hole spacetime satisfying suitable assumptions, every affinely parametrised, future-directed, future-inextendible timelike geodesic whose initial point lies in the black hole exterior is either future complete or crosses the horizon. His argument exploits spherical symmetry, in particular the corresponding conserved quantities, to reduce the analysis to the plane $\theta=\pi/2$. In \Cref{eitherorprop1}, we generalise this result to a subclass of warped-product black hole spacetimes with a general compact and complete Riemannian fibre. In this setting, spherical symmetry need not be available, and hence the conserved quantities associated with rotations of the sphere may no longer exist. Instead, we use the conserved quantity determined by the Riemannian metric $h$ on the fibre $\mathcal F$, established in \Cref{conservedquantityforh}. In \Cref{eitherorprop2}, we further show that, for a future-directed, future-inextendible timelike geodesic contained in the black hole interior, future incompleteness is equivalent to the radial coordinate tending to zero. Combining \Cref{eitherorprop1} and \Cref{eitherorprop2}, we obtain the timelike geodesic dichotomy in \Cref{eitheror}. Moreover, in \Cref{app:multiplebhh}, we show that this dichotomy continues to hold even in the presence of multiple regular black hole horizons, and hence \Cref{maintheorem2} holds in such cases.

\section{Examples}
\label{sec:examples-topology}

For the class of warped-product black hole spacetimes described in \eqref{sssst}, the fibre $\mathcal F$ determines the topology of the horizon cross-section. Indeed, the horizon is located at the regular positive root $r=r_H$ of $r-2m(r)$, and each cross-section of the corresponding Killing horizon is diffeomorphic to $\mathcal F$. The topological censorship theorems (see, for e.g., \cite{Galloway_1995, Galloway_1999, Birmingham_1999, Cai_2001, Galloway_2006, Galloway_2011}) put constraints on the allowed topology of black hole horizons. In four spacetime dimensions, Hawking's topology censorship theorem \cite{Hawking_1972}, recalled as Theorem~1.2.1 in \cite{Galloway_2011}, states that cross-sections of the black hole horizon in asymptotically flat stationary black-hole spacetimes satisfying the dominant energy condition are topologically $S^2$. Later, Galloway and Schoen obtained a natural generalisation of Hawking's topology censorship theorem to general $d+1$ dimensions, with $d\geq 3$ (see Theorem 2.1 in \cite{Galloway_2006}, and also Theorem~1.4.1 in \cite{Galloway_2011}). In the stationary setting, they showed that for spacetimes satisfying the dominant energy condition, the cross-sections of the black hole horizons admit a metric of positive scalar curvature, apart from a Ricci flat exceptional case. This exceptional case was subsequently ruled out by Galloway later (see Theorem 1.1 and Corollary 1.3 in \cite{Galloway_2008}, and also Theorem~1.6.1 in \cite{Galloway_2011}). Consequently, if the dominant energy condition is satisfied by the spacetime, then the horizon cannot be toroidal. In the following subsections, we present two examples of future $C^0$-inextendible black hole spacetimes within our class that are compatible with Galloway's topological censorship theorem, as stated in Corollary 1.3 of \cite{Galloway_2011}. The first one is shown to also satisfy the dominant energy condition. In both cases, the scalar curvature of the horizon is strictly positive. Hence, they are not ruled out by Corollary 1.3 of \cite{Galloway_2011}. Finally, we discuss black hole spacetimes in anti de-Sitter space as described by Birmingham in \cite{Birmingham_1999}.

\subsection{A spherically symmetric non-vacuum black hole spacetime}

Here we discuss an example of a future $C^0$-inextendible spherically symmetric black hole spacetime, which satisfies the dominant energy condition, and which is permitted by the topological censorship theorem in \cite{Galloway_2008}. Let $(\mathcal M, g)$ be of the form \eqref{sssst}, and such that
$
\mathcal F=S^2,$ and $h=g_{S^2}.$
Let
$$
m(r):=M+\mu\frac{r}{1+r},
\qquad
\psi(r):=0,
$$
where
$
M>0$, and $0<\mu<\frac{1}{2}.$
First note that
$
m'(r)=\frac{\mu}{(1+r)^2}$, and
$m''(r)=-\frac{2\mu}{(1+r)^3}.
$
Moreover,
$
\lim_{r\to0}m(r)=M>0.
$
Define the horizon equation as
$
\mathcal  H (r):=r-2m(r).
$
Then
$
\mathcal H(0)=-2M<0$, $\lim_{r\to\infty}\mathcal H(r)=+\infty,
$
and
$$
\mathcal H'(r)=1-2m'(r)
=
1-\frac{2\mu}{(1+r)^2}
>
1-2\mu
>0.
$$
Hence $\mathcal H$ has a unique positive root, denoted $r_H$, and this root is regular.  Such $m(r)$ and $\psi(r)$ satisfy all four conditions in \eqref{conditionfinaltheorem}. Additionally, the fibre $(\mathcal F,h)=(S^2,g_{S^2})$ is compact, connected, and homogeneous. Hence, from \Cref{maintheorem2}, the spacetime is future $C^0$-inextendible.  Moreover, we have 
$$
G^t_t=-\frac{2m'(r)}{r^2},
\qquad
G^r_r=-\frac{2m'(r)}{r^2},
\qquad
G^{\theta}_{\theta}=-\frac{m''(r)}{r},
$$
where $G
=
\operatorname{Ric}(g)
-
\frac12\operatorname{Scal}_g\,g.$ is the Einstein tensor. Since $G= T$ is the Einstein's equation, and $T^{\mu}_{\nu}=\textrm{diag}(-\rho,p_r,p_{\theta}, p_{\phi})$, where $\mu,\nu\in\{t,r,\theta,\phi\}$, we have 
$$
\rho=\frac{2m'(r)}{r^2}=\frac{2\mu}{r^2(1+r)^2},
\qquad
p_r=-\frac{2m'(r)}{r^2}=-\frac{2\mu}{ r^2(1+r)^2},
\qquad \textrm{and}\qquad
p_{\theta}=-\frac{m''(r)}{r}=\frac{2\mu}{ r(1+r)^3}.
$$
Clearly, $\rho\geq\vert p_r\vert$, and $\rho\geq\vert p_{\theta}\vert$.
Hence, $(\mathcal M,g)$ satisfies the dominant energy condition. Since $(\mathcal F,h)$ has strictly positive scalar curvature, such black hole spacetimes are permitted by the Corollary 1.3 of \cite{Galloway_2008}. In summary, this gives a spherically symmetric, non-vacuum example of a black hole that is future $C^0$-inextendible and the topological censorship theorem does not rule them out. 

\subsection{A non-spherically symmetric black hole with fibre $S^p\times S^q$}

Let $p,q\geq 2$ and set
$
d-1:=p+q=\dim \mathcal F .
$
Consider the fibre $(\mathcal F, h)$ such that
$$
\mathcal F:=S^p\times S^q, \quad \textrm{and} \quad
h:=\left(\frac{p-1}{d-2}\right)\,g_{S^p}
\oplus
\left(\frac{q-1}{d-2}\right)g_{S^q},
$$
where $g_{S^p}$ and $g_{S^q}$ denote the unit round metrics. Let
$
m(r)\equiv M>0$, and $\psi(r)\equiv 0.
$
The horizon is located at
$
r_H=2M.
$
For such a choice of $m(r)$ and $\psi(r)$, all the conditions in \eqref{conditionfinaltheorem} are satisfied. Additionally,  $(\mathcal F,h)$ is compact, connected, and homogeneous. Hence, from \Cref{maintheorem2}, the spacetime is future $C^0$-inextendible.

Moreover, $\left(\mathcal F,h\right)$ is an Einstein manifold. Indeed, for the unit round metrics,
\begin{equation}\label{ex2einsteinmanifold}
    \operatorname{Ric}_{g_{S^p}}=(p-1)g_{S^p},
\qquad
\operatorname{Ric}_{g_{S^q}}=(q-1)g_{S^q}.
\end{equation}
Under a constant rescaling $c \times g$, where $c\in \mathbb{R}/\{0\}$, the Ricci tensor as a $(0,2)$-tensor is unchanged. Hence
$$
\operatorname{Ric}_{\frac{p-1}{d-2}g_{S^p}}
=(p-1)g_{S^p}
=(d-2)\left(\frac{p-1}{d-2}g_{S^p}\right), \quad \textrm{and}\quad
\operatorname{Ric}_{\frac{q-1}{d-2}g_{S^q}}
=(d-2)\left(\frac{q-1}{d-2}g_{S^q}\right).
$$
Therefore
$
\operatorname{Ric}_h=(d-2)h,$ and hence,
$\operatorname{Scal}_h=(d-1)(d-2)>0.
$
Thus, the fibre has strictly positive scalar curvature. If $(\mathcal M,g)$ doesn't satisfy the dominant energy condition, then Corollary 1.3 of \cite{Galloway_2008} doesn't apply to it. If it satisfies the dominant energy condition, then since $(\mathcal F,h)$ has strictly positive scalar curvature, Corollary 1.3 of \cite{Galloway_2008} permits such black hole spacetimes. In summary, this example gives a non-spherical, non-vacuum example of a black hole spacetime that is future $C^0$-inextendible, and the topological censorship theorem does not rule them out.

On similar lines, one can consider future $C^0$-inextendible black hole spacetimes as in \eqref{sssst} such that $m(r)=M>0$, and $\psi(r)\equiv0$, along with a closed, connected, homogeneous, and orientable fibre $(\mathcal{F},h)$ that admits positive scalar curvature.   

\subsection{Birmingham black holes in anti-de Sitter space}
The spacetimes described by Birmingham in \cite{Birmingham_1999} are solutions of Einstein's equations with a negative cosmological constant. In the class of spacetimes described by \eqref{sssst} in \Cref{sec2.2}, the $(d+1)$-dimensional Birmingham spacetimes correspond to $\psi(r)\equiv 0$, and 
$$
m(r)=\frac{\alpha M}{2r^{d-3}}+\frac{\Lambda r^3}{d(d-1)}, \quad \textrm{where}\quad \alpha\in \mathbb{R}^+, \quad \Lambda\in \mathbb{R^-}. 
$$
Moreover, $(\mathcal F,h)$ is closed and orientable.  If the fibre is an Einstein space of the form
$$
\textrm{Ric}_h=(d-2)h,
$$
then, the spacetime is an Einstein space with a negative cosmological constant, namely
$$
\textrm{Ric}_g=\frac{2\Lambda }{(d-1)}g.
$$
Such $m(r)$ and $\psi(r)$ satisfy conditions in \eqref{conditionthm2}, but fail to satisfy the second inequality in \eqref{eitherorconstraints}. If $(\mathcal F,h)$ is closed, connected, orientable, and homogeneous, then \Cref{maintheorem} is applicable to the corresponding Birmingham spacetime; however, \Cref{eitheror}, and hence, \Cref{maintheorem2} does not apply. Hence, such a class of black hole spacetimes are future $C^0$-inextendible through $r\to0$.   
 
\section*{Appendix}
\addcontentsline{toc}{section}{Appendix} 

\renewcommand{\thesection}{A\arabic{section}}
\setcounter{section}{0}

Here, we first discuss how the proof of \Cref{maintheorem2} still holds if there exists more than one positive real regular root of the function $r-2m(r)$, in \Cref{app:multiplebhh}. Then we derive the formula for the Kretschmann scalar of a general warped product spacetime (with codimension-two fibre) in \Cref{app:Kret}. Finally, in \Cref{app:kret2} we show that for the subclass of black hole spacetimes considered here (see  \Cref{sec2.2}), $m_0>0$ is a sufficient condition for the Kretschmann scalar to diverge as $r\to 0$.

\section{~Multiple regular black hole horizons}\label{app:multiplebhh}
\setcounter{equation}{0}
\renewcommand{\theequation}{A.\arabic{equation}}

In the main body of the paper, we assumed that the function $r-2m(r)$ has a unique positive regular root. This ensures that there exists only one regular black hole horizon in the spacetime $(\mathcal{M},g)$ given by \eqref{sssst}. We now explain how  \Cref{eitheror} (and hence the final \Cref{maintheorem2}) still hold when there is more than one positive regular root, or in other words, more than one regular black hole horizon. In such a case,  \Cref{eitherorprop1} applies to the outermost horizon, while  \Cref{eitherorprop2} applies only after the innermost horizon has been crossed by the geodesic. In order to obtain \Cref{eitheror} in such a multiple-horizon case, one must rule out the possibility that a future-directed timelike geodesic becomes future-inextendible at a finite affine parameter in one of the intermediate regions between two consecutive regular horizons. 

It should be noted that a future-inextendible timelike geodesic in the intermediate region that is also future complete does not cause any hindrance for the desired dichotomy in \Cref{eitheror}, since such a geodesic falls into the complete alternative. Thus, in the intermediate region between two consecutive regular horizons, it is enough to rule out the possibility of a future-directed timelike geodesic which remains in that region and is future-inextendible and future incomplete. The following lemma shows that this possibility cannot occur, and hence, \Cref{eitheror} continues to hold.

\begin{lemma}{}[No future-incomplete timelike geodesics in the intermediate region between two consecutive horizons]
Consider the spacetime $(\mathcal{M},g)$ described by \eqref{sssst} in  \Cref{sec2.2} but allowing multiple (finitely many) horizons. Let $r_1<r_2$, be two consecutive positive regular roots of the function $r-2m(r)$. Moreover, assume that $m(r)$ and $\psi(r)$ are smooth in the neighbourhood of each of $r_i$. Let $\mathcal A:=\mathbb R\times (r_1,r_2)\times \mathcal F$ denote the region between the corresponding two horizons. Assume that the fibre $(\mathcal F,h)$ is compact, and that the spacetime smoothly extends across the two horizons $\mathcal{H}_1=\{r=r_1\}$ and $\mathcal{H}_2=\{r=r_2\}$. Let $\rho:[0,b)\to \mathcal M$ be an affinely parametrised future-directed timelike geodesic whose image is contained in $\mathcal A$. If $\rho$ is future-inextendible, then $\rho$ is future complete, i.e. $b=\infty$. 
\end{lemma}

\begin{proof}
We prove the contrapositive of the above statement. Let $b<\infty$. Suppose $\rho$ stays at a positive $r$-distance away from both horizons. Then there exist a $\delta>0$ and $s_1\in[0,b)$ such that, for all $s\in[s_1,b)$,
$
r_1+\delta
<
\rho_r(s)
<
r_2-\delta .
$
Set
$
I_\delta:=[r_1+\delta,r_2-\delta].
$
Since $r_1$ and $r_2$ are consecutive roots, the function $r-2m(r)$ has no zero on $I_\delta$. Hence, both $1-2m(r)/r$ and its inverse are bounded on $I_\delta$. Moreover, since $r$ is bounded away from zero on $I_\delta$ and also bounded from above, the functions $m$ and $\psi$ are smooth and bounded on $I_{\delta}$. We now show that $\rho$ is contained in a compact set. From the first and second of the equations \eqref{conservedquantityinterior}, we obtain that $\vert\dot \rho_t\vert~\mathrm{and}~h(Y,Y)$ are both bounded in the domain $I_{\delta}$.  Since $\mathcal{F}$ is compact, the image of $\rho_y([0,b))$ is contained in a compact subset of $\mathcal{F}$. Additionally, from \eqref{kappa}, we have that $\vert\dot \rho_r\vert$ is bounded on $I_{\delta}$ too.  Since by assumption $b<\infty$, the bound on $\dot\rho_t$ gives that $\vert\rho_t\vert$ is bounded in the domain $I_{\delta}$. We now use the following result that we also used in \Cref{eitherorprop2}: For smooth manifold $\mathcal{M}$, the image of the future-directed timelike geodesic $(\rho,\dot \rho): [0,b)\to T\mathcal{M}$, where $b<\infty$, lies in a compact set of $T\mathcal M$ if and only if it is extendible (refer to Lemma 56, Chap 1 of O'Neill \cite{oneill}). Hence $\rho$ is extendible.

It remains to treat the case in which $\rho$ does not stay a positive
$r$-distance away from the horizons. We show that in this case $\rho$ is
again future extendible, which gives the desired contradiction. Set $f(r):=1-\frac{2m(r)}{r}.$ Using \Cref{conservedquantityinterior} and \Cref{kappa}, we have
\begin{equation}\label{mh1}
    \dot\rho_r^2
        =
        E^2 \exp\left({2\psi(\rho_r)}\right)
        -
        f(\rho_r)
        \left(
        \kappa+\frac{K}{\rho_r^2}
        \right).
\end{equation}
Since $\rho_r\in (r_1,r_2)$ and $r_1>0$, the right-hand side of \Cref{mh1} is bounded. Hence $\dot\rho_r$ is bounded. By the definition of $\rho$, $\rho_r\in(r_1,r_2)$ is bounded too. We now use coordinates which are regular at the horizon
$\mathcal H_i=\{r=r_i\}$, where $i\in\{1,2\}$. Since $r_i$ is a regular root of $f$, we have
$f(r_i)=0$ and $f'(r_i)\neq 0$. Consider \footnote{Note that $\kappa_i$ defined in \Cref{mbhki} is different from $\kappa$ that appears in the third of \Cref{conservedquantityinterior}.}
\begin{equation}\label{mbhki}
    \kappa_i:=\frac12 e^{-\psi(r_i)} f'(r_i)\neq 0,
\end{equation}
and introduce coordinates
$$
u=t-r^*,
        \qquad
v=t+r^*, \quad\textrm{where} \quad r^*:=\int\frac{\exp\left({\psi(r)}\right)}{f(r)}dr.
$$
Since $f$ has a simple zero at $r=r_i$, for $r$ near $r_i$, we may write 
$$
f(r)=(r-r_i)F_i(r), 
$$
where 
$F_i$ is smooth near $r_i$ and $F_i(r_i)=f'(r_i)\neq 0$. Hence 
$$
\frac{\exp\left({\psi(r)}\right)}{f(r)} = \frac{\exp\left({\psi(r)}\right)}{(r-r_i)F_i(r)}= \frac{1}{r-r_i}\left(\frac{\exp\left({\psi(r_i)}\right)}{f'(r_i)} + \left(r-r_i\right)G_i(r)\right)= \frac{\exp\left({\psi(r_i)}\right)}{f'(r_i)}\frac{1}{r-r_i} + G_i(r). 
$$
for some smooth function $G_i$. Therefore for $r$ sufficiently close to $r_i$, we have
\begin{equation}\label{mbhrstar}
    r_* = \frac{\exp\left({\psi(r_i)}\right)}{f'(r_i)} \log |r-r_i| + R_i(r), 
\end{equation} 
where $R_i$ is smooth near $r_i$. Now define Kruskal-type coordinates by
\begin{equation}\label{UandVmbh}
    U=-\exp\left({-\kappa_i u}\right), \qquad V=\exp\left({\kappa_i v}\right).
\end{equation}
After possibly changing the signs of $U$ and $V$, these coordinates are
smooth across $H_i$. Again, sufficiently close to $r=r_i$, we have 
\begin{equation}\label{mbhuv}
    UV=\chi_i(r)(r-r_i),\quad\textrm{where} \quad \chi_i=-\exp\left(2\kappa_iR_i(r)\right).
\end{equation}
Here, we used \Cref{mbhrstar}. The function $\chi_i$ is smooth and non-vanishing near $r_i$. In these coordinates the metric extends smoothly and non-degenerately across $H_i$. It remains to show that $(\rho,\dot\rho)$ stays in a compact subset of the tangent bundle in these regular coordinates, and use the same argument as in the previous case using Lemma 56, Chap. 1 of O'Neill \cite{oneill}. From the second equation in \Cref{conservedquantityinterior}, and the fact that $r_i>0$, we have $h(Y,Y)$ is bounded. Since $\mathcal F$ is compact, $\rho_y$ is also bounded. Now, we prove that $U,V, \dot U$, and $\dot V$ are bounded in the domain of $\rho$. Write $\bar{A}(\rho_r):=\exp{\left(\psi(\rho_r)\right)}$. From the
definitions of $U$ and $V$ we obtain
\begin{equation}\label{udotbyuvdotbyv}
     \frac{\dot U}{U}
        =
        -\kappa_i
        \frac{E  \bar{A}(\rho_r)^2- \bar{A}(\rho_r)\dot\rho_r}{f(\rho_r)},\quad \textrm{and} \quad \frac{\dot V}{V}
        =
        \kappa_i
        \frac{E  \bar{A}(\rho_r)^2+ \bar{A}(\rho_r)\dot\rho_r}{f(\rho_r)}.
\end{equation}
First, suppose that $E\neq 0$. From \Cref{mh1},
$$
        \dot\rho_r^2
        =
        E^2  \bar{A}(\rho_r)^2
        +
        O(f(\rho_r)).
$$
Thus, after restricting to a sufficiently small final interval
$[s_0,b)$, the sign of $\dot\rho_r$ is constant and
$$
        \dot\rho_r
        =
        \sigma |E|  \bar{A}(\rho_r)
        +
        O(f(\rho_r)),
        \qquad
        \sigma\in\{+1,-1\}.
$$
If $\sigma=\operatorname{sgn}(E)$, then
$
        E  \bar{A}(\rho_r)^2- \bar{A}(\rho_r)\dot\rho_r
        =
        O(f(\rho_r)).
$
Hence $\dot U/U$ is bounded, i.e., there exists constant $C>0$ such that  
$$ 
\left|\frac{\dot U(s)}{U(s)}\right|\leq C \qquad \text{for all } s\in[s_0,b). 
$$ 
Then 
$$
\frac{d}{ds}\log |U(s)|=\frac{\dot U(s)}{U(s)}. 
$$ 
Hence, for any $s\in(s_0,b)$, 
$$
\log |U(s)|-\log |U(s_0)| = \int_{s_0}^{s}\frac{\dot U(\sigma)}{U(\sigma)}\,d\sigma . 
$$ 
Taking absolute values and using $b<\infty$, we obtain 
$$\left|\log |U(s)|-\log |U(s_0)|\right| \leq C(s-s_0) \leq C(b-s_0). 
$$ 
Therefore 
$$
|U(s_0)|\exp\left({-C(b-s_0)}\right) \leq |U(s)| \leq |U(s_0)|\exp\left({C(b-s_0)}\right). 
$$ 
Thus $U$ is bounded above and bounded away from zero on $[s_0,b)$. In particular, 
$$ 
|\dot U(s)| = |U(s)|\left|\frac{\dot U(s)}{U(s)}\right| \leq C |U(s_0)|\exp\left({C(b-s_0)}\right) , 
$$ 
so $\dot U$ is also bounded on $(s_0,b)$. From \Cref{mbhuv}, we have $UV=\chi_i(\rho_r)(\rho_r-r_i),$ and
it follows that $V=O(\rho_r-r_i)$. The formula for $\dot V/V$ then implies
that $\dot V$ is bounded. Thus, both $\dot U$ and $\dot V$ are bounded.

If instead $\sigma=-\operatorname{sgn}(E)$, the same argument with $U$ and
$V$ interchanged shows that $V$ stays bounded and bounded away from zero,
while $U=O(\rho_r-r_i)$. Again, $\dot U$ and $\dot V$ are bounded.

It remains to consider the case $E=0$. From the first of \Cref{conservedquantityinterior}, we have $\dot\rho_t=0$, so $t$ is
constant. \Cref{mh1} becomes
$$
 \dot\rho_r^2
        =
        -f(\rho_r)
        \left(
        \kappa+\frac{K}{\rho_r^2}
        \right).
$$
Since $f$ has a simple zero at $r_i$, this implies 
\begin{equation}\label{mbhhdotrho}
    |\dot\rho_r|
        = O \left(|\rho_r-r_i|^{1/2}\right)
\end{equation}
near $r_i$. On the other hand, because $t$ is constant, we have from \Cref{UandVmbh},
$$
        U=O(|\rho_r-r_i|^{1/2}),
        \qquad
        V=O(|\rho_r-r_i|^{1/2}).
$$
Using \Cref{udotbyuvdotbyv} shows that $\dot U$ and $\dot V$
are bounded.

Therefore, in the smooth horizon-regular coordinates $(U,V,y)$, both
$\rho$ and $\dot\rho$ remain in a compact subset of the tangent bundle, implying that $\rho$ is future extendible (using Lemma 56, Chap. 1 of O'Neill \cite{oneill}). 
\end{proof}

\section{~Kretschmann scalar for general warped product spaces}\label{app:Kret}
Consider a warped product space given by
$$
(\mathcal{M}^{d+1},g)=(\mathcal{B},\gamma)\times_r(\mathcal{F},h),
\qquad \textrm{dim}(\mathcal{F})=d-1,
$$
with warped product metric
$$
g_{AB}\ dx^A \otimes  dx^B
=
\gamma_{ab}(x)\ dx^a \otimes  dx^b
+
r^2 h_{ij}(y)\ dy^i \otimes  dy^j .
$$
We denote base indices by $a,b,c,d$, and fibre indices by
$i,j,k,l$. Thus
\begin{equation}\label{indices1}
  g_{ab}=\gamma_{ab},
\qquad
g_{ij}=r^2h_{ij},
\qquad
g_{ai}=0,\quad \textrm{and}
\end{equation}
\begin{equation}\label{indices2}
g^{ab}=\gamma^{ab},
\qquad
g^{ij}=r^{-2}h^{ij},
\qquad
g^{ai}=0.
\end{equation}

We write
$
|\hat\nabla r|_\gamma^2
=
\gamma^{ab}\hat\nabla_a r \hat\nabla_b r,
$
where \(\hat\nabla\) is the Levi-Civita connection of \(\gamma\). For $X,Y,Z\in\Gamma\left(T\mathcal{B}\right)$, and $U,V,W\in\Gamma\left(T\mathcal{F}\right)$, we have the following identities of the Riemann curvature tensors $R:\Gamma\left(T\mathcal{M}\right)\times \Gamma\left(T\mathcal{M}\right)\times \Gamma\left(T\mathcal{M}\right)\to \Gamma\left(T\mathcal{M}\right)$, $R^{\mathcal{
B}}:\Gamma\left(T\mathcal{B}\right)\times \Gamma\left(T\mathcal{B}\right)\times \Gamma\left(T\mathcal{B}\right)\to \Gamma\left(T\mathcal{B}\right)$, and $R^{\mathcal{F}}:\Gamma\left(T\mathcal{F}\right)\times \Gamma\left(T\mathcal{F}\right)\times \Gamma\left(T\mathcal{F}\right)\to \Gamma\left(T\mathcal{F}\right)$ for a warped product space (see \S 7, Proposition 42 in O'Neill \cite{oneill}).
\begin{align}
R(X,Y)Z &= R^{\mathcal{B}}(X,Y)Z, \label{eq:oneill-base}\\
R(X,Y)U &= 0, \label{eq:oneill-vanish-1}\\
R(U,V)X &= 0, \label{eq:oneill-vanish-2}\\
g \left( R(X,U)Y,V\right)
&=
-\frac{\operatorname{Hess}\left(r\right)(X,Y)}{r}
g\left( U,V\right), \label{eq:oneill-mixed}\\
R(U,V)W
&=
R^{\mathcal{F}}(U,V)W
-
\frac{|\hat\nabla r|_\gamma^2}{r^2}
\left(
g\left( U,W\right) V
-
g\left( V,W\right) U
\right).
\label{eq:oneill-fibre}
\end{align}
Here $\operatorname{Hess}{\left(r\right)}(X,Y)=\gamma\left(\hat\nabla_X~\textrm{grad}_{\gamma}r,Y\right)$ is the Hessian of $r$ with respect to $\gamma$. 
In the index notation, \Cref{eq:oneill-base} gives
\begin{equation}\label{baserabcd}
    R_{abcd}(g)=R_{abcd}(\gamma).
\end{equation}
The vanishing identities \Cref{eq:oneill-vanish-1} and
\Cref{eq:oneill-vanish-2} imply
\begin{equation}
    R_{abci}(g)=0,
\qquad
R_{aijk}(g)=0,
\end{equation}
together with the corresponding components obtained by the usual curvature
symmetries. Consider the identity \Cref{eq:oneill-mixed} that includes mixed terms from the base as well as the fibre. We have
$
  R_{aibj}(g)
=
g\left( R(\partial_a,\partial_i)\partial_b,\partial_j\right).  
$
Since
$g\left( \partial_i,\partial_j\right)
=
g_{ij}
=
r^2 h_{ij},$
we obtain
\begin{equation}\label{mixedraibj}
    R_{aibj}(g)
=
-r\,\left(\hat\nabla_a\hat\nabla_b r\right)\,h_{ij}.
\end{equation}
Finally, the identity \Cref{eq:oneill-fibre} gives
\begin{equation}\label{fibreijkl}
    R_{ijkl}(g) = r^2\left[R_{ijkl}(h) - \vert\hat\nabla r\vert_\gamma^2 \left( h_{ik}h_{jl}-h_{il}h_{jk} \right)\right].
\end{equation}
We now compute the Kretschmann scalar
$
K_g=R_{ABCD}(g)R^{ABCD}(g).
$
Using the decomposition into base, mixed, and fibre curvature components,
one has
\begin{equation}\label{Kretschmannf1}
    K_g
=
R_{abcd}R^{abcd}
+
4R_{aibj}R^{aibj}
+
R_{ijkl}R^{ijkl}.
\end{equation}
The factor of $4$ accounts for the equivalent placements of the two base and two fibre indices in the mixed curvature block. The purely base contribution is obtained using \eqref{baserabcd}
\begin{equation}\label{Kretschmannfb}
    R_{abcd}(g)R^{abcd}(g)
=
R_{abcd}(\gamma)R^{abcd}(\gamma)
=
K_\gamma=R_{\gamma}^2\geq 0,
\end{equation}
where $K_{\gamma}$ is the Kretschmann scalar of the manifold $\left(\mathcal{B},\gamma\right)$. Since $\mathcal{B}$ is two-dimensional, $K_{\gamma}$ is equal to the square of its Ricci scalar $R_{\gamma}$. For the mixed contribution, using \Cref{mixedraibj} and using \Cref{indices1} and \Cref{indices2} to raise the indices, we find
\begin{equation}\label{Kretschmannfm}
\begin{split}
4 R_{aibj}R^{aibj}
&=
4 r^2
\left(\hat\nabla_a\hat\nabla_b r\right)
\left(\hat\nabla^a\hat\nabla^b r\right)
h_{ij}r^{-4}h^{ij} \\
&=
\frac{4\left(d-1\right)}{r^2}
\left(\hat\nabla_a\hat\nabla_b r\right)
\left(\hat\nabla^a\hat\nabla^b r\right),
\end{split}
\end{equation}
because $h_{ij}h^{ij}=d-1.$
It remains to compute the fibre contribution. Define
$$
q:=|\hat\nabla r|_\gamma^2,
\qquad
A_{ijkl}:=h_{ik}h_{jl}-h_{il}h_{jk}.
$$
Then \Cref{fibreijkl} can be rewritten as
$$
R_{ijkl}(g)
=
r^2\left[
R_{ijkl}(h)-qA_{ijkl}
\right].
$$
Using \Cref{indices1} and \Cref{indices2} to raise the indices finally gives
\begin{equation}\label{Kretschmannff}
\begin{split}
    R_{ijkl}(g)R^{ijkl}(g)
&=
\frac{1}{r^4}
\left[
R_{ijkl}(h)-qA_{ijkl}
\right]
\left[
R^{ijkl}(h)-qA^{ijkl}
\right] \\
& =\frac{1}{r^4}
\left[
R_{ijkl}(h)R^{ijkl}(h)
-
2qR_{ijkl}(h)A^{ijkl}
+
q^2A_{ijkl}A^{ijkl}
\right]\\
&=\frac{1}{r^4}
\left[
K_h
-
4qR_h
+
2(d-1)(d-2)q^2
\right]\\
&=
\frac{1}{r^4}\left[K_h-4|\hat\nabla r|_\gamma^2R_h+2(d-1)(d-2)|\hat\nabla r|_\gamma^4\right].
\end{split}
\end{equation}
Here, in the third step, we have used
$
R_{ijkl}(h)R^{ijkl}(h)=K_h,
$
and
$
R_{ijkl}(h)A^{ijkl}=2R_h,
$
where
$
R_h=h^{ik}h^{jl}R_{ijkl}(h)
$
is the Ricci scalar of the fibre. Moreover
$
A_{ijkl}A^{ijkl}=2(d-1)(d-2).
$
In the final step, we restore \(q=|\hat\nabla r|_\gamma^2\). Combining the base, mixed, and fibre components from \Cref{Kretschmannfb}, \Cref{Kretschmannfm}, and \Cref{Kretschmannff} the Kretschmann scalar of
$(\mathcal{M}^{d+1},g)$ is obtained from \Cref{Kretschmannf1} as
\begin{equation}\label{Kretschmannfinalformula}
    K_g
=
\underbrace{K_\gamma\rule[-1.2em]{0pt}{0pt}}_{\textrm{base component}}
+
\underbrace{\frac{4(d-1)}{r^2}
\left(\hat\nabla_a\hat\nabla_b r\right)
\left(\hat\nabla^a\hat\nabla^b r\right)\rule[-1.2em]{0pt}{0pt}}_{\textrm{mixed component}}
+
\underbrace{\frac{1}{r^4}
\left[
K_h
-
4|\hat\nabla r|_\gamma^2R_h
+
2(d-1)(d-2)|\hat\nabla r|_\gamma^4
\right]\rule[-1.2em]{0pt}{0pt}}_{\textrm{fibre component}}.
\end{equation}

\section{~Sufficient condition for the divergence of  Kretschmann scalar as $r\to 0$}\label{app:kret2}

\begin{proposition}\label{appendixproposition}
Let $\left(\mathcal{M},g\right)$ be as described  by \eqref{sssst} in \Cref{sec2.2}. Suppose that
$
\lim_{r\to 0}m(r)=m_0>0.
$
Moreover, let $(\mathcal{F},h)$ be compact. Then  
$$
\lim_{r\to0}\inf_{y\in\mathcal F} K_g(r,y)=\infty,
$$ 
where $K_g$ is the Kretschmann scalar.
\end{proposition}

\begin{proof}
We show that the fibre component in the formula for the Kretschmann scalar (see \Cref{Kretschmannfinalformula}) of $(\mathcal{M},g)$ diverges to infinity as $r\to 0$. Then we argue that the remaining two components are also nonnegative. This will then imply that the Kretschmann scalar $K_g$ diverges to infinity as $r\to 0$.  We have
$$
|\hat\nabla r|_\gamma^2
=
\gamma^{-1}\left(\hat\nabla r,\hat\nabla r\right) 
=
\gamma^{rr}=1-\frac{2m(r)}{r}.
$$
Substituting for $|\hat\nabla r|_\gamma^2$ in  the fibre component of \Cref{Kretschmannfinalformula}, and taking the limit as $r\to 0$, we obtain
\begin{equation}
   \lim_{r\to 0 }\frac{1}{r^4}
\left[
K_h
-
4\left(1-\frac{2m(r)}{r}\right)R_h
+
2(d-1)(d-2)\left(1-\frac{2m(r)}{r}\right)^2
\right]=\infty.
\end{equation}
Here, we assumed $\lim_{r\to 0}m(r)=m_0>0$, along with using the fact that $K_h$ and $R_h$ are bounded and does not depend on $r$ ($\mathcal{F}$ is compact).

From \Cref{Kretschmannfb} it is clear that the base component in \Cref{Kretschmannfinalformula} is non-negative. Regarding the mixed component, consider the Hessian
$$
\hat\nabla_a\hat\nabla_b r=\partial_a\partial_b r-\Gamma^c_{ab}\partial_c r=-\Gamma^c_{ab}\partial_c r=-\Gamma^c_{ab}\delta_c^r=-\Gamma^1_{ab}, \quad a,b\in\{0,1\}.
$$
One can see that $\Gamma^1_{01}=\Gamma^1_{10}=0$. Consequently, the Hessian has no
off-diagonal $tr$-components, and hence 
$$
\hat\nabla^0\hat\nabla^0r=\left(\gamma^{00}\right)^2 \hat\nabla_0\hat\nabla_0r,\quad \hat\nabla^1\hat\nabla^1r=\left(\gamma^{11}\right)^2 \hat\nabla_1\hat\nabla_1r, \quad \hat\nabla^0\hat\nabla^1r=\hat\nabla^1\hat\nabla^0r=0.
$$
This implies that $\left(\hat\nabla_a\hat\nabla_b r\right)
\left(\hat\nabla^a\hat\nabla^b r\right)\geq 0$. Thus, the second term in the RHS of \Cref{Kretschmannfinalformula} is nonnegative.
\end{proof}

\bibliography{main}

\end{document}